\documentclass[11pt]{article}
\usepackage{amsfonts}

\usepackage{graphics}
\usepackage{indentfirst}
\usepackage{cite}
\usepackage{latexsym}
\usepackage{amsmath,amsthm}
\usepackage{amssymb}
\usepackage[dvips]{epsfig}
\usepackage{amscd}
\usepackage{mathrsfs}

\usepackage{color}
\newtheorem{theorem}{Theorem}[section]
\newtheorem{remark}{Remark}[section]

\newtheorem{definition}{Definition}[section]
\newtheorem{lemma}[theorem]{Lemma}

\newcommand{\n}{\rho}

\newcommand{\mr}{\mathbb{R}}
\newcommand{\lm}{\lambda}

\renewcommand{\div}{ {\rm div }  }

\newcommand{\na}{\nabla }

\newcommand{\pa}{\partial}

\newcommand{\bt}{\begin{theorem}}
\newcommand{\bl}{\begin{lemma}}
\newcommand{\el}{\end{lemma}}
\newcommand{\et}{\end{theorem}}
\newcommand{\ga}{\gamma}

\newcommand{\OM}{\Omega}

\newcommand{\curl}{{\rm curl} }

\newcommand{\de}{\delta}

\newcommand{\la}{\label}
\newcommand{\si}{\sigma}

\newcommand{\om}{\Omega}
\newcommand{\ol}{\overline}

\newcommand{\bn}{\begin{eqnarray}}
\newcommand{\en}{\end{eqnarray}}
\newcommand{\bnn}{\begin{eqnarray*}}
\newcommand{\enn}{\end{eqnarray*}}

\newcommand{\bnnn}{\begin{eqnarray*}}
\newcommand{\ennn}{\end{eqnarray*}}
\newcommand{\ben}{\begin{enumerate}}
\newcommand{\een}{\end{enumerate}}

\newcommand{\T}{\mathbb{T}}

\newcommand{\ba}{\begin{aligned}}
\newcommand{\ea}{\end{aligned}}
\newcommand{\be}{\begin{equation}}
\newcommand{\ee}{\end{equation}}

\def\p{\partial}
\def\norm[#1]#2{\|#2\|_{#1}}

\def\lam{\lambda}
\def\ep{\varepsilon}

\def\o{\omega}
\def\r{\mathbb{R}}
\def\rr{\mathbb{R}^2}
\def\rrr{\mathbb{R}^3}

\makeatletter      
\@addtoreset{equation}{section}
\makeatother       

\title{Global Existence and Large-Time Behavior of Strong Solutions to the Two-Dimensional Compressible Nematic Liquid Crystal Flows with Large Initial Data and Vacuum}

\date{}

\author{$\text{Qinghao L{\small EI}}^{a,b}, \text{Lu W{\small{ANG}}}^{b}\thanks{Email addresses:  leiqinghao22@mails.ucas.ac.cn (Q. H. Lei), wlu1130@163.com (L. Wang) }$\\
a. School of Mathematical Sciences,\\ University of Chinese Academy of Sciences,
Beijing 100049, P. R. China;\\
b. Institute of Applied Mathematics,\\ Academy of Mathematics and Systems Science, \\
Chinese Academy of Sciences, Beijing 100190, P. R. China}

\begin{document}
\maketitle

\begin{abstract}
In this paper, we investigate two-dimensional compressible nematic liquid crystal flows with a planar director.
We assume that the shear viscosity is a positive constant and that the bulk viscosity is given by $\lambda(\rho)=\rho^\beta$ with $\beta>4/3$.
In the periodic domain or in a bounded simply connected smooth domain with Navier-slip and Neumann boundary conditions, we establish the global existence and uniqueness of strong solutions without any restrictions on the size of the initial data.
In particular, no geometric angle condition is imposed on the initial orientation field.
Moreover, we derive a time-uniform upper bound for the density and establish exponential decay estimates for the strong solutions. 
\end{abstract}

\section{Introduction and main results}
We study the two-dimensional compressible nematic liquid crystal flows which read as follows:
\be\ba\la{nlckv}
\begin{cases}
  \rho_t + \div(\rho u) = 0, \\
(\n u)_t + \div(\n u\otimes u) -\mu \Delta u 
-\na ( (\mu + \lm) \div u) + \na P = - \na d \cdot \Delta d, \\
d_t + u \cdot \na d = \Delta d + |\na d|^2 d, \\
|d|=1,
\end{cases}
\ea\ee
where $t \ge 0$ is time, $x \in \OM \subset \rr$ is the spatial coordinate,
$\n=\n(x,t)$, $u(x,t)=(u^1(x,t),u^2(x,t))$ and $d=(d^1(x,t),d^2(x,t))$
represent the density, velocity, and macroscopic orientation of the liquid crystal molecules, respectively.
The pressure $P$ is given by
\be\ba\la{i1}
P=R \n^\ga,
\ea\ee
with constants $R>0$ and $\ga>1$.
The shear viscosity coefficient $\mu$ and the bulk viscosity coefficient $\lam$ satisfy the following hypothesis:
\be\la{i2}
0<\mu = \text{constant},\quad \lam(\n)=b \n^\beta,
\ee
where $b$ and $\beta$ are positive constants.
Without loss of generality, we assume that $R=b=1$.

The system is supplemented with the initial data
\be\la{i30}
\n(x,0)=\n_0(x),\quad \n u(x,0)= \n_0 u_0(x), \quad d(x,0)=d_0(x),\quad x\in \OM,
\ee
together with one of the following two types of boundary conditions:

(1) Periodic boundary conditions.
In this case,
\be\la{zqbjtj}\ba
\OM=\T^2 \triangleq \mr^2/\mathbb{Z}^2, \ \text{and} \  (\n,u,d) \ \text{is 1-periodic in each spatial direction};
\ea\ee

(2) Navier-slip and Neumann boundary conditions.
In this case, $\OM$ is a simply connected bounded smooth domain in $\rr$, and
\be\la{yjybjtj}\ba
u \cdot n = 0,\quad \curl u = -A u \cdot n^\bot, \quad \frac{\p d}{\p n}=0 \ \text{ on } \ \p \OM,
\ea\ee
where $A$ is a non-negative smooth function on the boundary, $n=(n_1,n_2)$ denotes the unit outer normal vector of the boundary $\partial \Omega$, and $n^\bot$ is the unit tangential vector on $\partial \Omega$ denoted by
\be\ba\nonumber
n^\bot\triangleq (n_2,-n_1).
\ea\ee
In addition, we extend $n$, $n^\bot$, and $A$ smoothly to $\overline\Omega$.

The system (\ref{nlckv}) is a simplified hydrodynamic model for liquid crystal flows, originally introduced by Ericksen \cite{EJL} and Leslie \cite{LFM} in the 1960s.
Mathematically, it is a strongly coupled system comprising the compressible Navier-Stokes equations and the harmonic map heat flow.
In particular, when $d$ is a constant unit vector, the system (\ref{nlckv})--(\ref{i30}) reduces to the compressible Navier-Stokes system.
There is an extensive literature on the strong solvability of the multidimensional compressible Navier-Stokes equations with constant viscosity coefficients.
The first global classical solutions were obtained by Matsumura-Nishida \cite{MN1} for initial data close to a non-vacuum equilibrium in the $H^s$-norm.
Subsequently, Hoff \cite{H1,H3} investigated the problem for discontinuous initial data.
For arbitrarily large initial data, Lions \cite{L2} (see also Feireisl \cite{F,FNP}) proved the global existence of finite-energy weak solutions provided that the adiabatic exponent $\ga$ is suitably large.
More recently, Huang-Li-Xin \cite{HLX2} and Li-Xin \cite{LX2} established the global existence and uniqueness of classical solutions to the three-dimensional and two-dimensional Cauchy problems for initial data with small total energy but possibly large oscillations and vacuum.
Their results were later extended by Cai-Li \cite{CL} to general bounded domains subject to slip boundary conditions for the velocity field.
In contrast, there are few results on the global existence of strong solutions without restrictions on the size of the initial data.
For the case where $d$ is a constant unit vector, the system (\ref{nlckv})--(\ref{i30}) corresponds to the model introduced by Vaigant-Kazhikhov \cite{VK}, who established the existence and uniqueness of global strong solutions for large initial data with density away from vacuum in rectangular domains under the assumption that $\beta>3$.
Subsequently, Jiu-Wang-Xin \cite{JWX1} generalized this result to allow for initial vacuum in a periodic domain.
Recently, Huang-Li \cite{HL2,HL3} (see also \cite{JWX2}) relaxed the crucial condition from $\beta>3$ to $\beta>\frac{4}{3}$ by applying new techniques based on commutator theory and blow-up criteria in periodic domains or in the whole space, allowing the density to vanish.
More recently, Fan-Li-Li \cite{FLL} established the global existence of strong solutions for $\beta>\frac{4}{3}$ in general bounded simply connected domains where the velocity field satisfies the Navier-slip boundary conditions.
Later, Fan-Li-Wang \cite{FLW} derived a time-independent upper bound for the density and the exponential decay of global strong solutions in both periodic and bounded simply connected domains under the assumption that $\beta>\frac{4}{3}$.

Returning to the compressible nematic liquid crystal system (\ref{nlckv}), there is an extensive literature on the global existence of solutions with constant viscosity coefficients.
In particular, the global existence of weak solutions for large initial data was established in \cite{JJW1,JJW2,LLW}, provided that the initial orientation field satisfies a geometric angle condition.
The local existence and uniqueness of strong solutions were proved by Huang-Wang-Wen \cite{HWW}, allowing the initial density to vanish on open subsets.
Hu-Wu \cite{HW} established the global existence and uniqueness of strong solutions in critical Besov spaces for the three-dimensional Cauchy problem, under the assumption that the initial data are close to an equilibrium state $(1,0,e)$ for some constant vector $e \in \mathbb{S}^2$.
For initial data close to a non-vacuum equilibrium state in $H^s(\rrr)$ with $s \ge 3$, Gao-Tao-Yao \cite{GTY} established the global existence and long-time behavior of classical solutions.
Recently, by employing ideas from \cite{HLX2,LX2}, Li-Xu-Zhang \cite{LXZ} and Wang \cite{WT} obtained the global existence and uniqueness of classical solutions to (\ref{nlckv}) for the three-dimensional and two-dimensional Cauchy problems, respectively.
These results hold for initial data with possibly large oscillations and vacuum, provided the initial energy is sufficiently small.
Later, for general three-dimensional bounded domains with slip boundary conditions for the velocity field and Neumann boundary conditions for the orientation field, Liu-Zhong \cite{LZ} applied the methodology of Cai-Li \cite{CL} to establish the global existence and uniqueness of classical solutions under the condition that the initial energy is small enough.

More recently, Zhong-Zhou \cite{ZZ,ZZ2} established the global existence and uniqueness of strong solutions to the system (\ref{nlckv}) in a simply connected bounded smooth domain and in the whole space with vacuum far-field density, under the assumptions that $\beta>\frac{4}{3}$ and that the initial orientation field satisfies the geometric condition
\be\la{jhtj}\ba
d_0^2 \ge \ep_0 \ \text{ for some positive } \ep_0>0.
\ea\ee
The aim of this paper is to establish the global existence and large-time behavior of strong solutions for arbitrarily large initial data allowing vacuum, both in the periodic domain and in bounded simply connected domains, without imposing the geometric condition (\ref{jhtj}), provided that $\beta>\frac{4}{3}$.

Before stating the main results, we first explain the notations
and conventions used throughout this paper. We define
\be\ba\nonumber
\na^\bot \triangleq (\p_2,-\p_1),\quad
\int f dx \triangleq \int_{\OM} fdx,\quad
\ol{f} \triangleq \frac{1}{|\OM|}\int f dx,
\ea\ee
and let $\na d \odot \na d$ denote the matrix whose
$(i,k)$-th entry is $\p_i d \cdot \p_k d$.

For a positive integer $s$ and $1\leq r\leq \infty$, we denote the standard Lebesgue and Sobolev spaces as follows:
\be\ba\nonumber
\begin{cases}
L^r =L^r(\OM),\quad W^{s,r} =W^{s,r}(\OM),\quad H^s =W^{s,2}, \\
\tilde{H}^1 =\{v \in H^1(\Omega )\vert v \cdot n=0, \curl v = -A v \cdot n^\bot \,\,\,\text{on}\,\,\, \partial\Omega \}.
\end{cases}
\ea\ee
The material derivative is defined by
\be\ba\nonumber
\frac{D}{Dt}f=\dot{f} \triangleq f_t + u\cdot\na f.
\ea\ee
We define the effective viscous flux $G$ and the vorticity $\o$ as
\be\ba\la{gw}
G \triangleq (2\mu + \lam)\div u - (P-P(\ol{\n})), \quad \o \triangleq \na^\bot \cdot u = \pa_2 u^1 - \pa_1 u^2.
\ea\ee
We now give the definition of strong solutions to (\ref{nlckv}).
\begin{definition}
If all derivatives involved in \eqref{nlckv} for $(\n,u,d)$ are regular distributions,
and equations \eqref{nlckv} hold almost everywhere in 
$\OM \times(0,T)$, then $(\n,u,d)$ is called a strong solution.
\end{definition}

Our first result concerns the global existence and exponential decay of strong solutions in the periodic domain.
\begin{theorem}\la{thpp}
Let $\OM = \T^2$.
Assume that
\be\la{pbg}\ba
\beta>\frac{4}{3}, \quad \ga>1,
\ea\ee
and that the initial data $(\n_0,u_0,d_0)$ satisfy, for some $q>2$,
\be\la{pssol1}\ba
& \n_0 \in W^{1,q}(\T^2), \quad \n_0 \ge 0, \quad \ol{\n}_0 \triangleq \int_{\T^2} \n_0 dx > 0, \\
& u_0 \in H^1(\T^2), \quad \nabla d_{0} \in H^{1}(\T^2),\quad |d_{0}|=1.
\ea\ee
Then the problem \eqref{nlckv}--\eqref{zqbjtj} admits a unique global
strong solution $(\n,u,d)$
in $\T^2 \times (0,\infty)$ satisfying, for any $0<T<\infty$,
\be\la{pssol2}\ba
\begin{cases}
\rho\in C([0,T];W^{1,q} ), \quad \n_t\in L^\infty(0,T;L^2), \\ 
u, \na d \in L^\infty(0,T; H^1) \cap L^{(q+1)/q}(0,T; W^{2,q}), \\ 
\sqrt{t} u, \sqrt{t} \na d \in L^2(0,T; W^{2,q}) \cap L^\infty(0,T;H^2), \\
\sqrt{t} u_t, \sqrt{t} \na d_t, \sqrt{t} \na^3 d \in L^2(0,T;H^1), \\
\n u \in C([0,T];L^2), \quad \sqrt{\n} u_t, \na d_t, \na^3 d \in L^2(\T^2 \times(0,T)).
\end{cases}
\ea\ee
Moreover, the global solution $(\n,u,d)$ satisfies the following properties:

1) (Uniform boundedness) There exists a positive constant $C$ depending only on
$\ga$, $\beta$, $\mu$, $\ol{\n}_0$, $\| \n_0 \|_{L^\infty}$, $\| u_0 \|_{H^1}$, and $\| \na d_0 \|_{H^1}$ such that for any $0<T<\infty$,
\be\la{pup}\ba
\sup_{0 \le t \le T} \| \n(\cdot,t) \|_{L^\infty} \le C.
\ea\ee

2) (Exponential decay) For any $p \in [1,\infty)$, there exist positive constants
$C$ and $\alpha_1$ depending only on
$p$, $\ga$, $\beta$, $\mu$, $\ol{\n}_0$, $\| \n_0 \|_{L^\infty}$, $\| u_0 \|_{H^1}$, and $\| \na d_0 \|_{H^1}$ such that for any $1 \le t <\infty$,
\be\la{ped}\ba
\| \n - \ol{\n_0}\|_{L^p} + \| \na u \|_{L^p}
+ \| \na d - \mathbf{k} \otimes d^\bot \|_{H^2}
+ \| d_t + (m_0 \cdot \mathbf{k}) d^\bot \|_{H^1}
\le C e^{-\alpha_1 t},
\ea\ee
where
\be\la{ped3}\ba
& k_i \triangleq \int_0^1 \left( -d_0^2 \p_i d_0^1 + d_0^1 \p_i d_0^2 \right) (x+se_i) ds \quad i = 1,2, \\
& d^\bot \triangleq (-d^2,d^1), \quad
m_0 \triangleq \frac{1}{\ol{\n}_0} \int \n_0 u_0 dx, \quad \mathbf{k} \triangleq (k_1,k_2),
\ea\ee
with
\be\nonumber\ba
e_1 \triangleq (1,0), \quad e_2 \triangleq (0,1).
\ea\ee
\end{theorem}

The second result establishes the global existence and exponential decay of strong solutions in bounded domains.
\begin{theorem}\la{thpb}
Let $\OM \subset \rr$ be a simply connected bounded smooth domain.
Assume that \eqref{pbg} holds and that the initial data $(\n_0,u_0,d_0)$ satisfy, for some $q>2$,
\be\la{ssol1}\ba
0 \le \n_0 \in W^{1,q}, \quad M_0 \triangleq \int_{\OM} \n_0 dx > 0, \quad u_0 \in \tilde{H}^1, \\
\nabla d_{0} \in H^{1},\quad |d_{0}|=1, \quad (n \cdot \na d_0)|_{\p \OM}=0.
\ea\ee
Then the problem \eqref{nlckv}--\eqref{i30} with the boundary conditions \eqref{yjybjtj} admits a unique global strong solution $(\n,u,d)$ in $\OM \times (0,\infty)$ satisfying, for any $0<T<\infty$,
\be\la{ssol2}\ba
\begin{cases}
\rho\in C([0,T];W^{1,q} ), \quad \n_t\in L^\infty(0,T;L^2), \\ 
u, \na d \in L^\infty(0,T; H^1) \cap L^{(q+1)/q}(0,T; W^{2,q}), \\ 
\sqrt{t} u, \sqrt{t} \na d \in L^2(0,T; W^{2,q}) \cap L^\infty(0,T;H^2), \\
\sqrt{t} u_t, \sqrt{t} \na d_t, \sqrt{t} \na^3 d \in L^2(0,T;H^1), \\
\n u \in C([0,T];L^2), \quad \sqrt{\n} u_t, \na d_t, \na^3 d \in L^2(\OM \times(0,T)).
\end{cases}
\ea\ee
Moreover, the global solution $(\n,u,d)$ satisfies the following properties:

1) (Uniform boundedness) There exists a positive constant $C$ depending only on
$\ga$, $\beta$, $\mu$, $M_0$, $\|\n_0 \|_{L^\infty}$, $\| u_0 \|_{H^1}$, $\| \na d_0 \|_{H^1}$, $A$, and $\OM$ such that for any $0<T<\infty$,
\be\la{up}\ba
\sup_{0 \le t \le T} \| \n(\cdot,t) \|_{L^\infty} \le C.
\ea\ee

2) (Exponential decay) For any $p \in [1,\infty)$, there exist positive constants
$C$ and $\alpha_1$ depending only on
$p$, $\ga$, $\beta$, $\mu$, $M_0$, $\| \n_0 \|_{L^\infty}$, $\| u_0 \|_{H^1}$, $\| \na d_0 \|_{H^1}$, $A$, and $\OM$ such that for any $1 \le t <\infty$,
\be\la{ed}\ba
\| \n(\cdot,t)-\ol{\n_0}\|_{L^p} + \| \na u(\cdot,t) \|_{L^p}
+ \| \na d(\cdot,t) \|_{H^2} + \| \na d_t(\cdot,t) \|_{L^2} \le C e^{-\alpha_1 t}.
\ea\ee
\end{theorem}

As a consequence of the decay estimates \eqref{ped} and \eqref{ed}, we can further establish the large-time behavior of the spatial gradient of the density for the strong solutions obtained in Theorems \ref{thpp} and \ref{thpb} when vacuum occurs initially, following the approach in \cite{CL,LX}.
\begin{theorem}\la{thbp2}
Let $\OM=\T^2$ or let $\OM$ be a simply connected bounded smooth domain in $\rr$.
Assume that \eqref{pbg} holds and the initial data satisfy \eqref{pssol1} or \eqref{ssol1}.
Suppose, in addition, that there exists some point $x_0 \in \OM$ such that $\n_0(x_0)=0$.
Let $q$ be as in the corresponding theorem above, and let
$\alpha_1$ denote the decay rate in \eqref{ped} or \eqref{ed}
with $p=2$. Then for any $r \in(2,q]$, there exists a positive constant $C$ depending only on
$r$, $\ga$, $\beta$, $\mu$, $\| u_0 \|_{H^1}$, $\| \rho_0 \|_{L^1\cap L^\infty}$, $\| \na d_0 \|_{H^1}$, $\OM$, and additionally on $A$ in the bounded-domain case, such that for any $t \ge 1$,
\be\la{pbu0}\ba
\| \na \n(\cdot ,t) \|_{L^r} \ge C e^{\alpha_1 \frac{r-2}{r} t }.
\ea\ee
\end{theorem}

A few remarks are in order.

\begin{remark}\la{lrk1}
If the initial data $(\n_0,u_0,d_0)$ further satisfy higher regularity and the compatibility condition:
\be\la{csol2}\ba
- \mu \Delta u_0 - \nabla( (\mu + \lm(\n_0) ) \div u_0 ) + \nabla P(\n_0) + \na d_0 \cdot \Delta d_0 =\n_0^{1/2}g,
\ea\ee
for some $g \in L^2$, then the strong solutions obtained in Theorems \ref{thpp} and \ref{thpb} become classical solutions for positive time.
The detailed proofs follow from arguments analogous to those in \cite{JWX1,LZZ,HL,HLX2}.
\end{remark}

\begin{remark}\la{lrk3}
Compared with the results of Zhong-Zhou \cite{ZZ2}, Theorem \ref{thpb} establishes the global existence of strong solutions without imposing the geometric condition \eqref{jhtj}.
Hence, our results generalize and improve upon those in \cite{ZZ2}.
\end{remark}

\begin{remark}\la{lrk4}
It should be noted that it seems that $\beta>1$ is the critical case for the system \eqref{nlckv}--\eqref{i2} (see \cite{VK}).
Therefore, it would be interesting to study the case $1<\beta \le 4/3$, which is left for the future.
\end{remark}

\begin{remark}\la{lrk5}
When $\beta=0$, the system \eqref{nlckv}--\eqref{i30} reduces to the compressible nematic liquid crystal flows with constant viscosity coefficients.
For this system, Jiang-Jiang-Wang \cite{JJW1,JJW2} established the global existence of finite-energy weak solutions in two-dimensional bounded domains and in the whole space, under the assumptions that $\ga>1$ and that the initial orientation field satisfies the geometric condition \eqref{jhtj}.
The method used to derive the basic energy estimate allows us to remove the geometric condition \eqref{jhtj} and thereby extend their results.
\end{remark}

We now make some comments on the analysis of this paper.
First, the local existence and uniqueness of strong solutions can be established by arguments similar to those in \cite{LLL}.
To extend the local solutions globally in time, we need to derive global a priori estimates, where the key issue is to obtain the upper bound for the density.

Compared with the previous results \cite{ZZ,ZZ2}, in which the initial orientation field is assumed to satisfy the geometric condition (\ref{jhtj}), the main difficulty in removing this assumption lies in obtaining an $L^2(\OM \times (0,T))$ estimate for $\na^2 d$.
Specifically, the standard energy estimate provides only an $L^2(\OM \times (0,T))$ bound for $\Delta d + |\na d|^2 d$.
To obtain an estimate for $\na^2 d$, the arguments in \cite{ZZ,ZZ2} use the geometric condition (\ref{jhtj}) together with the constraint $|d|=1$ and apply the maximum principle to $d$, yielding
\be\la{gjgj}\ba
\| \Delta d + |\na d|^2 d \|^2_{L^2} \ge \frac{\bar{\o}}{2} ( \| \Delta d \|^2_{L^2} + \| \na d \|^4_{L^4} ), \  \text{ for some } \bar{\o} \in (0,1).
\ea\ee
This, combined with the standard energy estimate, gives the desired $L^2(\OM \times (0,T))$ estimate for $\na^2 d$.
However, without the geometric condition (\ref{jhtj}), inequality (\ref{gjgj}) can no longer be derived.
To overcome this difficulty, we fully exploit the condition $|d|=1$ by representing $d$ in polar coordinates as $d=(\cos \theta, \sin \theta )$ (see Lemmas \ref{jzb} and \ref{zqjzb}).
A direct calculation shows that
\be\nonumber\ba
|\Delta d + |\na d|^2 d|^2 = |\Delta \theta|^2, \quad  |\na d|^2 = |\na \theta|^2.
\ea\ee
In a bounded simply connected smooth domain, these identities, the Neumann estimate, and the Gagliardo-Nirenberg inequality give the required time-uniform $L^2(\OM \times (0,T))$ bound for $\na^2 d$.
On the torus we instead control the periodic angle fluctuation, as explained below.

Moreover, in contrast to the compressible Navier-Stokes equations, the introduction of the liquid crystal director field $d$ presents new challenges arising from the strong coupling term $u \cdot \na d$ and the nonlinear terms $\na d \cdot \Delta d$ and $|\na d|^2 d$.
To handle these terms, we exploit the parabolic structure of the director equation $(\ref{nlckv})_3$ to derive a key $L^\infty(0,T;L^p)$ estimate for $\na d$, where $2<p<\infty$.
This estimate plays a crucial role in controlling the coupling terms throughout the analysis.

In the periodic domain case, although the director field $d$ satisfies the periodic boundary condition, when $d$ is represented in polar coordinates, we cannot guarantee that $\theta$ is also periodic.
In fact, we can only conclude that $\na \theta$ is periodic.
The lack of periodicity of $\theta$ prevents us from deriving a time-uniform $L^2(\T^2 \times (0,T))$ bound for $\na^2 d$, and hence from establishing a time-uniform upper bound for the density and the large-time behavior.
To overcome this difficulty, we observe that the function
\be\nonumber\ba
\phi \triangleq \theta - \mathbf{k} \cdot x,
\ea\ee
is periodic, where $\mathbf{k}=(k_1,k_2)$ is defined in (\ref{ped3}).
We therefore reformulate the system (\ref{nlckv}) in terms of $\phi$ and carry out the corresponding time-uniform a priori estimates.
Combining these estimates with an adaptation of the framework developed in \cite{HL3}, we derive the upper bound for the density and the required higher-order derivative estimates, provided that $\beta>4/3$.

In the bounded domain case, the Neumann boundary condition for $d$ implies the corresponding Neumann boundary condition for $\theta$:
\be\nonumber\ba
\frac{\p \theta}{\p n} = -d^2 \frac{\p d^1}{\p n} + d^1 \frac{\p d^2}{\p n} = 0 \quad \text{ on } \p \OM.
\ea\ee
This, together with the standard energy estimate, the Neumann elliptic estimates, and the Gagliardo-Nirenberg inequality, yields the desired time-uniform $L^2(\OM \times (0,T))$ bound for $\na^2 d$.
This bound also allows us to derive the required $L^\infty(0,T;L^p)$ estimate for $\na d$.
Moreover, the boundary terms require additional care.
On the one hand, the Neumann boundary condition $n \cdot \na d = 0$ on $\p \OM$ implies that, for any $j =1,2$,
\be\nonumber
\na d^j \cdot \na (n \cdot \na d^j) = 0 \quad \text{ on } \p \OM.
\ee
This identity is essential for controlling the boundary term arising in the $L^p$ estimate for $\na d$.
On the other hand, the higher-order energy estimates rely on the following two identities from \cite{CL}:
\be\la{bjds}\ba
u=(u \cdot n^{\bot})n^{\bot}, \quad (u \cdot \na) u \cdot n=-(u \cdot \na) n \cdot u \quad \text{ on } \p \om,
\ea\ee
which follow from the boundary condition $u \cdot n = 0$ on $\p \OM$.
Using these identities and following the approach in \cite{FLL,FLW}, we establish the upper bound for the density and the higher-order derivative estimates under the condition $\beta>4/3$.

The rest of this paper is organized as follows: Section 2 collects several known results and inequalities needed in the subsequent analysis.
Sections 3 and 4 are devoted to the derivation of the a priori estimates in the periodic domain $\OM=\T^2$ and in bounded simply connected domains, respectively.
Finally, Section 5 presents the proofs of our main results, Theorems \ref{thpp}--\ref{thbp2}.

\section{Preliminaries}
In this section, we recall some known facts and elementary inequalities that will be used frequently later.

First, we have the following local existence result for strong solutions, which can be proved by arguments similar to those in \cite{LLL}.
\begin{lemma}\la{lct}
Let $\beta \ge 1$ and $\ga>1$.
Assume the initial data $\left(\n_0,u_0,d_0 \right)$ satisfy \eqref{pssol1} or \eqref{ssol1}.
Then there exists a small time $T>0$ such that the problem \eqref{nlckv}--\eqref{i30} with the boundary conditions \eqref{zqbjtj} or \eqref{yjybjtj} admits a unique strong solution $(\n,u,d)$ in $\OM \times (0,T]$.
Moreover, the solution satisfies \eqref{pssol2} or \eqref{ssol2}.
\end{lemma}

The following Gagliardo-Nirenberg inequalities (see \cite{NI,TG}) will be used frequently.
\begin{lemma}\la{gn1}
Assume that $\OM=\T^2$ or $\OM$ is a bounded smooth domain in $\rr$.
For $2<p<\infty$, there exists a positive constant $C$ depending only on $\OM$ such that for any $f \in H^1(\OM)$,
\be\ba\la{gn11}
\| f \|_{L^p} \le Cp^{1/2}\| f \|^{2/p}_{L^2} \| f \|^{1-2/p}_{H^1}.
\ea\ee
Moreover, the term $\| f \|_{H^1}$ on the right-hand side of \eqref{gn11} can be replaced by $\| \na f \|_{L^2}$ if $\ol{f}=0$ or $f \cdot n|_{\p \OM}=0$.
\end{lemma}

The following elliptic estimate plays an important role in deriving the time-uniform estimates in bounded domains.
Its proof can be found in \cite{ADN,GT}.
\begin{lemma}\la{tygj}
Let $\OM$ be a bounded smooth domain in $\rr$.
Then for any $f\in H^{2}(\OM)$ satisfying $\frac{\p f}{\p n}=0$ on $\p \OM$, there exists a positive constant $C$ depending only on $\OM$ such that
\be\la{tygj1}\ba
\| \na^2 f \|_{L^2} \le C \| \Delta f \|_{L^2}.
\ea\ee
\end{lemma}

We will frequently use the following div-curl estimates (see \cite{AJ,MD,WWV}) to bound $\na u$.
\begin{lemma}\la{dc}
Let $k \ge 0$ be an integer and $1<p<\infty$, and let $\OM$ be a simply connected bounded domain in $\rr$ with $C^{k+1,1}$ boundary $\p \OM$.
Then there exists a positive constant $C$ depending only on $k$, $p$, and $\OM$ such that for every $u\in W^{k+1,p}(\OM)$ with $u \cdot n=0$ on $\p \OM$, the following estimate holds:
\be\ba\la{dc1}
\| \na u \|_{W^{k,p}} \le C\left( \|\div u\|_{W^{k,p}} +\| \curl u \|_{W^{k,p}} \right).
\ea\ee
\end{lemma}

Let $\mathcal{H}^1(\T^2)$ and $\mathcal{BMO}(\T^2)$ denote the usual Hardy and BMO spaces.
Given a function $b$, define the commutator
\be\ba\la{p4}
[b,R_iR_j](f) \triangleq bR_i\circ R_j(f) - R_i \circ R_j(bf),\quad i,j=1,2,
\ea\ee
where $R_i = (-\Delta)^{-\frac{1}{2}}\pa_i$ is the usual Riesz transform on $\T^2$.

Now, we present the following well-known commutator estimates established by Coifman-Rochberg-Weiss \cite{CRW} and Coifman-Meyer \cite{CM}.
\begin{lemma}\la{jhzyl}
Let $b,f\in C^\infty(\T^2)$.
Then for $p\in (1,\infty)$, there is a positive constant $C$ depending only on $p$ such that
\be\ba\la{jhz1}
\|[b,R_iR_j](f)\|_{L^p}\leq C \| b\|_{\mathcal{BMO}}\|f\|_{L^p}.
\ea\ee
Moreover, for $p,q,r \in (1,\infty)$ with $\frac{1}{r} = \frac{1}{p}+ \frac{1}{q}$, there exists a positive constant $C$ depending only on $p$, $q$, and $r$ such that
\be\ba\la{jhz2}
\|\na[b,R_iR_j](f)\|_{L^r}\leq C \|\na b\|_{L^p}\|f\|_{L^q}.
\ea\ee
\end{lemma}

The following Brezis-Wainger inequality (see \cite{BW,E}) will be used to derive the upper bound for the density in the periodic domain.
\begin{lemma}\la{BWI}
For $2<q<\infty$, there exists some positive constant $C$ depending only on $q$
such that for any $v \in W^{1,q}(\T^2)$,
\be\la{bwi}\ba
\| v \|_{L^\infty(\T^2)} \le C \| \na v \|_{L^2(\T^2)} \log^{1/2}(e+\| \na v \|_{L^q(\T^2)}) + C \| v \|_{L^2(\T^2)} + C.
\ea\ee
\end{lemma}

To estimate $\| \na u\|_{L^{\infty}}$ and $\| \na \n\|_{L^{q}}$, we require the following Beale-Kato-Majda type inequality; its proof can be found in \cite{K,BKM,CL}.
\begin{lemma}\la{bkm}
Let $2<q<\infty$.
Assume that either $\OM$ is a bounded smooth domain in $\rr$ and
$u \in \left\{ W^{2,q}(\Omega ) \big| u \cdot n=0, \curl u = -A u \cdot n^\bot \,\,\,\text{on}\,\,\, \partial\Omega \right\}$
or $\OM = \T^2$ and $u \in W^{2,q}(\T^2)$.
Then there exists a positive constant $C$ depending only on $q$ and $\OM$ such that
\be\ba\la{bkm1}
\|\na u\|_{L^\infty} \le C \left( \|\div u \|_{L^\infty}
+ \|\curl u\|_{L^\infty} \right) \log \left(e+ \|\na^2 u\|_{L^q} \right)+ C\|\na u\|_{L^2}+C.
\ea\ee
\end{lemma}

We now state the following ``inversion'' of the divergence operator in bounded domains; see \cite{CL} for details.
\begin{lemma}\la{iod}
Let $\OM$ be a bounded smooth domain.
For $1<p<\infty$, there exists a bounded linear operator $\mathcal{B}$ given by
\be\ba\nonumber
\mathcal{B}:\left\{f \in L^p(\OM) :  \int_\Omega fdx=0 \right\}&\rightarrow (W^{1,p}_0(\OM))^2,
\ea\ee
such that $v=\mathcal{B}(f)$ solves
\be\la{iod1}\ba
\begin{cases}
\mathrm{div}v=f&\ \textnormal{ in }\Omega,\\
v=0&\ \textnormal{ on }\partial\Omega.
\end{cases}
\ea\ee
Moreover, the operator has the following properties:

(1) For $1<p<\infty$, there exists a positive constant $C$ depending on $\Omega$ and $p$ such that
\bnn
\|\mathcal{B}(f)\|_{W^{1,p}}\leq C(p)\|f\|_{L^p}.
\enn

(2) If $f=\mathrm{div}h$, for some $h\in L^p$ with $h\cdot n=0$ on $\partial\Omega$, then $v=\mathcal{B}(f)$ is a weak solution to the problem \eqref{iod1} and satisfies
\bnn
\|\mathcal{B}(f)\|_{L^p}\leq C(p) \| h \|_{L^p}.
\enn
\end{lemma}

The following Zlotnik inequality (see \cite{ZAA}) will be used to derive the time-uniform upper bound for the density $\n$.
\begin{lemma}\la{zli}
Let the function $y(t) \in W^{1,1}(0,T)$ satisfy
\bnn
y'(t)= g(y)+h'(t) \textnormal{ on } [0,T], \quad y(0)=y_0, 
\enn
with $ g\in C(\r)$ and $h \in W^{1,1}(0,T).$ If $g(\infty)=-\infty$
and 
\bnn 
h(t_2)-h(t_1) \le N_0 +N_1(t_2-t_1),
\enn
for all $0 \le t_1<t_2\le T$
with some $N_0\ge 0$ and $N_1\ge 0$, then
\bnn
y(t)\le \max\left\{y_0,\overline{\zeta} \right\}+N_0<\infty
\textnormal{ on } [0,T],
\enn
where $\overline{\zeta}$ is a constant such
that 
\bnn
g(\zeta)\le -N_1 \quad \textnormal{ for }\quad \zeta\ge \overline{\zeta}.
\enn
\end{lemma}

Next, we show that any unit vector field defined on a simply connected domain in $\rr$ admits a polar representation.
\begin{lemma}\la{jzb}
Let $\OM \subset \rr$ be a simply connected domain.
Assume that $d=(d^1,d^2) \in H^2_{\textnormal{loc}}(\OM;\rr)$ with $|d|=1$.
Then there exists a function $\theta \in H^2_{\textnormal{loc}}(\OM)$ such that
\be\la{jzb1}\ba
d=(\cos \theta, \sin \theta).
\ea\ee
\end{lemma}
\begin{proof}
Define the vector field
\be\nonumber\ba
Q = (Q_1,Q_2),
\ea\ee
where
\be\ba\la{jzb101}
Q_1 \triangleq -d^2 \frac{\p d^1}{\p x_1} + d^1 \frac{\p d^2}{\p x_1}, \quad
Q_2 \triangleq -d^2 \frac{\p d^1}{\p x_2} + d^1 \frac{\p d^2}{\p x_2}.
\ea\ee
A direct computation shows that
\be\ba\la{jzb11}
\frac{\p Q_1}{\p x_2} - \frac{\p Q_2}{\p x_1}
= 2 \left( \frac{\p d^1}{\p x_2} \frac{\p d^2}{\p x_1} - \frac{\p d^1}{\p x_1} \frac{\p d^2}{\p x_2} \right).
\ea\ee
Since $|d|=1$, we have
\be\ba\la{jzb12}
d^1\frac{\p d^1}{\p x_1} + d^2\frac{\p d^2}{\p x_1} = 0, \quad
d^1\frac{\p d^1}{\p x_2} + d^2\frac{\p d^2}{\p x_2} = 0.
\ea\ee
Combining (\ref{jzb11}) with (\ref{jzb12}) and using $|d|=1$, we obtain
\be\ba\nonumber
\na^\bot \cdot Q = \frac{\p Q_1}{\p x_2} - \frac{\p Q_2}{\p x_1} = 0.
\ea\ee
Since $\OM$ is simply connected, Poincar\'e's theorem (see \cite{Ma}) together with $d\in H^2_{\textnormal{loc}}(\OM;\rr)$ and (\ref{jzb101}) implies that there exists $\theta \in H^2_{\textnormal{loc}}(\OM)$ such that
\be\ba\la{jzb13}
\na \theta = Q = -d^2 \na d^1 + d^1 \na d^2.
\ea\ee

It remains to verify that $\theta$ satisfies (\ref{jzb1}).
The Sobolev embedding theorem ensures that $d$ and $\theta$ have continuous representatives.
Moreover, since $\theta$ is determined by (\ref{jzb13}) only up to an additive constant, we may choose this constant such that
\be\ba\nonumber
d(x_0)=(\cos \theta(x_0), \sin \theta(x_0)),
\ea\ee
for some $x_0 \in \OM$.
Set
\be\ba\la{jzb14}
U_1 \triangleq d^1 - \cos \theta, \quad U_2 \triangleq d^2 - \sin \theta.
\ea\ee
Using (\ref{jzb13}), a straightforward calculation yields
\be\ba\nonumber
\na U_1 = \left( 1 - d^2 \sin \theta \right) \na d^1 + d^1 \sin \theta \na d^2, \quad
\na U_2 = d^2 \cos \theta \na d^1 + \left( 1 - d^1 \cos \theta \right) \na d^2.
\ea\ee
Consequently, we have
\be\ba\la{jzb15}
\frac{1}{2}\na \left( U^2_1 + U^2_2 \right) = U_1 \na U_1 + U_2 \na U_2
= \left( 1 - d^2 \sin \theta - d^1 \cos \theta \right) \left( d^1 \na d^1 + d^2 \na d^2 \right).
\ea\ee
In view of (\ref{jzb12}), the right-hand side of (\ref{jzb15}) vanishes identically.
Hence, $U^2_1 + U^2_2$ is constant in $\OM$.
Since $U_1(x_0)=U_2(x_0)=0$, it follows that
\be\nonumber\ba
U_1(x)=U_2(x)=0 \quad \text{ for all } x \in \OM.
\ea\ee
Combining this with (\ref{jzb14}) yields (\ref{jzb1}) and completes the proof.
\end{proof}

Finally, we establish the corresponding polar representation for unit vector fields on $\T^2$.
\begin{lemma}\la{zqjzb}
Assume that $d=(d^1,d^2) \in H^2(\T^2;\rr)$ with $|d|=1$.
Then there exist functions $\theta \in H^2_{\textnormal{loc}}(\rr)$ and $\phi \in H^2(\T^2)$ such that
\be\la{zqjzb02}\ba
d=(\cos \theta, \sin \theta),
\ea\ee
and
\be\la{zqjzb03}\ba
\theta(x) = \phi(x) + \mathbf{c} \cdot x,
\ea\ee
where $\mathbf{c}=(c_1,c_2)$ is a constant vector whose components are given by
\be\la{zqjzb01}\ba
c_i \triangleq \int_0^1 \left( -d^2 \p_i d^1 + d^1 \p_i d^2 \right) (x+se_i) ds \quad i = 1,2,
\ea\ee
with
\be\nonumber\ba
e_1 \triangleq (1,0), \quad e_2 \triangleq (0,1).
\ea\ee
\end{lemma}
\begin{proof}
Since $d=(d^1,d^2) \in H^2(\T^2)$ with $|d|=1$, we can extend $d$ to $\rr$ such that $d=(d^1,d^2) \in H^2_{\textnormal{loc}}(\rr)$ and $|d|=1$.
Since $\rr$ is simply connected, Lemma \ref{jzb} implies that there exists $\theta \in H^2_{\textnormal{loc}}(\rr)$ such that
\be\la{zqjzb1}\ba
d=(\cos \theta, \sin \theta),
\ea\ee
and
\be\la{zqjzb2}\ba
\na \theta = -d^2 \na d^1 + d^1 \na d^2.
\ea\ee
Since $d$ is periodic, we obtain from (\ref{zqjzb2}) that $\na \theta$ is also periodic.
Hence, for $i=1,2$,
\be\la{zqjzb3}\ba
\na( \theta(x+e_i) - \theta(x) ) = 0 \quad \text{ a.e. in } \rr,
\ea\ee
which shows that $\theta(x+e_i) - \theta(x)$ is constant.

Moreover, the fundamental theorem of calculus and (\ref{zqjzb2}) give
\be\la{zqjzb4}\ba
\theta(x+e_i) - \theta(x) & = \int_0^1 \p_s \theta(x+se_i) ds = \int_0^1 \p_i \theta(x+se_i) ds \\
& = \int_0^1 \left( -d^2 \p_i d^1 + d^1 \p_i d^2 \right) (x+se_i) ds = c_i.
\ea\ee
Thus, $\mathbf{c}=(c_1,c_2)$ is a constant vector and
\be\la{zqjzb5}\ba
\theta(x+e_i) = \theta(x) + c_i.
\ea\ee
Define
\be\la{zqjzb6}\ba
\phi(x) \triangleq \theta(x) - \mathbf{c} \cdot x.
\ea\ee
Then, for any $x \in \rr$ and $i=1,2$,
\be\la{zqjzb7}\ba
\phi(x+e_i)
& = \theta(x+e_i) - \mathbf{c} \cdot (x+e_i)
= \theta(x) + c_i - \mathbf{c} \cdot x - c_i \\
& = \theta(x) - \mathbf{c} \cdot x = \phi(x).
\ea\ee
Combining (\ref{zqjzb6}) and (\ref{zqjzb7}), we obtain $\phi \in H^2(\T^2)$ and (\ref{zqjzb03}) holds, which finishes the proof of Lemma \ref{zqjzb}.
\end{proof}

\section{A priori estimates on $\T^2$}

In this section, we assume that (\ref{pbg}) holds and the initial data $(\n_0,u_0,d_0)$ satisfy (\ref{pssol1}) with $\inf_{x \in \T^2}\n_0(x)>0$.
Let $(\n,u,d)$ be the strong solution to (\ref{nlckv})--(\ref{zqbjtj}) on $\T^2 \times (0,T]$, provided by Lemma \ref{lct}.

Since $|d|=1$, from Lemma \ref{zqjzb}, we conclude that, for any fixed $t \in [0,T]$, there exist functions $\theta \in H^2_{\textnormal{loc}}(\rr)$ and $\phi \in H^2(\T^2)$ such that
\be\la{zqbh1}\ba
d(x,t)=(\cos \theta(x,t), \sin \theta(x,t)),
\ea\ee
and
\be\la{zqbh2}\ba
\theta(x,t) = \phi(x,t) + \mathbf{k}(t) \cdot x,
\ea\ee
where $\mathbf{k}(t)=(k_1(t),k_2(t))$ is a spatially constant vector whose components are given by
\be\la{zqbh3}\ba
k_i(t) = \int_0^1 \left( -d^2 \p_i d^1 + d^1 \p_i d^2 \right) (x+se_i,t) ds \quad i = 1,2.
\ea\ee
We next verify that $\mathbf{k}(t)$ is independent of time.
By $|d|=1$, we have
\be\la{zqbh4}\ba
d \cdot \p_i d = 0, \quad d \cdot \p_t d = 0,
\ea\ee
which implies
\be\la{zqbh5}\ba
\p_i d^1 \p_t d^2 - \p_i d^2 \p_t d^1 = 0.
\ea\ee
Combining (\ref{zqbh4}) and (\ref{zqbh5}), we arrive at
\be\la{zqbh6}\ba
& \frac{d}{dt} \int_0^1 \left( -d^2 \p_i d^1 + d^1 \p_i d^2 \right) (x+se_i,t) ds \\
& = \int_0^1 \left( -\p_t d^2 \p_i d^1 + \p_t d^1 \p_i d^2 -d^2 \p_t \p_i d^1 + d^1 \p_t \p_i d^2 \right) (x+se_i,t) ds \\
& = \int_0^1 \left( \p_t d^2 \p_i d^1 - \p_t d^1 \p_i d^2 -d^2 \p_t \p_i d^1 + d^1 \p_t \p_i d^2 \right) (x+se_i,t) ds \\
& = \int_0^1 \p_i ( \p_t d^2 d^1 - \p_t d^1 d^2 ) (x+se_i,t) ds \\
& = \int_0^1 \p_s ( \p_t d^2 d^1 - \p_t d^1 d^2 ) (x+se_i,t) ds = 0,
\ea\ee
where in the last equality we have used the periodicity of $d$ and $d_t$.
Thus,
\be\la{zqbh6a}\ba
k_i(t) = k_i(0) = k_i, \quad i=1,2,
\ea\ee
with $k_i$ as defined in (\ref{ped3}).

Recall that
\be\la{zqbh7}\ba
d^\bot \triangleq (-d^2,d^1).
\ea\ee
A direct calculation yields
\be\la{zqbh8}\ba
d_t = d^\bot \theta_t, \quad \p_i d = d^\bot \p_i \theta, \quad \Delta d + d |\na d|^2 = d^\bot \Delta \theta,
\ea\ee
and
\be\la{zqbh9}\ba
\na \theta = \na \phi + \mathbf{k}, \quad \theta_t = \phi_t, \quad \Delta \theta = \Delta \phi.
\ea\ee
It follows from (\ref{zqbh8}) and (\ref{zqbh9}) that the system (\ref{nlckv}) becomes
\be\ba\la{zqbhfc}
\begin{cases}
\rho_t + \div(\rho u) = 0, \\
(\n u)_t + \div(\n u\otimes u) -\mu \Delta u 
-\na ( (\mu + \lm) \div u) + \na P = - (\na \phi + \mathbf{k}) \Delta \phi, \\
\phi_t + u \cdot (\na \phi + \mathbf{k}) = \Delta \phi.
\end{cases}
\ea\ee

Moreover, we set
\be\ba\la{a1}
A_1^2(t) \triangleq \int \left( (2\mu + \lam(\n)) (\div u)^2 + |\na u|^2
+ |\na^2 \phi|^2 + (\n+1)^{\ga-1} (\n-\ol{\n})^2 \right) dx,
\ea\ee
\be\ba\la{a2}
A_2^2(t) \triangleq \int \left( \rho |\dot{u}|^2 + |\na \Delta \phi|^2 + |\na \phi_t|^2 \right) dx,
\ea\ee
and
\be\ba\la{mdsj}
R_T \triangleq 1 + \sup_{0 \le t \le T} \| \n(t) \|_{L^\infty}.
\ea\ee

\subsection{A priori estimates (I): upper bound of the density}

This subsection is devoted to deriving the upper bound for the density.

We begin with the standard energy estimate for the transformed system.

\begin{lemma}\la{ppl1}
There exists a positive constant
$C$ depending only on $\mu$, $\gamma$, $\| \n_0 \|_{L^\infty}$, $\| u_0 \|_{H^1}$, and $\| \na d_0 \|_{L^2}$ such that
\be\ba\la{pp01}
& \sup_{0\leq t\leq T} \int \left( \rho |u|^2 + \n^\ga + |\na \phi|^2 \right) dx \\
& + \int_0^T \int \left( |\na u|^2 + (2\mu + \lam(\n)) (\div u)^2 + |\na^2 \phi|^2 \right) dx dt
\le C.
\ea\ee
\end{lemma}
\begin{proof}
First, multiplying $(\ref{zqbhfc})_2$ by $u$, integrating by parts over $\T^2$, and using $(\ref{zqbhfc})_1$, we derive
\be\la{3pp1}\ba
& \frac{d}{dt} \int \left( \frac{1}{2}\rho |u|^2 + \frac{P}{\ga-1} \right) dx
+ \int \left( \mu |\na u|^2 + (\mu + \lam(\n)) (\div u)^2 \right) dx \\
& = - \int u \cdot (\na \phi + \mathbf{k}) \Delta \phi dx.
\ea\ee
On the other hand, multiplying $(\ref{zqbhfc})_3$ by $\Delta \phi$ and integrating by parts over $\T^2$, we arrive at
\be\la{3pp2}\ba
\frac{d}{dt} \int \frac{1}{2} |\na \phi|^2 dx + \int |\Delta \phi|^2 dx
= \int u \cdot (\na \phi + \mathbf{k}) \Delta \phi dx.
\ea\ee
Adding (\ref{3pp1}) and (\ref{3pp2}) yields
\be\la{3pp3}\ba
& \frac{d}{dt} \int \left( \frac{1}{2}\rho |u|^2 + \frac{P}{\ga-1} + \frac{1}{2} |\na \phi|^2 \right) dx \\
& + \int \left( \mu |\na u|^2 + (\mu + \lam(\n)) (\div u)^2 + |\Delta \phi|^2 \right) dx = 0.
\ea\ee
Moreover, by (\ref{zqbh8}) and (\ref{zqbh9}), we have
\be\nonumber\ba
\int |\na d(0)|^2 dx = \int |\na \theta(0)|^2 dx = \int |\na \phi(0) + \mathbf{k}|^2 dx
= \int \left( |\na \phi(0)|^2 + |\mathbf{k}|^2 \right) dx,
\ea\ee
which gives
\be\la{3pp4}\ba
\|\na \phi(0)\|_{L^2} \le \| \na d_0 \|_{L^2}, \quad |\mathbf{k}| \le \| \na d_0 \|_{L^2}.
\ea\ee
Integrating (\ref{3pp3}) over $(0,T)$ and using (\ref{3pp4}), we obtain
\be\la{3pp5}\ba
& \sup_{0 \le t \le T} \int \left( \frac{1}{2}\rho |u|^2 + \frac{P}{\ga-1} + \frac{1}{2} |\na \phi|^2 \right) dx \\
& + \int_0^T \int \left( \mu |\na u|^2 + (\mu + \lam(\n)) (\div u)^2 + |\Delta \phi|^2 \right) dx dt \le C,
\ea\ee
which yields (\ref{pp01}) and completes the proof of Lemma \ref{ppl1}.
\end{proof}

The following $L^\infty(0,T;L^p)$ estimate for $\na \phi$ will be used frequently in the subsequent analysis.
\begin{lemma}\la{ppl2}
For any $2<p<\infty$, there exists a positive constant $C$ depending only on
$\mu$, $\ga$, $p$, $\| \n_0 \|_{L^\infty}$, $\| u_0 \|_{H^1}$, and $\| \na d_0 \|_{H^1}$ such that
\be\ba\la{pp02}
\sup_{0\leq t\leq T} \| \na \phi \|^p_{L^p} + \int_0^T \int |\na \phi|^{p-2} |\na^2 \phi|^2 dx dt \leq C.
\ea\ee
\end{lemma}
\begin{proof}
First, applying $\na$ to $\eqref{zqbhfc}_3$, we obtain
\be\la{pp21}\ba
\na \phi_t - \na \Delta \phi = - \na \left( u \cdot (\na \phi + \mathbf{k}) \right).
\ea\ee
Multiplying (\ref{pp21}) by $p |\na \phi|^{p-2} \na \phi$, integrating by parts over $\T^2$, and using (\ref{gn11}), (\ref{3pp4}), and Young's inequality, we derive
\be\la{pp22}\ba
& \frac{d}{dt} \int |\na \phi|^p dx + p \int |\na \phi|^{p-2} |\na^2 \phi|^2 dx
+ \frac{4(p-2)}{p} \int |\na (|\na \phi|^\frac{p}{2})|^2 dx \\
& \le -p \int |\na \phi|^{p-2} \p_i \phi u^j \p_{ij} \phi dx
+ C \int |\na \phi|^{p-1} |\na u| \left( |\na \phi| + 1 \right) dx \\
& \le C \int |\na u| |\na \phi|^p dx + C \int |\na u| |\na \phi| dx \\
& \le C \| \na u \|_{L^2} \| |\na \phi|^{\frac{p}{2}} \|^2_{L^4}
+ C \| \na u \|_{L^2} \| \na \phi \|_{L^2} \\
& \le C \| \na u \|_{L^2} \| |\na \phi|^{\frac{p}{2}} \|_{L^2} \| |\na \phi|^{\frac{p}{2}} \|_{H^1}
+ C \| \na u \|_{L^2} \| \na^2 \phi \|_{L^2} \\
& \le \frac{p}{2} \int |\na \phi|^{p-2} |\na^2 \phi|^2 dx + C \| \na \phi \|^p_{L^p}
+ C \| \na u \|^2_{L^2} \| \na \phi \|^p_{L^p} + C \| \na u \|_{L^2} \| \na^2 \phi \|_{L^2} \\
& \le \frac{p}{2} \int |\na \phi|^{p-2} |\na^2 \phi|^2 dx
+ C \| \na \phi \|^{p-2}_{L^{p}} \| \na^2 \phi \|^2_{L^2}
+ C \| \na u \|^2_{L^2} \| \na \phi \|^p_{L^p} + C \| \na u \|_{L^2} \| \na^2 \phi \|_{L^2} \\
& \le \frac{p}{2} \int |\na \phi|^{p-2} |\na^2 \phi|^2 dx
+ C \| \na \phi \|^p_{L^p} ( \| \na^2 \phi \|^2_{L^2} + \| \na u \|^2_{L^2} )
+ C \| \na u \|^2_{L^2} + C \| \na^2 \phi \|^2_{L^2},
\ea\ee
which yields
\be\la{pp25}\ba
& \frac{d}{dt} \int |\na \phi|^p dx + \frac{p}{2} \int |\na \phi|^{p-2} |\na^2 \phi|^2 dx \\
& \le C \| \na \phi \|^p_{L^p} ( \| \na^2 \phi \|^2_{L^2} + \| \na u \|^2_{L^2} )
+ C \| \na u \|^2_{L^2} + C \| \na^2 \phi \|^2_{L^2}.
\ea\ee
Applying Gr\"onwall's inequality to (\ref{pp25}) and using (\ref{pp01}), (\ref{zqbh8}), and (\ref{zqbh9}), we arrive at (\ref{pp02}) and complete the proof of Lemma \ref{ppl2}.
\end{proof}

By combining (\ref{pp01}) and (\ref{pp02}) with an adaptation of the arguments in \cite[Proposition 5.2]{FLW}, we can obtain the following estimates.

\begin{lemma}\la{ppl3}
Let $g_{+} \triangleq \max\{ g,0 \}$.
Then, for any $2 \le p <\infty$, there exist positive constants $C$ and $M_1$ depending only on
$p$, $\mu$, $\ga$, $\beta$, $\| \n_0 \|_{L^\infty}$, $\| u_0 \|_{H^1}$, and $\| \na d_0 \|_{H^1}$ such that
\be\la{pyzgj}\ba
\sup_{0\le t\le T} \| \n \|_{L^p} + \int_0^T \int (\n-M_1)^p_{+} dxdt \le C.
\ea\ee
\end{lemma}

\begin{lemma}\la{ppl4}
There exists a positive constant $C$ depending only on
$\mu$, $\ga$, $\beta$, $\ol{\n}_0$, $\| \n_0 \|_{L^\infty}$, $\| u_0 \|_{H^1}$, and $\| \na d_0 \|_{H^1}$ such that
\be\la{pp04}\ba
\int_0^T \int (\n+1)^{\ga-1} (\n-\ol{\n})^2 dx dt \le C.
\ea\ee
\end{lemma}
\begin{proof}
First, integrating the mass equation $(\ref{zqbhfc})_1$ over $\T^2$, we arrive at
\be\la{pp401}\ba
\int \n dx = \int \n_0 dx = \ol{\n}_0.
\ea\ee
Observe that
\be\la{pp41}\ba
(\na \phi + \mathbf{k}) \Delta \phi
= \div \left( (\na \phi+\mathbf{k}) \otimes (\na \phi+\mathbf{k}) \right) - \frac{1}{2} \na \left( |\na \phi+\mathbf{k}|^2 \right).
\ea\ee
Integrating the momentum equation $(\ref{zqbhfc})_2$ over $\T^2$ and using (\ref{pp41}), we obtain the conservation of total momentum:
\be\la{pp42}\ba
\int \n u dx = \int \n_0 u_0 dx = \ol{\n}_0 m_0.
\ea\ee
It follows from (\ref{pp01}), (\ref{pp42}), and Poincar\'e's inequality that for any $2 \le p <\infty$,
\be\la{pp45}\ba
\| u - m_0 \|_{L^p}
& \le \| u - \ol{u} \|_{L^p} + |\ol{u}-m_0| \\
& \le C \| \na u \|_{L^2} + \frac{1}{\ol{\n}_0} \left| \int \n (\ol{u}-u) dx \right| \\
& \le C \| \na u \|_{L^2} + C \| \n \|_{L^\ga} \| u - \ol{u} \|_{L^{\frac{\ga}{\ga-1}}} \\
& \le C \| \na u \|_{L^2}.
\ea\ee
We rewrite the momentum equation $(\ref{zqbhfc})_2$ as
\be\la{pp46}\ba
\na(\n^\ga - \ol{\n}^{\ga}) = - \n (u - m_0)_t - \n u \cdot \na u + \mu \Delta u + \na( (\mu + \lam) \div u ) - (\na \phi + \mathbf{k}) \Delta \phi.
\ea\ee
Multiplying (\ref{pp46}) by $\na (-\Delta)^{-1}(\n-\ol{\n})$ and integrating by parts over $\T^2$, we arrive at
\be\la{pp47}\ba
& \int (\n^\ga - \ol{\n}^{\ga}) (\n-\ol{\n}) dx \\
& = - \int \n (u - m_0)_t \cdot \na (-\Delta)^{-1}(\n-\ol{\n}) dx
- \int \n u \cdot \na u \cdot \na (-\Delta)^{-1}(\n-\ol{\n}) dx \\
& \quad - \mu \int \na u : \na^2 (-\Delta)^{-1}(\n-\ol{\n}) dx
+ \int (\mu + \lam) \div u (\n-\ol{\n}) dx \\
& \quad - \int (\na \phi + \mathbf{k}) \Delta \phi \cdot \na (-\Delta)^{-1}(\n-\ol{\n}) dx
\triangleq \sum_{i=1}^{5} I_i.
\ea\ee
By $(\ref{zqbhfc})_1$, we have
\be\la{pp48}\ba
I_1 & = - \frac{d}{dt} \int \n (u - m_0) \cdot \na (-\Delta)^{-1}(\n-\ol{\n}) dx
+ \int \n_t (u - m_0) \cdot \na (-\Delta)^{-1}(\n-\ol{\n}) dx \\
& \quad + \int \n (u - m_0) \cdot \na (-\Delta)^{-1}\n_t dx \\
& = - \frac{d}{dt} \int \n (u - m_0) \cdot \na (-\Delta)^{-1}(\n-\ol{\n}) dx
- \int \div(\n u) (u - m_0) \cdot \na (-\Delta)^{-1}(\n-\ol{\n}) dx \\
& \quad - \int \n (u - m_0) \cdot \na (-\Delta)^{-1}\div(\n u) dx.
\ea\ee
Using (\ref{pp01}), (\ref{pyzgj}), (\ref{pp45}), and H\"older's inequality, we get
\be\la{pp49}\ba
& - \int \div(\n u) (u - m_0) \cdot \na (-\Delta)^{-1}(\n-\ol{\n}) dx \\
& = \int \n u \cdot \na( (u - m_0) \cdot \na (-\Delta)^{-1}(\n-\ol{\n}) ) dx \\
& \le C \| \n u \|_{L^4} \| \na u \|_{L^2} \| \n - \ol{\n} \|_{L^2} \\
& \le C \left( 1 + \| \na u \|_{L^2} \right) \| \na u \|_{L^2} \| \n - \ol{\n} \|_{L^2} \\
& \le \ep \| \n - \ol{\n} \|_{L^2}^2 + C(\ep) \| \na u \|_{L^2}^2.
\ea\ee
From (\ref{pp01}), (\ref{pyzgj}), (\ref{pp45}), and Young's inequality, we deduce that
\be\la{pp410}\ba
& - \int \n (u - m_0) \cdot \na (-\Delta)^{-1}\div(\n u) dx \\
& = - \int (\n-\ol{\n}) (u - m_0) \cdot \na (-\Delta)^{-1}\div(\n u) dx
- \ol{\n} \int (u - m_0) \cdot \na (-\Delta)^{-1}\div(\n u) dx \\
& \le C \| \n u \|_{L^4} \| \na u \|_{L^2} \| \n - \ol{\n} \|_{L^2}
- \ol{\n} \int (u - m_0) \cdot \na (-\Delta)^{-1}\div( \ol{\n} u) dx \\
& \quad - \ol{\n} \int (u - m_0) \cdot \na (-\Delta)^{-1}\div( (\n-\ol{\n}) u) dx \\
& \le C \left( 1 + \| \na u \|_{L^2} \right) \| \na u \|_{L^2} \| \n - \ol{\n} \|_{L^2}
+ C \| u - m_0 \|_{L^2} \| \na u \|_{L^2} \\
& \le \ep \| \n - \ol{\n} \|_{L^2}^2 + C(\ep) \| \na u \|_{L^2}^2,
\ea\ee
where we have used the following estimate
\be\la{pp411}\ba
& - \ol{\n} \int (u - m_0) \cdot \na (-\Delta)^{-1}\div( (\n-\ol{\n}) u) dx \\
& = \ol{\n} \int \div u (-\Delta)^{-1}\div( (\n-\ol{\n}) u) dx \\
& \le C \| \div u \|_{L^2} \| (-\Delta)^{-1}\div( (\n-\ol{\n}) u) \|_{L^2} \\
& \le C \| \na u \|_{L^2} \| (-\Delta)^{-1}\div( (\n-\ol{\n}) u) \|_{W^{1,\frac{4}{3}}} \\
& \le C \| \na u \|_{L^2} \| (\n-\ol{\n}) u \|_{L^\frac{4}{3}}
\le C \| \na u \|_{L^2} \| \n-\ol{\n} \|_{L^2} \| u \|_{L^4} \\
& \le C \left( 1 + \| \na u \|_{L^2} \right) \| \na u \|_{L^2} \| \n - \ol{\n} \|_{L^2}.
\ea\ee
Moreover, (\ref{pp01}), (\ref{3pp4}), (\ref{pp02}), (\ref{pyzgj}), (\ref{pp45}), and Young's inequality ensure that
\be\la{pp412}\ba
\sum_{i=2}^{5} I_i & \le C \| \n u \|_{L^4} \| \na u \|_{L^2} \| \n - \ol{\n} \|_{L^2}
+ C \| \na u \|_{L^2} \| \n - \ol{\n} \|_{L^2} \\
& \quad + C \int \left( (\n - M_1)^\beta_+ + M_1^\beta \right) |\na u| |\n-\ol{\n}| dx
+ C (1 + \| \na \phi \|_{L^4}) \| \na^2 \phi \|_{L^2} \| \n - \ol{\n} \|_{L^2} \\
& \le \ep \int (\n+1)^{\ga-1} (\n-\ol{\n})^2 dx
+ C(\ep) \int (\n -M_1)^{\frac{2\beta(\ga+1)}{\ga-1}}_+ dx \\
& \quad + \ep \| \n - \ol{\n} \|_{L^2}^2 + C(\ep) \left( \| \na u \|^2_{L^2} + \|\na^2 \phi\|^2_{L^2} \right).
\ea\ee
The combination of (\ref{pp47}), (\ref{pp48}), (\ref{pp49}), (\ref{pp410}), and (\ref{pp412}) gives
\be\la{pp413}\ba
\int (\n^\ga - \ol{\n}^{\ga}) (\n-\ol{\n}) dx
\le & - \frac{d}{dt} \int \n (u - m_0) \cdot \na (-\Delta)^{-1}(\n-\ol{\n}) dx \\
& + 4 \ep \int (\n+1)^{\ga-1} (\n-\ol{\n})^2 dx
+ C(\ep) \int (\n -M_1)^{\frac{2\beta(\ga+1)}{\ga-1}}_+ dx \\
& + C(\ep) \left( \| \na u \|^2_{L^2} + \|\na^2 \phi\|^2_{L^2} \right).
\ea\ee
In addition, by (\ref{pp01}) and (\ref{pyzgj}), we have
\be\la{pp414}\ba
\int (\n+1)^{\ga-1} (\n-\ol{\n})^2 dx \le C \int (\n^\ga - \ol{\n}^\ga) (\n-\ol{\n}) dx,
\ea\ee
and
\be\la{pp415}\ba
& \left| \int \n (u - m_0) \cdot \na (-\Delta)^{-1}(\n-\ol{\n}) dx \right| \\
& \le C \| \sqrt{\n} \|_{L^4} \| \sqrt{\n} (u-m_0) \|_{L^2} \| \n-\ol{\n} \|_{L^2}
\le C.
\ea\ee

Substituting (\ref{pp414}) into (\ref{pp413}), choosing $\ep$ sufficiently small, integrating over $(0,T)$, and using (\ref{pp01}), (\ref{pyzgj}), (\ref{pp414}), and (\ref{pp415}), we arrive at (\ref{pp04}).
This completes the proof of Lemma \ref{ppl4}.
\end{proof}

\begin{lemma}\la{ppl5}
For any $\ep \in (0,1)$, there exists a positive constant $C$ depending only on
$\ep$, $\mu$, $\ga$, $\beta$, $\ol{\n}_0$, $\| \n_0 \|_{L^\infty}$, $\| u_0 \|_{H^1}$, and $\| \na d_0 \|_{H^1}$ such that
\be\la{pp05}\ba
\sup_{0 \le t \le T} \log(e+A^2_1) + \int_0^T \frac{A^2_2}{e+A^2_1} dt
\le C R^{1+\ep}_T.
\ea\ee
\end{lemma}
\begin{proof}
First, we use (\ref{gw}) to rewrite $(\ref{zqbhfc})_2$ as
\be\la{pp51}\ba
\n\dot{u} = \na G + \mu\na^{\bot}\o - (\na \phi + \mathbf{k}) \Delta \phi,
\ea\ee
which implies that $G$ and $\o$ satisfy
\be\ba\la{pp52}
\Delta G = \div(\n\dot{u} + (\na \phi + \mathbf{k}) \Delta \phi), \quad \mu \Delta \o = \na^{\bot} \cdot(\n \dot{u} + (\na \phi + \mathbf{k}) \Delta \phi).
\ea\ee
Applying the standard elliptic estimates, we deduce that for any $1<p<\infty$ and any integer $k \ge 1$, there exists a positive constant $C$ depending only on $k$, $p$, and $\mu$ such that
\be\la{pp53}
\|\na^k G\|_{L^p}+\|\na^k \o\|_{L^p} \leq C \left( \| \na ^{k-1}(\n \dot{u}) \|_{L^p} + \| \na ^{k-1}( (\na \phi + \mathbf{k}) \Delta \phi ) \|_{L^p} \right).
\ee
Multiplying (\ref{pp51}) by $2 \dot{u}$ and integrating by parts over $\T^2$, we derive
\be\la{pp54}\ba
& \frac{d}{dt} \int \left(\mu \o^2 + \frac{G^2}{2\mu + \lam}\right)dx
+ 2 \int \n |\dot{u}|^2 dx \\
& = - \mu \int \o^2 \div u dx + 4 \int G\nabla u^1 \cdot\nabla^{\perp}u^2 dx
-2 \int G (\div u)^2 dx \\
& \quad - \int \frac{ (\beta-1)\lam - 2\mu }{(2\mu + \lam)^2} G^2\div u dx
-2\beta \int \frac{ \lam (P-P(\ol{\n})) }{ (2\mu +\lam)^2 } G \div u dx \\
& \quad + 2\ga \int \frac{P}{2\mu +\lam} G \div u dx -2 \int \dot{u} \cdot (\na \phi + \mathbf{k}) \Delta \phi dx
= \sum_{i=1}^7 I_i,
\ea\ee
where we have used the identities
\be\ba\la{bp701}
\na^{\bot}\cdot \dot u= \frac{D}{Dt}\o - (\p_1 u\cdot\na) u^2
+ (\p_2 u \cdot \na) u^1
 = \frac{D}{Dt}\o + \o \div u,
\ea\ee
and
\be\ba\la{bp7001}
\div \dot u&=\frac{D}{Dt}\div u +(\p_1u\cdot\na) u^1
+(\p_2u\cdot\na)u^2 \\ 
& = \frac{D}{Dt} \left( \frac{G}{2\mu+\lam} \right)
+ \frac{D}{Dt} \left( \frac{P-P(\ol{\n})}{2\mu+\lam} \right)
- 2\nabla u^1\cdot\nabla^{\perp}u^2 + (\div u)^2.
\ea\ee
Using (\ref{pp01}) and (\ref{pyzgj}), and adapting the arguments in \cite[Proposition 3.6]{FLW}, we can obtain
\be\la{pp55}\ba
\sum_{i=1}^6 I_i \le \frac{1}{8} A^2_2
+ C R^{1+\ep}_T (1+A^2_1) A^2_1.
\ea\ee
It remains to estimate $I_7$.
Integrating by parts over $\T^2$ and using H\"older's inequality, we derive
\be\la{pp56}\ba
& -2 \int u_t \cdot (\na \phi + \mathbf{k}) \Delta \phi dx \\
& = 2 \int \p_j u^i_t (\p_i \phi + \mathbf{k}_i) (\p_j \phi + \mathbf{k}_j) dx
+ 2 \int u^i_t \p_i \p_j \phi (\p_j \phi + \mathbf{k}_j) dx \\
& = 2 \int ( (\na \phi + \mathbf{k}) \otimes (\na \phi + \mathbf{k}) ) : \na u_t dx
+ \int u_t \cdot \na ( |\na \phi + \mathbf{k}|^2 ) dx \\
& = 2 \frac{d}{dt} \int ( (\na \phi + \mathbf{k}) \otimes (\na \phi + \mathbf{k}) ) : \na u dx
- 2 \int ( \na \phi_t \otimes (\na \phi + \mathbf{k}) ) : \na u dx \\
& \quad - 2 \int ( (\na \phi + \mathbf{k}) \otimes \na \phi_t ):\na u dx
- \int \div u_t |\na \phi + \mathbf{k}|^2 dx \\
& \le \frac{d}{dt} \int \left( 2 ( (\na \phi + \mathbf{k}) \otimes (\na \phi + \mathbf{k}) ):\na u - |\na \phi + \mathbf{k}|^2 \div u \right) dx \\
& \quad + C \| \na \phi_t \|_{L^2} \| \na \phi+\mathbf{k} \|_{L^\infty} \| \na u \|_{L^2}.
\ea\ee
It follows from (\ref{pp56}), (\ref{gn11}), and Young's inequality that
\be\la{pp57}\ba
I_7 & = - 2 \int u_t \cdot (\na \phi + \mathbf{k}) \Delta \phi dx
- 2 \int u \cdot \na u \cdot (\na \phi + \mathbf{k}) \Delta \phi dx \\
& \le \frac{d}{dt} \int \left( 2 ( (\na \phi + \mathbf{k}) \otimes (\na \phi + \mathbf{k}) ):\na u - |\na \phi + \mathbf{k}|^2 \div u \right) dx \\
& \quad + C \| \na \phi_t \|_{L^2} \| \na \phi+\mathbf{k} \|_{L^\infty} \| \na u \|_{L^2}
+ C \| u \|_{L^8} \| \na u \|_{L^2} \| \na \phi + \mathbf{k} \|_{L^8} \| \Delta \phi \|_{L^4} \\
& \le \frac{d}{dt} \int \left( 2 ( (\na \phi + \mathbf{k}) \otimes (\na \phi + \mathbf{k}) ):\na u - |\na \phi + \mathbf{k}|^2 \div u \right) dx \\
& \quad + C \| \na \phi_t \|_{L^2} \| \na \phi \|^{\frac{1}{2}}_{L^2} \| \na \Delta \phi \|^{\frac{1}{2}}_{L^2} \| \na u \|_{L^2}
+ C \| \na \phi_t \|_{L^2} \| \na u \|_{L^2} \\
& \quad + C (1+\| \na u \|_{L^2}) \| \na u \|_{L^2} ( \| \na \phi \|_{L^8} + 1) \| \Delta \phi \|^{\frac{1}{2}}_{L^2} \| \na \Delta \phi \|^{\frac{1}{2}}_{L^2} \\
& \le \frac{d}{dt} \int \left( 2 ( (\na \phi + \mathbf{k}) \otimes (\na \phi + \mathbf{k}) ):\na u - |\na \phi + \mathbf{k}|^2 \div u \right) dx \\
& \quad + C A_1 A_2^{\frac{3}{2}} + C A_1 A_2 + C (1+A_1) A_1^{\frac{3}{2}} A_2^{\frac{1}{2}} \\
& \le \frac{d}{dt} \int \left( 2 ( (\na \phi + \mathbf{k}) \otimes (\na \phi + \mathbf{k}) ):\na u - |\na \phi + \mathbf{k}|^2 \div u \right) dx \\
& \quad + \frac{1}{4} A^2_2 + C (1+A^2_1) A^2_1.
\ea\ee
Furthermore, applying $\na$ to $\eqref{zqbhfc}_3$ gives
\be\la{pp58}\ba
\na \phi_t - \na \Delta \phi = - \na \left( u \cdot (\na \phi + \mathbf{k}) \right).
\ea\ee
Taking the squared $L^2$ norm of both sides of (\ref{pp58}) and using (\ref{gn11}) and Young's inequality, we get
\be\la{pp59}\ba
& \frac{d}{dt} \| \Delta \phi \|^2_{L^2} + \| \na \phi_t \|^2_{L^2} + \| \na \Delta \phi \|^2_{L^2} \\
& \le C \int \left( |\na u|^2 |\na \phi + \mathbf{k}|^2 + |u|^2 |\na^2 \phi|^2 \right) dx \\
& \le C \| \na u \|^2_{L^2} \| \na \phi + \mathbf{k} \|^2_{L^\infty}
+ C \| u \|^2_{L^8} \| \na^2 \phi \|^2_{L^\frac{8}{3}} \\
& \le C \| \na u \|_{L^2}^2 ( \| \na \phi \|_{L^2} \| \na \Delta \phi \|_{L^2} + 1 )
+ C (1 + \| \na u \|^2_{L^2}) \| \na \phi\|_{L^4} \left( \| \na^2 \phi\|_{L^2} + \| \na^3 \phi\|_{L^2} \right) \\
& \le C A_1^2 (1 + A_2) + C A_1 (A_1+A_2) + C A_1^2 (A_1+A_2) \\
& \le \frac{1}{8} A^2_2 + C (1+A^2_1) A^2_1.
\ea\ee
Putting (\ref{pp55}) and (\ref{pp57}) into (\ref{pp54}), and adding the resulting inequality to (\ref{pp59}), we arrive at
\be\la{pp510}\ba
\frac{d}{dt} A_3 + \frac{1}{2} A^2_2
\le C R^{1+\ep}_T (1+A^2_1) A^2_1,
\ea\ee
where
\be\ba\nonumber
A_3 \triangleq \int \left( \mu \o^2 + \frac{G^2}{2\mu + \lam} + |\Delta \phi|^2
- 2 ( (\na \phi + \mathbf{k}) \otimes (\na \phi + \mathbf{k}) ):\na u + |\na \phi + \mathbf{k}|^2 \div u \right) dx.
\ea\ee
In addition, by (\ref{pp01}), (\ref{3pp4}), (\ref{pp02}), and (\ref{pyzgj}), there exists a constant $\hat{C}_0>e$ such that
\be\la{pp511}\ba
\frac{1}{C} (\hat{C}_0+A_3) \le \hat{C}_0 + A^2_1 \le C(\hat{C}_0+A_3).
\ea\ee
Dividing (\ref{pp510}) by $\hat{C}_0+A_3$ and using (\ref{pp511}), we obtain
\be\la{pp512}\ba
\frac{d}{dt} \log (\hat{C}_0+A_3) + \frac{1}{2} \frac{A^2_2}{\hat{C}_0+A_3}
\le C R^{1+\ep}_T A^2_1.
\ea\ee
Integrating (\ref{pp512}) over $(0,T)$ and using (\ref{pp01}) and (\ref{pp04}), we arrive at (\ref{pp05}) and complete the proof of Lemma \ref{ppl5}.
\end{proof}

\begin{lemma}\la{ppl6}
For any $p>4$, there exists a positive constant $C$ depending only on
$p$, $\mu$, $\ga$, $\beta$, $\ol{\n}_0$, $\| \n_0 \|_{L^\infty}$, $\| u_0 \|_{H^1}$, and $\| \na d_0 \|_{H^1}$ such that
\be\la{pp06}\ba
\| \n |u - m_0| \|_{L^p} \le C R_T^{1+\beta/4} A_1^{1-2/p}.
\ea\ee
\end{lemma}
\begin{proof}
First, set
\be\la{pp61}\ba
\nu \triangleq \min\left\{ \left( \frac{\mu}{\mu+1} \right)^{1/2} R_T^{-\beta/2}, \frac{\ga-1}{2(\ga+1)} \right\} \in (0,1].
\ea\ee
We rewrite the momentum equation $(\ref{zqbhfc})_2$ as
\be\la{pp62}\ba
\n (u - m_0)_t + \n u \cdot \na (u - m_0) - \mu \Delta u - \na( (\mu+\lam) \div u ) + \na P = - (\na \phi + \mathbf{k}) \Delta \phi.
\ea\ee
Multiplying (\ref{pp62}) by $|u-m_0|^\nu (u-m_0)$ and integrating by parts over $\T^2$, we derive
\be\la{pp63}\ba 
& \frac{1}{(2+\nu)} \frac{d}{dt} \int \n |u-m_0|^{2+\nu} dx
+ \int |u-m_0|^\nu \left(\mu |\na u|^2 + (\mu+\lam) (\div u)^2 \right) dx \\
& \le \nu \int (\mu+\lam) |\div u| |u-m_0|^\nu |\na u| dx
+ C \int |\n^\ga - \ol{\n}^\ga| |u-m_0|^\nu |\na u|dx \\
& \quad - \int |u-m_0|^\nu (u-m_0) \cdot (\na \phi + \mathbf{k}) \Delta \phi dx \\
& \triangleq I_1+I_2+I_3.
\ea\ee
For $I_1$, Young's inequality together with (\ref{pp61}) gives
\be\la{pp64}\ba
I_1 & \le \frac{1}{2} \int (\mu+\lam) |u-m_0|^\nu (\div u)^2 dx
+ \frac{\nu^2}{2} \int (\mu+\lam) |u-m_0|^\nu |\na u|^2 dx \\
& \le \frac{1}{2} \int (\mu+\lam) |u-m_0|^\nu (\div u)^2 dx
+ \frac{\mu}{2} \int |u-m_0|^{\nu} |\na u|^2 dx.
\ea\ee
Since $\nu < \frac{\ga-1}{\ga+1}$, we have $\frac{1-\nu}{2}-\frac{1}{\ga+1} \in (0,1)$ and
\be\ba\nonumber
|\n-\ol{\n}|^{\frac{2}{1-\nu}} \le C(\n+1)^{\frac{2\nu}{1-\nu}} (\n-\ol{\n})^2
\le C(\n+1)^{\ga-1} (\n-\ol{\n})^2.
\ea\ee
Thus, for $s$ satisfying $\frac{1}{s}=\frac{1-\nu}{2}-\frac{1}{\ga+1}$, we use (\ref{pp45}) and Young's inequality to obtain
\be\la{pp65}\ba
I_2 & \le C \int (\n^{\ga-1}+1) |\n-\ol{\n}| |u-m_0|^\nu |\na u| dx \\
& \le C \int \left( (\n-M_1)_{+}^{\ga-1} + 1 \right) |\n-\ol{\n}|
|u-m_0|^\nu |\na u| dx \\
& \le C \left( \int (\n-M_1)_{+}^{s(\ga-1)} dx
+ \int \left( |\n-\ol{\n}|^{\ga+1} + |\n-\ol{\n}|^{\frac{2}{1-\nu}} \right) dx
+ \int |\na u|^2 dx \right) \\
& \le C \int (\n-M_1)_{+}^{s(\ga-1)} dx + C A^2_1.
\ea\ee
For $I_3$, integrating by parts over $\T^2$ and using (\ref{pp02}) and Poincar\'e's inequality, we get
\be\la{pp66}\ba
I_3 & = \int \p_j (|u-m_0|^{\nu} (u-m_0)^i) (\p_i \phi+\mathbf{k}_i) \p_j \phi dx
+ \int |u-m_0|^{\nu} (u-m_0)^i \p_i \p_j \phi \p_j \phi dx \\
& \le C \int |u-m_0|^\nu |\na u| |\na \phi|(1+|\na \phi|) dx
\le C \int (1 + |u|) |\na u| |\na \phi|(1+|\na \phi|) dx \\
& \le C (1+\| u \|_{L^4}) \| \na u \|_{L^2} \| \na \phi \|_{L^8} (1+\| \na \phi \|_{L^8}) \\
& \le C (1+\| \na u \|_{L^2}) \| \na u \|_{L^2} \| \na \phi \|_{L^8} \\
& \le C \| \na u \|^2_{L^2} + C \| \na^2 \phi \|^2_{L^2}.
\ea\ee
Substituting (\ref{pp64})--(\ref{pp66}) into (\ref{pp63}) leads to
\be\la{pp67}\ba
\frac{d}{dt} \int \n |u-m_0|^{2+\nu} dx
\le C \int (\n-M_1)_{+}^{s(\ga-1)} dx + C A^2_1.
\ea\ee
Integrating (\ref{pp67}) over $(0,T)$ and using (\ref{pp01}), (\ref{pyzgj}), and (\ref{pp04}), we obtain
\be\la{pp68}\ba
\sup_{0 \le t \le T} \int \n|u - m_0|^{2+\nu} dx \le C.
\ea\ee
For $p>4$, we choose $r=(p-2)(2+\nu)/\nu$ and use (\ref{gn11}), (\ref{pp45}), (\ref{pp61}), (\ref{pp68}), and H\"older's inequality to derive
\be\nonumber\ba
\| \n |u - m_0| \|_{L^p} & \le \| \n |u - m_0| \|_{L^{2+\nu}}^{2/p} \| \n |u - m_0| \|_{L^r}^{1-2/p} \\
& \le C R_T^{ (1+\nu)/p} \|\n^{1/(2+\nu)} |u - m_0|\|_{L^{2+\nu}}^{2/p}
\left( r^{1/2} R_T \|\na u\|_{L^2} \right)^{1-2/p} \\
& \le C R_T^{ (1+\nu)/p} \left( R_T^{1+\beta/4} \|\na u\|_{L^2} \right)^{1-2/p} \\
& \le C R_T^{1+\beta/4} A_1^{1-2/p},
\ea\ee
which gives (\ref{pp06}) and completes the proof of Lemma \ref{ppl6}.
\end{proof}

To obtain a time-uniform upper bound for the density, we need the following $L^\infty$ estimate for the commutator $F_1$ defined by
\be\la{f1}\ba
F_1 \triangleq \sum_{i,j=1}^{2} [u^i,R_iR_j](\n u^j).
\ea\ee

\begin{lemma}\la{ppl7}
For any $\ep \in (0,1)$,
there exists a positive constant $C$ depending only on
$\ep$, $\mu$, $\ga$, $\beta$, $\ol{\n}_0$, $\| \n_0 \|_{L^\infty}$, $\| u_0 \|_{H^1}$, and $\| \na d_0 \|_{H^1}$ such that
\be\la{pp07}\ba
\| F_1 \|_{L^\infty} \le C(\ep) \frac{A_2^2}{e+A_1^2} + C(\ep) R_T^{1 +\beta/4+\ep} A_1^2 + C R_T^\ep.
\ea\ee
\end{lemma}
\begin{proof}
First, it follows from (\ref{gw}), (\ref{gn11}), (\ref{pyzgj}), and (\ref{pp53}) that
\be\la{pp71}\ba
\| \na u \|_{L^4}
& \le C ( \|\div u\|_{L^4} + \|\omega\|_{L^4} ) \\
& \le C \left\| \frac{G}{2\mu+\lambda} \right\|_{L^4}
+ C \left\| \frac{P- P(\ol{\n})}{2\mu+\lambda} \right\|_{L^4} + C \| \o \|_{L^4} \\
& \le C \left\| G \right\|_{L^4} + C + C \| \o \|_{L^4} \\
& \le C \| G \|^{1/2}_{L^2} \| G \|^{1/2}_{H^1}
+ C + C \| \o \|^{1/2}_{L^2} \| \na \o \|^{1/2}_{L^2} \\
& \le C R^{\beta/4}_T A_1^{1/2}
\left( A_1 + R_T^{1/2} A_2 + |\ol{G}| \right)^{1/2} + C \\
& \le C R^{\beta/4}_T A_1^{1/2}
\left( A_1 + R_T^{1/2} A_2 \right)^{1/2} + C,
\ea\ee
where in the last inequality we have used the following estimate
\be\nonumber\ba
\left| \int G dx \right|
= \left| \int \left( \lam(\n) \div u - (\n^\ga - \ol{\n}^\ga) \right) dx \right|
\le C A_1,
\ea\ee
due to (\ref{a1}), (\ref{pyzgj}), and H\"older's inequality.

In addition, using integration by parts together with (\ref{f1}), (\ref{a1}), (\ref{pp01}), (\ref{pyzgj}), and Poincar\'e's inequality, we arrive at
\be\la{pp710}\ba
\ol{F_1} & = \int F_1 dx = \int u \cdot \na (-\Delta)^{-1} \div(\n u) dx \\
& = - \int \div u (-\Delta)^{-1} \div(\n u) dx
\le C \| \na u \|_{L^2} \| \n u \|_{L^2} \\
& \le C \| \na u \|_{L^2} \| \n \|_{L^4} \| u \|_{L^4}
\le C \| \na u \|_{L^2} \left( 1 + \| \na u \|_{L^2} \right) \\
& \le C (1 + A_1^2).
\ea\ee
By (\ref{jhz1}), (\ref{jhz2}), (\ref{pyzgj}), (\ref{pp06}), (\ref{pp71}), (\ref{pp710}), and Young's inequality, we derive for $p>4$,
\be\nonumber\ba
\| F_1 - \ol{F_1} \|_{L^\infty}
& \le C(p) \| F_1 - \ol{F_1} \|_{L^p}^{1-4/p} \| \na F_1 \|_{L^{4p/(p+4)}}^{4/p} \\
& \le C(p) \left( \|\na u\|_{L^2} \| \n u \|_{L^p} + \|\na u\|_{L^2} \| \n u \|_{L^2} \right)^{1-4/p}
\left(\|\na u\|_{L^4} \| \n u \|_{L^p} \right)^{4/p} \\
& \le C(p) \|\na u\|_{L^2}^{1-4/p} \|\na u\|_{L^4}^{4/p} \| \n u \|_{L^p} \\
& \le C(p) \|\na u\|_{L^2}^{1-4/p} \|\na u\|_{L^4}^{4/p} \left( \| \n |u - m_0| \|_{L^p} + \|\n m_0\|_{L^p} \right) \\
& \le C(p) R_T^{1 +\beta/4+(1+\beta)/p} A_1^{2-4/p} (A_1+A_2)^{2/p}
+ C(p) R_T^{1+\beta/4} A_1^{2-6/p} \\
& \quad + C(p) R_T^{(1+\beta)/p} A_1^{1-2/p} (A_1+A_2)^{2/p} + C A_1^{1-4/p} \\
& \le C(p) R_T^{1 +\beta/4+(1+\beta)/p} ( A_1^{2-2/p} + A_1^{2-6/p} )
+ C(p) R_T^{1 +\beta/4+(1+\beta)/p} A_1^{2-4/p} A_2^{2/p} \\
& \quad + C(p) R_T^{(1+\beta)/p} A_1 + C(p) R_T^{(1+\beta)/p} A_1^{1-2/p} A_2^{2/p} + C A_1^{1-4/p} \\
& \le C(p) R_T^{1 +\beta/4+(1+\beta)/p} ( A_1^{2-2/p} + A_1^{2-6/p} ) + C(p) R_T^{(1+\beta)/p} A_1 + C A_1^{1-4/p} \\
& \quad + C(p) R_T^{1 +\beta/4+(1+\beta)/p} A_1^{2-4/p} \left( e+A_1^2 \right)^{1/p} \left( \frac{A_2^2}{e+A_1^2} \right)^{1/p} \\
& \quad + C(p) R_T^{(1+\beta)/p} (e+A_1) \left( \frac{A_2^2}{e+A_1^2} \right)^{1/p} \\
& \le C(p) \frac{A_2^2}{e+A_1^2} + C(p) R_T^{p (1+\beta/4+(1+\beta)/p)/(p-3)} A_1^2 + C(p) R_T^{2(1+\beta)/(p-2)},
\ea\ee
which together with (\ref{pp710}) yields for $p>4$,
\be\la{pp72}\ba
\| F_1 \|_{L^\infty}
\le C(p) \frac{A_2^2}{e+A_1^2} + C(p) R_T^{p (1+\beta/4+(1+\beta)/p)/(p-3)} A_1^2 + C(p) R_T^{2(1+\beta)/(p-2)}.
\ea\ee
Choosing $p>4$ sufficiently large in (\ref{pp72}), we obtain (\ref{pp07}) and complete the proof of Lemma \ref{ppl7}.
\end{proof}

\begin{lemma}\la{ppl8}
There exists a positive constant $C$ depending only on
$\mu$, $\ga$, $\beta$, $\ol{\n}_0$, $\| \n_0 \|_{L^\infty}$, $\| u_0 \|_{H^1}$, and $\| \na d_0 \|_{H^1}$ such that
\be\ba\la{pp08}
& \sup_{0\leq t\leq T} \left( \|\n\|_{L^\infty} + \| \na u \|_{L^2} + \| \na \phi \|_{H^1} \right) \\
& + \int_0^T \left( \| \na u \|_{L^2}^2 + \| \sqrt{\n} \dot{u} \|^2_{L^2} + \| \na^2 \phi \|^2_{H^1} + \| \na \phi_t \|^2_{L^2} \right) dt
\le C.
\ea\ee
\end{lemma}
\begin{proof}
First, we deduce from $(\ref{zqbhfc})_1$ and (\ref{gw}) that
\be\la{pp81}\ba
\frac{D}{Dt} \Lambda(\n) + (P-P(\ol{\n})) = -G,
\ea\ee
where
\be\la{pp82}\ba
\Lambda(\n) \triangleq 2 \mu \log \n + \beta^{-1} \n^\beta.
\ea\ee
Moreover, recall from (\ref{pp52}) that $G$ satisfies
\be\la{pp83}\ba
\Delta G = \div(\n \dot{u} + (\na \phi + \mathbf{k}) \Delta \phi),
\ea\ee
which implies
\be\la{pp84}\ba
G - \ol{G} + \frac{D}{Dt} \left( (- \Delta)^{-1} \div(\n u) \right)
= F_1 + F_2,
\ea\ee
where
\be\la{pp85}\ba
F_1 \triangleq \sum_{i,j=1}^{2} [u^i,R_iR_j](\n u^j), \quad
F_2 \triangleq - (- \Delta)^{-1} \div((\na \phi + \mathbf{k}) \Delta \phi).
\ea\ee
The combination of (\ref{pp81}) and (\ref{pp84}) gives
\be\la{pp86}\ba
\frac{D}{Dt} (\Lambda(\n)-\psi) + P - P(\ol{\n}) = - \ol{G} - F_1 - F_2,
\ea\ee
where
\be\la{pp87}\ba
\psi \triangleq (- \Delta)^{-1} \div(\n u).
\ea\ee
Since the function $y=\Lambda(\n)$ is strictly increasing on $(0,\infty)$ and maps $(0,\infty)$ onto $(-\infty,\infty)$, its inverse function $\n=\Lambda^{-1}(y)$ is well defined for all $y\in(-\infty,\infty)$.
Along any particle trajectory, we then rewrite (\ref{pp86}) as
\be\ba\nonumber
y'(t) = g(y) + h'(t),
\ea\ee
with
\be\la{pp88}\ba
y=\Lambda(\n), \quad g(y)=-P( \Lambda^{-1}(y) ), \quad h= \psi + \int_0^t \left( P(\ol{\n}) - \ol{G} - F_1 - F_2 \right) ds.
\ea\ee
Clearly, $g(\infty)=-\infty$.
Moreover, it follows from (\ref{gw}), (\ref{pyzgj}), and H\"older's inequality that
\be\la{pp89}\ba
|\ol{G}| & = \left| \int \left( \n^\beta \div u - (P-P(\ol{\n})) \right) dx \right| \\
& \le C \| \n^\beta \|_{L^2} \| \na u \|_{L^2} + C
\le C (1 + \| \na u \|_{L^2}^2).
\ea\ee
Using (\ref{bwi}), (\ref{pp01}), (\ref{pyzgj}), and (\ref{pp05}), we arrive at
\be\la{pp810}\ba
\| \psi \|_{L^\infty}
& \le C \| \na \psi \|_{L^2} \log^{1/2}(e+\| \na \psi \|_{L^3}) + C \| \psi \|_{L^2} + C \\
& \le C \| \n u \|_{L^2} \log^{1/2}(e+\| \n u \|_{L^3}) + C \| \psi \|_{W^{1,2\ga/(\ga+1)}} + C \\
& \le C R_T^{1/2} \| \sqrt{\n} u \|_{L^2} \log^{1/2}\left( e + R_T (1+\| \na u \|_{L^2}) \right) + C \| \n u \|_{L^{2\ga/(\ga+1)}} + C \\
& \le C R_T^{1/2} \log^{1/2}\left( e + R_T (1+\| \na u \|_{L^2}) \right) + C \| \sqrt{\n} \|_{L^{2\ga}} \| \sqrt{\n} u \|_{L^2} + C \\
& \le C R_T^{1/2} \log^{1/2}\left( e + \| \na u \|_{L^2}^2 \right) + C R_T \\
& \le C R_T^{1+\ep/2}.
\ea\ee
In view of (\ref{gn11}), (\ref{3pp4}), (\ref{pp02}), (\ref{pp41}), the standard elliptic estimates, and Young's inequality, it holds that
\be\la{pp811}\ba
\| F_2 \|_{L^\infty}
& \le C \| F_2 \|^{1/2}_{L^4} \| \na F_2 \|^{1/2}_{L^4}
\le C \| |\na \phi+\mathbf{k}|^2 \|^{1/2}_{L^4} \| (\na \phi + \mathbf{k}) \Delta \phi \|^{1/2}_{L^4} \\
& \le C \| \na \phi+\mathbf{k} \|_{L^8} \| \na \phi+\mathbf{k} \|^{1/2}_{L^8} \| \Delta \phi \|^{1/2}_{L^8}
\le C \left( \| \Delta \phi \|_{L^2} + \| \na \Delta \phi \|_{L^2} \right)^{1/2} \\
& \le C A_1^{1/2} + C A_2^{1/2}
\le C A_1^{1/2} + C ( e+A_1^{1/2} ) \left( \frac{A_2^2}{e+A_1^2} \right)^{1/4} \\
& \le C + C A^2_1 + C \frac{A_2^2}{e+A_1^2}.
\ea\ee
Combining (\ref{pp07}), (\ref{pp811}), (\ref{pp01}), (\ref{pp04}), and (\ref{pp05}) yields, for any $0 \le t_1 < t_2 \le T$,
\be\la{pp812}\ba
\int_{t_1}^{t_2} \left( \| F_1 \|_{L^\infty} + \| F_2 \|_{L^\infty} \right) dt
\le C R_T^{1+\frac{\beta}{4}+\ep} + C R_T^\ep (t_2-t_1).
\ea\ee
By (\ref{pp88}), (\ref{pp89}), (\ref{pp810}), (\ref{pp812}), and (\ref{pp01}), we have for any $0 \le t_1 < t_2 \le T$,
\be\ba\nonumber
h(t_2)-h(t_1) \le C R^{1+\frac{\beta}{4}+\ep}_T + C R_T (t_2-t_1).
\ea\ee
Choose $N_0$, $N_1$, and $\overline{\zeta}$ in Lemma \ref{zli} as follows:
\be\la{pp813}\ba
N_0 = C R^{1+\frac{\beta}{4}+\ep}_T, \quad N_1 = C R_T,\quad
\overline{\zeta} = \Lambda \left( \left( C R_T \right)^{1/\ga} \right),
\ea\ee
which together with (\ref{pp88}) gives
\be\ba\nonumber
g(\zeta)=-( \Lambda^{-1}(\zeta) )^\ga \le - N_1 = - C R_T
\quad \text{ for all } \zeta \ge \overline{\zeta}.
\ea\ee
In addition, $R_T \ge 1$ implies $\overline{\zeta} \le C R^{\frac{\beta}{\ga}}_T$.
From (\ref{pp813}) and Lemma \ref{zli}, we conclude that
\be\la{pp814}\ba
R^\beta_T \le C R_T^{ \max\{ 1+\frac{\beta}{4}+\ep,\frac{\beta}{\ga} \} }.
\ea\ee
Since $\beta>4/3$ and $\ga>1$, we can take $0<\ep<(3\beta-4)/4$ and obtain from (\ref{pp814}) that
\be\la{pp815}\ba
\sup_{0 \le t \le T} \| \n \|_{L^\infty} \le C.
\ea\ee
Finally, the combination of (\ref{pp01}), (\ref{pp05}), and (\ref{pp815}) gives (\ref{pp08}), thereby completing the proof of Lemma \ref{ppl8}.
\end{proof}

\subsection{A priori estimates (II): higher order estimates}

The goal of this subsection is to establish the higher-order estimates needed to extend the local strong solution globally in time.

\begin{lemma}\la{ppgl1}
There exists a positive constant $C$ depending only on
$\mu$, $\ga$, $\beta$, $\ol{\n}_0$, $\| \n_0 \|_{L^\infty}$, $\| u_0 \|_{H^1}$, and $\| \na d_0 \|_{H^1}$ such that
\be\ba\la{ppg01}
\sup_{0\le t\le T}
\si \left( \| \sqrt{\n} \dot{u} \|^2_{L^2} + \| \na \phi_t \|^2_{L^2} + \| \na^3 \phi \|^2_{L^2} \right)
+ \int_0^{T} \si \left( \| \na \dot{u} \|^2_{L^2} + \| \na^2 \phi_t \|^2_{L^2} \right) dt \le C,
\ea\ee
with $\si \triangleq \min\{ 1,t \}$.
\end{lemma}
\begin{proof}
First, applying the operator $\p_t + \div(u \cdot)$ to $(\ref{zqbhfc})^j_2$ and using (\ref{pp41}), we obtain
\be\la{ppg11}\ba
& (\n \dot u^j)_t + \div(\n u \dot u^j) - \mu \Delta \dot u^j - \p_j ((\mu+\lm) \div \dot u) \\
& = \mu \p_i ( -\p_i u \cdot \na u^j + \div u \p_i u^j) -\mu \div(\p_i u\p_i u^j) \\
& \quad -\p_j \left[ (\mu+\lambda) \pa_i u \cdot \na u^i - ( \mu+(1-\beta) \lam )(\div u)^2\right] \\
& \quad - \div (\p_j u (\mu+\lambda) \div u) + (\ga-1) \p_j (P\div u) + \text{div} (P\p_ju) \\
& \quad - \p_i ( (\p_j \phi+\mathbf{k}_j) (\p_i \phi+\mathbf{k}_i) )_t + \frac{1}{2} \p_j ( |\na \phi+\mathbf{k}|^2 )_t - \div (u (\p_j \phi + \mathbf{k}_j) \Delta \phi).
\ea\ee

Then, multiplying (\ref{ppg11}) by $\dot{u}^j$, summing over $j$, and integrating by parts over $\T^2$,
we derive after using (\ref{pp08}) and Young's inequality that
\be\la{ppg12}\ba
& \frac{1}{2} \frac{d}{dt} \int \n |\dot u|^2 dx
+ \mu \int |\na \dot u|^2 dx + \int (\mu+\lambda) (\div \dot u)^2 dx \\
& \le \frac{\mu}{4} \| \na \dot{u} \|^2_{L^2}
+ C \left( \|\na u\|_{L^4}^4 + \|\nabla u\|^2_{L^2}
+ \| |\na \phi+\mathbf{k}| |\na \phi_t| \|^2_{L^2} + \| |u| |\na \phi+\mathbf{k}| |\Delta \phi| \|^2_{L^2} \right) \\
& \le \frac{\mu}{4} \| \na \dot{u} \|^2_{L^2}
+ C \left( A^2_1 + A^2_2 + \| \na \phi+\mathbf{k} \|^2_{L^4} \| \na \phi_t \|^2_{L^4} \right) + C \| \na \phi_t \|_{L^2} \| \na^2 \phi_t \|_{L^2} \\
& \quad + C \| u \|^2_{L^8} \| \na \phi+\mathbf{k} \|^2_{L^8} \| \Delta \phi \|^2_{L^{4}} \\
& \le \frac{\mu}{4} \| \na \dot{u} \|^2_{L^2}
+ C \left( A^2_1 + A^2_2 + \| \na \phi_t \|_{L^2} \| \na^2 \phi_t \|_{L^2} \right) + C \| \na \Delta \phi \|^2_{L^{2}} \\
& \le \frac{\mu}{4} \| \na \dot{u} \|^2_{L^2} + \frac{1}{4} \| \na^2 \phi_t \|^2_{L^2} + C (A^2_1 + A^2_2),
\ea\ee
where in the second inequality we have used
\be\la{ppg13}\ba
\| \na u \|^4_{L^4}
& \le C \left( \| \div u \|^4_{L^4} + \| \o \|^4_{L^4} \right) \\
& \le C \left( \| G \|^4_{L^4} + \| P-P(\ol{\n}) \|^4_{L^4} + \| \o \|^4_{L^4} \right) \\
& \le C \left( \| G \|^2_{L^2} \| G \|^2_{H^1} + \| P-P(\ol{\n}) \|^2_{L^2}
+ \| \o \|^2_{L^2} \| \o \|^2_{H^1} \right) \\
& \le C \left( \| G \|^2_{H^1} + \| \o \|^2_{H^1} + A^2_1 \right) \\
& \le C \left( A^2_1 + A^2_2 \right),
\ea\ee
which follows from (\ref{gw}), (\ref{gn11}), and (\ref{pp53}).

On the other hand, differentiating (\ref{pp58}) with respect to $t$ gives
\be\la{ppg14}\ba
\na \phi_{tt} - \na \Delta \phi_t = - \na \left( u \cdot (\na \phi + \mathbf{k}) \right)_t.
\ea\ee
Multiplying (\ref{ppg14}) by $\na \phi_t$, integrating by parts over $\T^2$ and using
(\ref{pp08}) and Young's inequality, we derive
\be\la{ppg15}\ba
& \frac{1}{2}\frac{d}{dt}\|\nabla \phi_{t}\|_{L^{2}}^{2} + \| \na^2 \phi_{t}\|_{L^{2}}^{2} \\
& = \int \left( u_t \cdot (\na \phi + \mathbf{k}) + u \cdot \na \phi_t \right) \cdot \Delta \phi_t dx \\
& = \int \left( \dot{u} \cdot (\na \phi + \mathbf{k}) - u \cdot \na u \cdot (\na \phi + \mathbf{k})
+ u \cdot \na \phi_t \right) \cdot \Delta \phi_t dx \\
& = - \int \left( \p_i \dot{u} \cdot (\na \phi + \mathbf{k}) + \dot{u} \cdot \na \p_i \phi \right) \cdot \p_i \phi_t dx
- \int \left( u \cdot \na u \cdot (\na \phi + \mathbf{k})
- u \cdot \na \phi_t \right) \cdot \Delta \phi_t dx \\
& \le C \left( \| \na \dot{u} \|_{L^2} \| \na \phi+\mathbf{k} \|_{L^4}
+ \| \dot{u} \|_{L^4} \| \na^2 \phi \|_{L^2} \right) \| \na \phi_t \|_{L^4} \\
& \quad + C \left( \| u \|_{L^8} \| \na u \|_{L^4} \| \na \phi \|_{L^8}
+ \| u - m_0 \|_{L^4} \| \na u \|_{L^4} + \| \na u \|_{L^2}
+ \| u \|_{L^4} \| \na \phi_t \|_{L^4} \right) \| \Delta \phi_t \|_{L^2} \\
& \le C \left( \| \na \dot{u} \|_{L^2} + \| \sqrt{\n} \dot{u} \|_{L^2} \right) \| \na \phi_t \|^{\frac{1}{2}}_{L^2} \| \na^2 \phi_t \|^{\frac{1}{2}}_{L^2} \\
& \quad + C \left( \| \na u \|^2_{L^4} + \| \na^2 \phi \|_{L^2}
+ \| \na u \|_{L^2} + \| \na \phi_t \|^{\frac{1}{2}}_{L^2} \| \na^2 \phi_t \|^{\frac{1}{2}}_{L^2} \right) \| \na^2 \phi_t \|_{L^2} \\
& \le \frac{1}{4} \| \na^2 \phi_t \|^2_{L^2} + \frac{\mu}{4} \| \na \dot{u} \|^2_{L^2}
+ C ( A^2_1 + A^2_2 ).
\ea\ee
Combining (\ref{ppg12}) and (\ref{ppg15}) leads to
\be\la{ppg17}\ba
\frac{d}{dt} \left( \| \sqrt{\n} \dot{u} \|^2_{L^2} + \|\nabla \phi_{t}\|_{L^{2}}^{2} \right)
+ \mu \| \na \dot{u} \|^2_{L^2} + \| \na^2 \phi_{t}\|_{L^{2}}^{2}
\le C ( A^2_1 + A^2_2 ).
\ea\ee

Multiplying (\ref{ppg17}) by $\si$, integrating over $(0,T)$, and using (\ref{pp08}), we arrive at
\be\la{ppg18}\ba
\sup_{0\le t\le T}
\si \left( \| \sqrt{\n} \dot{u} \|^2_{L^2} + \| \na \phi_t \|^2_{L^2} \right)
+ \int_0^{T} \si \left( \| \na \dot{u} \|^2_{L^2} + \| \na^2 \phi_t \|^2_{L^2} \right) dt \le C.
\ea\ee
In addition, it follows from (\ref{pp58}), (\ref{pp08}), (\ref{ppg13}), and Young's inequality that
\be\ba\nonumber
\| \nabla \Delta \phi \|_{L^{2}}^{2}
\leq & C \left( \|\nabla \phi_{t}\|_{L^{2}}^{2} + \| |\nabla u| |\na \phi + \mathbf{k}| \|_{L^{2}}^{2}
+ \| |u| |\nabla^{2} \phi| \|_{L^{2}}^{2} \right) \\
\leq & C \left( \|\nabla \phi_{t}\|_{L^{2}}^{2} + \|\nabla u\|_{L^{4}}^{2} \|\nabla \phi\|_{L^{4}}^{2}
+ \| \na u \|_{L^2}^2 + \| u \|_{L^{4}}^{2} \| \nabla^{2} \phi \|^2_{L^{4}} \right) \\
\leq & C \left( \|\nabla \phi_{t}\|_{L^{2}}^{2}
+ A_1 A_2 + A^2_1 + \| \nabla^{2} \phi \|_{L^{4}}^{2} \right) \\
\leq & C \|\nabla \phi_{t}\|_{L^{2}}^{2}
+ C A_1 \left( \| \sqrt{\n} \dot{u} \|_{L^2} + \| \na \phi_t \|_{L^2} + \| \na \Delta \phi \|_{L^2} \right) + C A^2_1 \\
\leq & \frac{1}{2} \| \nabla \Delta \phi \|_{L^{2}}^{2}
+ C \left( A^2_1 + \| \sqrt{\n} \dot{u} \|^2_{L^2} + \|\nabla \phi_{t}\|_{L^{2}}^{2} \right),
\ea\ee
which gives
\be\la{ppg19}\ba
\|\nabla^3 \phi\|_{L^{2}}^{2}
\le C \left( A^2_1 + \| \sqrt{\n} \dot{u} \|^2_{L^2} + \|\nabla \phi_{t}\|_{L^{2}}^{2} \right).
\ea\ee
This, together with (\ref{ppg18}), yields (\ref{ppg01}) and completes the proof of Lemma \ref{ppgl1}.
\end{proof}

We next establish the exponential decay of the strong solution by using the uniform estimates obtained in (\ref{pp01}) and (\ref{pp08}).
\begin{lemma}\la{ppel}
For any $p \in [1,\infty)$, there exist positive constants
$C$ and $\alpha_1$ depending only on
$p$, $\mu$, $\ga$, $\beta$, $\ol{\n}_0$, $\| \n_0 \|_{L^\infty}$, $\| u_0 \|_{H^1}$, and $\| \na d_0 \|_{H^1}$ such that for any $1 \le t <\infty$,
\be\la{ppe01}\ba
& \| \n-\ol{\n_0}\|_{L^p} + \| \na u \|_{L^p}
+ \| \sqrt{\n} \dot{u} \|_{L^2} \\
& + \| \na d - \mathbf{k} \otimes d^\bot \|_{H^2}
+ \| d_t + (m_0 \cdot \mathbf{k}) d^\bot \|_{H^1} \le C e^{-\alpha_1 t}.
\ea\ee
\end{lemma}
\begin{proof}
First, using (\ref{pp08}) and arguing as in the proof of Lemma \ref{ppl4}, we conclude that
\be\la{ppe1}\ba
\int (\n^\ga - \ol{\n}^{\ga}) (\n-\ol{\n}) dx
& \le - 2 \frac{d}{dt} \int \n (u - m_0) \cdot \na (-\Delta)^{-1}(\n-\ol{\n}) dx \\
& \quad + C \left( \| \na u \|^2_{L^2} + \|\na^2 \phi\|^2_{L^2} \right).
\ea\ee
Define
\be\la{ppe2}\ba
H(\n,\ol{\n}) \triangleq \n \int^\n_{\ol{\n}} \frac{P(s)-P(\ol{\n})}{s^2} ds,
\ea\ee
which satisfies
\be\la{ppe3}\ba
\frac{d}{dt} \int H(\n,\ol{\n}) dx = \frac{d}{dt} \int \frac{P}{\ga-1} dx.
\ea\ee
We rewrite the momentum equation $(\ref{zqbhfc})_2$ as
\be\la{reme}\ba
\n (u - m_0)_t + \n u \cdot \na (u - m_0) - \mu \Delta u - \na( (\mu+\lam) \div u ) + \na P = - (\na \phi + \mathbf{k}) \Delta \phi.
\ea\ee
Multiplying (\ref{reme}) and $(\ref{zqbhfc})_3$ by $u-m_0$ and $\Delta \phi$, respectively, integrating by parts over $\T^2$, adding the resulting identities, and using $(\ref{zqbhfc})_1$, (\ref{pp41}), and (\ref{ppe3}), we obtain
\be\la{ppe4}\ba
& \frac{d}{dt} E_1(t)
+ \int \left( \mu |\na u|^2 + (\mu + \lam(\n)) (\div u)^2 + |\Delta \phi|^2 \right) dx = 0,
\ea\ee
where
\be\la{ppe5}\ba
E_1(t) \triangleq \int \left( \frac{1}{2}\rho |u-m_0|^2 + H(\n,\ol{\n}) + \frac{1}{2} |\na \phi|^2 \right) dx.
\ea\ee
Thus, we have
\be\la{ppe6}\ba
& \frac{d}{dt} E_1(t)
+ \tilde{C_1} \left( \| \na u \|^2_{L^2} + \| \na^2 \phi \|^2_{L^2} \right) \le 0.
\ea\ee
On the other hand, by virtue of (\ref{ppe1}) and (\ref{ppe2}), we have
\be\la{ppe7}\ba
\tilde{C_2} \int H(\n,\ol{\n}) dx
\leq - 2 \frac{d}{dt} \int \n (u - m_0) \cdot \na (-\Delta)^{-1}(\n-\ol{\n}) dx
+ \tilde{C_3} \left( \| \na u \|^2_{L^2} + \| \na^2 \phi \|^2_{L^2} \right).
\ea\ee
Moreover, (\ref{pp08}) and Poincar\'e's inequality ensure
\be\la{ppe8}\ba
\left| \int \n (u - m_0) \cdot \na (-\Delta)^{-1}(\n-\ol{\n}) dx \right|
& \le C \| \sqrt{\n} |u-m_0| \|_{L^2} \| \n-\ol{\n} \|_{L^2} \\
& \le \tilde{C_4} \int \left( \frac{1}{2}\rho |u-m_0|^2 + H(\n,\ol{\n}) \right) dx.
\ea\ee
Choosing $\de = \min \left\{\frac{\tilde{C_1}}{2 \tilde{C_3}}, \frac{1}{4 \tilde{C_4}} \right\}$,
multiplying (\ref{ppe7}) by $\de$, and adding the result to (\ref{ppe6}), we derive
\be\la{ppe9}\ba
\frac{d}{dt}E_2(t) + \de \tilde{C_2} \int H(\n,\ol{\n}) dx
+ \frac{\tilde{C_1}}{2} \left( \| \na u \|^2_{L^2} + \| \na^2 \phi \|^2_{L^2} \right) \le 0,
\ea\ee
where
\be\la{ppe10}\ba
E_2(t) \triangleq E_1(t) + 2 \de \int \n (u - m_0) \cdot \na (-\Delta)^{-1}(\n-\ol{\n}) dx,
\ea\ee
which satisfies
\be\la{ppe11}\ba
\frac{1}{2} E_1(t) \le E_2(t) \le 2 E_1(t).
\ea\ee
In addition, by (\ref{pp45}), (\ref{pp08}), and Poincar\'e's inequality, we have
\be\la{ppe12}\ba
\| \sqrt{\n} |u-m_0| \|^2_{L^2} + \| \na \phi \|^2_{L^2}
& \le \tilde{C_5} \left( \| \na u \|^2_{L^2} + \| \Delta \phi \|^2_{L^2} \right).
\ea\ee
Combining this with (\ref{ppe9}) and (\ref{ppe11}), we obtain, for
$\alpha_0= \min \left\{\frac{\de \tilde{C_2}}{4}, \frac{\tilde{C_1}}{4 \tilde{C_5}} \right\}$,
\be\la{ppe13}\ba
\frac{d}{dt}E_2(t) + 2 \alpha_0 E_2(t) \le 0,
\ea\ee
which together with (\ref{ppe11}) and Gr\"onwall's inequality implies that
\be\la{ppe14}\ba
\| \sqrt{\n} |u-m_0| \|^2_{L^2} + \| \n-\ol{\n} \|^2_{L^2} + \| \na \phi \|^2_{L^2}
\le C e^{-2\alpha_0 t}.
\ea\ee
Multiplying (\ref{ppe6}) by $e^{\alpha_0 t}$, integrating over $(0,T)$, and using (\ref{ppe14}), we get
\be\la{ppe15}\ba
\int_0^T e^{\alpha_0 t} \left( \| \na u \|^2_{L^2} + \| \na^2 \phi \|^2_{L^2} \right) dt \le C.
\ea\ee
Furthermore, by (\ref{3pp4}), (\ref{pp08}), (\ref{ppe14}), and Poincar\'e's inequality, we have
\be\la{a31}\ba
& \int \left( \mu \o^2 + \frac{G^2}{2\mu + \lam} + |\Delta \phi|^2
- 2 ( (\na \phi + \mathbf{k}) \otimes (\na \phi + \mathbf{k}) ):\na u + |\na \phi + \mathbf{k}|^2 \div u \right) dx \\
& \le C \left( \| \na u \|_{L^2}^2 + \| \n - \ol{\n} \|_{L^2}^2 + \| \na^2 \phi \|_{L^2}^2 \right) + C \| \na \phi \|_{L^4}^2 \| \na u \|_{L^2} + C \| \na \phi \|_{L^2} \| \na u \|_{L^2} \\
& \le C \| \na u \|_{L^2}^2 + C \| \na^2 \phi \|_{L^2}^2 + C e^{-2\alpha_0 t},
\ea\ee
where we have used the fact
\be\nonumber\ba
\int (\mathbf{k} \otimes \mathbf{k} : \na u) dx = 0.
\ea\ee
In view of (\ref{pp08}), (\ref{ppe14}), (\ref{ppe15}), and (\ref{a31}), we multiply (\ref{pp510}) by $e^{\alpha_0 t}$ to obtain
\be\la{ppe16}\ba
& \sup_{0 \le t \le T} e^{\alpha_0 t} \left( \| \na u \|^2_{L^2} + \| \na^2 \phi \|^2_{L^2} \right) \\
& + \int_0^T e^{\alpha_0 t} \left( \| \sqrt{\n} \dot{u} \|^2_{L^2} + \| \na \phi_t \|^2_{L^2} + \| \na^3 \phi \|^2_{L^2} \right) dt \le C.
\ea\ee

Multiplying (\ref{ppg17}) by $e^{\alpha_0 t}$, integrating over $(1,T)$, and using (\ref{pp08}), (\ref{ppg19}), (\ref{ppe14}), (\ref{ppe15}), and (\ref{ppe16}), we derive that for any $t \ge 1$,
\be\la{ppe17}\ba
\| \sqrt{\n} \dot{u} \|^2_{L^2} + \| \na \phi_t \|^2_{L^2} + \| \na^3 \phi \|^2_{L^2}
\le C e^{-\alpha_0 t}.
\ea\ee
From (\ref{ppe14}), (\ref{ppe16}), and H\"older's inequality, we deduce that for any $p \in [1,2]$,
\be\la{ppe170}\ba
\| \n - \ol{\n} \|_{L^p} + \| \na u \|_{L^p} \le C e^{-\frac{\alpha_0}{2} t}.
\ea\ee
Moreover, H\"older's inequality, together with (\ref{pp815}) and (\ref{ppe14}), gives, for any $2< p<\infty$,
\begin{equation}\label{ppee17}
	\|\rho-\ol{\rho}\|_{L^p}\leq \|\rho-\ol{\rho}\|^{\frac{2}{p}}_{L^2}\|\rho-\ol{\rho}\|^{1-\frac{2}{p}}_{L^\infty}\leq C \|\rho-\ol{\rho}\|^{\frac{2}{p}}_{L^2}.
\end{equation}
It follows from (\ref{gw}), (\ref{pp53}), (\ref{pp08}), (\ref{ppe14}), (\ref{ppe16}), and Poincar\'e's inequality that, for any $2 < p < \infty$ and $t \ge 1$,
\be\la{ppe18}\ba
\| \na u \|_{L^p}
& \le C \left( \| \div u \|_{L^p}+ \| \o \|_{L^p} \right) \\
& \le C \left( \| G \|_{L^p} + \| P - P(\ol{\n}) \|_{L^p} + \| \o \|_{L^p} \right) \\
& \le C \left( \| G \|_{H^1} + \| \n - \ol{\n} \|_{L^p} + \| \na \o \|_{L^2} \right) \\
& \le C \left( \| \na u \|_{L^2} + \| \n - \ol{\n} \|_{L^p} + \| \sqrt{\n} \dot{u} \|_{L^2} + \| \na^3 \phi \|_{L^2} \right).
\ea\ee
Combining (\ref{ppe14}), (\ref{ppe16}), (\ref{ppe17}), (\ref{ppe170}), (\ref{ppee17}), and (\ref{ppe18}), we obtain for any $p \in [1,\infty)$,
\begin{equation}\label{p18}
	\|\rho-\ol{\rho}\|_{L^p}+\|\nabla u\|_{L^p}\leq C e^{-\alpha_1 t},
\end{equation}
where
\be\nonumber\ba
\alpha_1=\min\left\{\frac{2\alpha_0}{p},\frac{\alpha_0}{2}\right\}.
\ea\ee

On the other hand, using (\ref{zqbh8}) and (\ref{zqbh9}), we derive that for $i,j,l=1,2$,
\be\la{ppe19}\ba
\p_i d - k_i d^\bot = d^\bot \p_i \phi,
\ea\ee
\be\la{ppe20}\ba
\p_j(\p_i d - k_i d^\bot) = d^\bot \p_{ij} \phi - d(\p_j \phi + k_j) \p_i \phi,
\ea\ee
and
\be\la{ppe21}\ba
\p_{lj}(\p_i d - k_i d^\bot)
& = d^\bot \p_{ijl} \phi - d(\p_l \phi + k_l) \p_{ij} \phi
- d^\bot (\p_l \phi + k_l) (\p_j \phi + k_j) \p_i \phi \\
& \quad - d \p_{jl} \phi \p_i \phi - d(\p_j \phi + k_j) \p_{il} \phi.
\ea\ee
By (\ref{pp08}), (\ref{ppe19})--(\ref{ppe21}), and Poincar\'e's inequality, we arrive at
\be\la{ppe22}\ba
\| \na d - \mathbf{k} \otimes d^\bot \|^2_{H^1}
\le C \left( \| \na \phi \|^2_{L^2} + \| \na^2 \phi \|^2_{L^2} + \| \na \phi \|^4_{L^4} \right)
\le C \| \na^2 \phi \|^2_{L^2},
\ea\ee
and
\be\la{ppe23}\ba
& \| \na^2( \na d - \mathbf{k} \otimes d^\bot ) \|^2_{L^2} \\
& \le C \left( \| \na^3 \phi \|^2_{L^2} + \| \na^2 \phi \|^2_{L^4} \| \na \phi \|^2_{L^4} + \| \na^2 \phi \|^2_{L^2}
+ \| \na \phi \|^6_{L^6} + \| \na \phi \|^2_{L^2} \right) \\
& \le C \| \na^3 \phi \|^2_{L^2}.
\ea\ee
From (\ref{ppe17}) and Poincar\'e's inequality, we deduce that for any $1 \le t <\infty$,
\be\la{ppe26}\ba
\| \na d - \mathbf{k} \otimes d^\bot \|^2_{H^2} \le C e^{-\alpha_0 t}.
\ea\ee
On the other hand, integrating $(\ref{zqbhfc})_3$ over $\T^2$, we obtain
\be\la{ppe27}\ba
\int \phi_t dx = - \int u \cdot (\na \phi + \mathbf{k}) dx
= - \int (u-m_0) \cdot (\na \phi + \mathbf{k}) dx - m_0 \cdot \mathbf{k},
\ea\ee
which together with (\ref{3pp4}), (\ref{pp45}), (\ref{pp08}), and H\"older's inequality yields
\be\la{ppe28}\ba
\left| \int (\phi_t + m_0 \cdot \mathbf{k}) dx \right|
\le C \| u - m_0 \|_{L^2} (\|\na \phi \|_{L^2}+1)
\le C \| \na u \|_{L^2}.
\ea\ee
The combination of (\ref{ppe28}) and Poincar\'e's inequality gives
\be\la{ppe29}\ba
\| \phi_t + m_0 \cdot \mathbf{k} \|^2_{H^1}
& \le C \left| \int (\phi_t + m_0 \cdot \mathbf{k}) dx \right|^2 + C \| \na \phi_t \|^2_{L^2} \\
& \le C \left( \| \na u \|^2_{L^2} + \| \na \phi_t \|^2_{L^2} \right).
\ea\ee
Using (\ref{zqbh8}) and (\ref{zqbh9}), we get for any $i=1,2$,
\be\la{ppe30}\ba
d_t + (m_0 \cdot \mathbf{k}) d^\bot
= \left( \phi_t + (m_0 \cdot \mathbf{k}) \right) d^\bot,
\ea\ee
and
\be\la{ppe31}\ba
\p_i \left( d_t + (m_0 \cdot \mathbf{k}) d^\bot \right)
= \p_i \phi_t d^\bot - \left( \phi_t + (m_0 \cdot \mathbf{k}) \right) (\p_i \phi + k_i) d.
\ea\ee
It follows from (\ref{ppe29}) that
\be\la{ppe32}\ba
\| d_t + (m_0 \cdot \mathbf{k}) d^\bot \|^2_{H^1}
& \le C \| \phi_t + m_0 \cdot \mathbf{k} \|^2_{L^2}
+ C \| \na \phi_t \|^2_{L^2} \\
& \quad + C \| \phi_t + m_0 \cdot \mathbf{k} \|^2_{L^4} \| \na \phi + \mathbf{k} \|^2_{L^4} \\
& \le C \left( \| \na u \|^2_{L^2} + \| \na \phi_t \|^2_{L^2} \right).
\ea\ee
This, along with (\ref{ppe16}) and (\ref{ppe17}), yields for any $t \ge 1$,
\be\la{ppe34}\ba
\| d_t + (m_0 \cdot \mathbf{k}) d^\bot \|^2_{H^1} \le C e^{-\alpha_0 t}.
\ea\ee
Using (\ref{ppe17}), (\ref{ppe26}), (\ref{ppe34}), and the fact that $\alpha_1 \le \frac{\alpha_0}{2}$, we arrive at
\be\la{ppe35}\ba
\| \sqrt{\n} \dot{u} \|_{L^2} + \| \na d - \mathbf{k} \otimes d^\bot \|_{H^2}
+ \| d_t + (m_0 \cdot \mathbf{k}) d^\bot \|_{H^1} \le C e^{-\alpha_1 t}.
\ea\ee

Combining (\ref{p18}) and (\ref{ppe35}), we obtain (\ref{ppe01}), which completes the proof of Lemma \ref{ppel}.
\end{proof}

\begin{lemma}\la{ppgl2}
There exists a positive constant $C$ depending only on 
$T$, $q$, $\mu$, $\ga$, $\beta$, $\ol{\n}_0$, $\| \rho_0 \|_{W^{1,q}}$, $\| u_0 \|_{H^1}$, and $\| \na d_0 \|_{H^1}$ such that
\be\la{ppg02}\ba
&\sup_{0\le t\le T} \left( \| \n \|_{W^{1,q}} + t \| u \|^2_{H^2} + \| \na d \|^2_{H^1} + t \| \na^3 d \|^2_{L^2} + t \| \na d_t \|^2_{L^2} \right) \\
& + \int_0^T \left( \|\nabla^2 u\|^{(q+1)/q}_{L^q}
+ t \|\nabla^2 u\|_{L^q}^2+t\| u_t\|_{H^1}^2 + t \| \na^2 d_t \|^2_{L^2} + t \| \na^4 d \|^2_{L^2} \right) dt\le C.
\ea\ee
\end{lemma}
\begin{proof}
First, it follows from (\ref{pp01}), (\ref{pp08}), (\ref{ppg01}), (\ref{ppe19}), (\ref{ppe22})--(\ref{ppe23}), and (\ref{ppe32}) that
\be\la{ppg21}\ba
\sup_{0 \le t \le T} \left( \| u \|^2_{H^1} + \| \na d \|^2_{H^1} + t \| \na^3 d \|^2_{L^2} + t \| \na d_t \|^2_{L^2} \right) \le C.
\ea\ee
In addition, we denote $\Phi = (\Phi^1,\Phi^2)$
with $\Phi^i \triangleq (2\mu+\lam(\n)) \p_i \n$ $(i=1,2)$.
From $(\ref{nlckv})_1$, we deduce that $\Phi^{i}$ satisfies
\be\la{bpg21}\ba
\p_t \Phi^i + (u \cdot \na) \Phi^i
+ (2\mu+\lam(\n)) \na \n \cdot \p_i u + \n \p_i G + \n \p_i P
+ \Phi^i \div u = 0.
\ea\ee

Multiplying (\ref{bpg21}) by $|\Phi|^{q-2} \Phi^i$,
integrating by parts over $\T^2$, and applying H\"older's inequality, we obtain
\be\la{bpg22}\ba
\frac{d}{dt} \| \Phi \|_{L^q}
& \le C ( 1 + \| \na u \|_{L^\infty} ) \| \Phi \|_{L^q}
+ C \| \na G \|_{L^q}.
\ea\ee
Using (\ref{pp53}), (\ref{pp08}), and H\"older's inequality, we have
\be\la{bpg202}\ba
\| \na G \|_{L^q} + \| \na \o \|_{L^q} & \le C \left( \| \n \dot{u}\|_{L^q}
+ \| \na d \cdot \Delta d \|_{L^q} \right) \\
& \le C \left( \| \n \dot{u}\|_{L^q} + A_2 \right).
\ea\ee
Combining this with (\ref{pp08}) and the Gagliardo-Nirenberg inequality yields
\be\la{bpg23}\ba
& \| \div u \|_{L^\infty}+\| \o \|_{L^\infty} \\
& \le C \left( \| G \|_{L^\infty} + \| P-P(\ol{\n}) \|_{L^\infty} \right)
+ \| \o \|_{L^\infty} \\
& \le C \left( 1 + \| G \|_{L^2}^{\frac{q-2}{2(q-1)}}
\| \na G \|_{L^q}^{\frac{q}{2(q-1)}}
+ \| \o \|_{L^2}^{\frac{q-2}{2(q-1)}} \| \na \o \|_{L^q}^{\frac{q}{2(q-1)}} \right) \\
& \le C \left( 1 + A_2 + \| \n \dot{u} \|_{L^q} \right)^{\frac{q}{2(q-1)}}.
\ea\ee
Furthermore, we conclude from (\ref{pp08}), (\ref{bpg202}), and (\ref{bpg23}) that
\be\la{bpg25}\ba
\|\na^2 u\|_{L^q}
& \le C \left( \| \div u \|_{W^{1,q}} + \| \o \|_{W^{1,q}} \right) \\
& \le C \left( \| \na u \|_{L^q}
+ \| \na \div u \|_{L^q}
+ \| \na \o \|_{L^q} \right) \\
& \le C \left(1 + \| \na ( (2\mu+\lam) \div u ) \|_{L^q}
+ \| \div u \|_{L^\infty} \| \na \n \|_{L^q}
+ \| \n \dot{u} \|_{L^q} + A_2 \right) \\
& \le C \left( 1 + \| \n \dot{u} \|_{L^q} + A_2 \right)
\left( e + \| \na \n \|_{L^q} \right).
\ea\ee
The combination of (\ref{bpg23}), (\ref{pp08}), (\ref{bpg25}), and Lemma \ref{bkm} gives
\be\ba\la{bpg24}
\|\na u\|_{L^\infty} 
& \le C \left( \|\div u \|_{L^\infty}
+ \|\o\|_{L^\infty} \right) \log \left(e+ \|\na^2 u\|_{L^q} \right)+ C\|\na u\|_{L^2}+C \\
& \le C \left( 1 + A_2 + \| \n \dot{u} \|_{L^q} \right)
\log \left(e + \| \na \n \|_{L^q} \right).
\ea\ee
Note that the definition of $\Phi$ and (\ref{pp08}) imply
\be\la{bpg26}\ba
2\mu \| \na \n \|_{L^q} \le \| \Phi \|_{L^q} \le C \| \na \n \|_{L^q},
\ea\ee
which along with (\ref{bpg22}) and (\ref{bpg24}) leads to
\be\la{bpg27}\ba
\frac{d}{dt} \log( e + \| \Phi \|_{L^q} )
& \le C \left( 1 + A_2 + \| \n \dot{u} \|_{L^q} \right)
\log \left(e + \| \Phi \|_{L^q} \right).
\ea\ee
Furthermore, it follows from (\ref{pp08}) and Poincar\'e's inequality that
\be\ba\nonumber
\| \rho \dot{u} \|_{L^q} 
& \le C\| \rho \dot{u} \|_{L^2}^{2(q-1)/(q^2-2)}
\| \dot{u} \|_{L^{q^2}}^{q(q-2)/(q^2-2)} \\
& \le C\| \rho \dot{u} \|_{L^2}^{2(q-1)/(q^2-2)}
\| \dot{u} \|_{H^1}^{q(q-2)/(q^2-2)} \\
& \le C + C \| \rho \dot{u} \|_{L^2}
+ C\| \rho \dot{u} \|_{L^2}^{2(q-1)/(q^2-2)}
\| \na \dot{u} \|_{L^2}^{q(q-2)/(q^2-2)},
\ea\ee
which together with (\ref{ppg01}) yields
\be\la{bpg29}\ba
&\int_0^T \left(\| \rho \dot{u} \|^{1+1 /q}_{L^q}
+ t \| \dot{u} \|^2_{H^1} \right) dt \le C.
\ea\ee
Consequently, using (\ref{bpg26}), (\ref{bpg27}), (\ref{bpg29}), and Gr\"onwall's inequality, we arrive at
\be\la{bpg210}\ba
\sup_{0 \le t \le T} \| \n \|_{W^{1,q}} \le C.
\ea\ee
Combining this with (\ref{ppg01}), (\ref{bpg25}), and (\ref{bpg29}) shows that
\be\la{bpg211}\ba
\sup_{0\le t\le T} t \| \na^2 u \|^2_{L^2}
+ \int_0^T \left( \| \nabla^2 u \|^{(q+1)/q}_{L^q}
+ t \|\nabla^2 u \|_{L^q}^2 \right) dt
\le C.
\ea\ee
In addition, from (\ref{pp08}), (\ref{ppg01}), (\ref{ppg13}), (\ref{bpg211}), and H\"older's inequality, we have
\be\la{bpg212}\ba
\int_0^T t \| u_t \|^2_{H^1} dt
& \le C \int_0^T t \left( \| \dot{u} \|^2_{H^1}
+ \| u \cdot \na u \|^2_{H^1} \right) dt \\
& \le C \int_0^T t \left( 1 + \| \na \dot{u} \|^2_{L^2}
+ \| \nabla u \|_{L^4}^4
+ \| u \|_{L^{2q/(q-2)}}^2\|\nabla^2 u \|_{L^q}^2 \right) dt \\
& \le C.
\ea\ee
Finally, applying the $\na^2$ operator to $(\ref{nlckv})_3$ gives
\be\la{np59}\ba
\nabla^2 \Delta d = \nabla^2( d_t + u \cdot \na d - |\nabla d|^2 d).
\ea\ee
Using (\ref{gn11}), (\ref{pp08}), (\ref{ppg21}), and Young's inequality, we derive
\be\la{np510}\ba
\|\nabla^2 \Delta d\|_{L^2}
& \leq C \left( \|\nabla^2 d_t\|_{L^2} + \|\nabla^2 (u \cdot \nabla d)\|_{L^2}
+ \|\nabla^2 (|\nabla d|^2 d) \|_{L^2} \right) \\
& \leq C \|\nabla^2 d_t\|_{L^2} + C \|\nabla d\|_{L^{\frac{2q}{q-2}}} \|\nabla^2 u\|_{L^q}
+ C \|\nabla u\|_{L^4} \|\nabla^2 d\|_{L^4} + C \|u\|_{L^4} \|\nabla^3 d\|_{L^4} \\
& \quad + C \|\nabla^2 d\|_{L^4}^2 + C \|\nabla d\|_{L^4} \|\nabla^3 d\|_{L^4}
+ C \|\nabla d\|_{L^8}^2 \|\nabla^2 d\|_{L^4} \\
& \leq C \|\nabla^2 d_t\|_{L^2} + C \|\nabla^2 u\|_{L^q}
+ C \|\nabla u\|_{L^4} \|\nabla^2 d\|_{L^2}^{\frac{1}{2}} \|\nabla^3 d\|_{L^2}^{\frac{1}{2}} \\
& \quad + C \| \na u \|_{L^4} + C \|\nabla^3 d\|_{L^2}^{\frac{1}{2}} \|\nabla^4 d\|_{L^2}^{\frac{1}{2}} + C \| \na^2 d \|_{H^1} \\
& \leq \frac{1}{2} \| \nabla^2 \Delta d \|_{L^2}
+ C \left( \|\nabla^2 d_t\|_{L^2}
+ \|\nabla^2 u\|_{L^q} + \| \na^2 d \|_{H^1} + \| \na u \|_{H^1} \right),
\ea\ee
which yields
\be\la{np511}\ba
\|\nabla^2 \Delta d\|_{L^2} \leq C \left( \|\nabla^2 d_t\|_{L^2}
+ \|\nabla^2 u\|_{L^q} + \| \na^2 d \|_{H^1} + \| \na u \|_{H^1} \right).
\ea\ee
By (\ref{np511}), (\ref{pp08}), (\ref{ppg01}), and (\ref{bpg211}), we have
\be\la{np512}\ba
\int_0^T t \| \na^4 d \|^2_{L^2} dt
& \le C \int_0^T t \left( \| \na^2 \Delta d \|^2_{L^2} + \| \na^2 d \|^2_{H^1} \right) dt \le C.
\ea\ee
This, combined with (\ref{ppg21}), (\ref{bpg210}), (\ref{bpg211}), and (\ref{bpg212}), implies (\ref{ppg02}) and completes the proof of Lemma \ref{ppgl2}.
\end{proof}

\section{A priori estimates on bounded domains}
In this section, under the hypotheses of Theorem \ref{thpb} and assuming further that $\n_0>0$, we let $(\n,u,d)$ be the strong solution to (\ref{nlckv})--(\ref{i30}), (\ref{yjybjtj}) on $\OM \times (0,T]$, whose existence is guaranteed by Lemma \ref{lct}.

We set
\be\ba\la{b1}
B_1^2(t) \triangleq \int \left( (2\mu + \lam(\n)) (\div u)^2 + |\na u|^2
+ |\na^2 d|^2 + (\n+1)^{\ga-1} (\n-\ol{\n})^2 \right) dx,
\ea\ee
and
\be\ba\la{b2}
B_2^2(t) \triangleq \int \left( \rho |\dot{u}|^2 + |\na \Delta d|^2 + |\na d_t|^2 \right) dx.
\ea\ee

\subsection{A priori estimates (I): upper bound of the density}
This subsection is dedicated to deriving the time-uniform upper bound for the density.

We first establish the standard energy estimate.
\begin{lemma}\la{bpl1}
There exists a positive constant
$C$ depending only on $\mu$, $\gamma$, $\| \n_0 \|_{L^\infty}$, $\| u_0 \|_{H^1}$, $\| \na d_0 \|_{L^2}$, and $\OM$ such that
\be\ba\la{bp01}
& \sup_{0\leq t\leq T} \int \left( \frac{1}{2}\rho |u|^2 + \frac{P}{\ga-1} + |\na d|^2 \right) dx \\
& + \int_0^T \int \left( (2\mu + \lam(\n))  (\div u)^2 + |\na u|^2 + |\na^2 d|^2 \right) dx dt
\le C.
\ea\ee
\end{lemma}
\begin{proof}
First, multiplying $(\ref{nlckv})_2$ and $(\ref{nlckv})_3$ by $u$ and $-(\Delta d + |\na d|^2 d)$, respectively, integrating by parts, and summing the results, we obtain, by using $(\ref{nlckv})_1$, $(\ref{nlckv})_4$, and the boundary condition (\ref{yjybjtj}),
\be\la{bp11}\ba
& \frac{d}{dt} \int \left( \frac{1}{2}\rho |u|^2 + \frac{P}{\ga-1} + \frac{1}{2} |\na d|^2 \right) dx \\
& + \int \left( (2\mu + \lam(\n)) (\div u)^2 + \mu \o^2 + |\Delta d + |\na d|^2 d|^2 \right) dx
+ \mu \int_{\p \OM} A |u|^2 ds = 0.
\ea\ee
Integrating (\ref{bp11}) over $(0,T)$ yields
\be\la{bp12}\ba
& \sup_{0\le t \le T} \int \left( \frac{1}{2}\rho |u|^2 + \frac{P}{\ga-1} + \frac{1}{2} |\na d|^2 \right) dx \\
& + \int_0^T \int \left( (2\mu + \lam(\n)) (\div u)^2 + \mu \o^2 + |\Delta d + |\na d|^2 d|^2 \right) dx dt \le C.
\ea\ee

We next estimate the $L^2(\OM \times (0,T))$-norm of $\na^2 d$.
To this end, we exploit the geometric constraint $|d|=1$ by using the polar representation of $d$.

By Lemma \ref{jzb}, there exists a function $\theta \in H^2(\OM)$ such that
\be\la{bp13}\ba
d=(\cos \theta, \sin \theta ).
\ea\ee
A direct calculation gives
\be\la{bp19}\ba
\na \theta = - d^2 \na d^1  + d^1 \na d^2, \quad |\Delta d + |\na d|^2 d|^2 = |\Delta \theta|^2,
\ea\ee
and
\be\la{bp190}\ba
|\na d| = |\na \theta|, \quad |\na^2 d| \le C \left( |\na \theta|^2 + |\na^2 \theta| \right).
\ea\ee
In view of (\ref{bp12}) and (\ref{bp19}), we arrive at
\be\la{bp110}\ba
\int_0^T \| \Delta \theta \|^2_{L^2} dt \leq C.
\ea\ee
Moreover, (\ref{bp19}) together with the boundary conditions $n \cdot \na d = 0$ on $\p \OM$
gives $n \cdot \na \theta = 0$ on $\p \OM$.
Then, the elliptic estimate (\ref{tygj1}) ensures that
\be\la{bp113}\ba
\| \na^2 \theta \|^2_{L^2} \le C \| \Delta \theta \|^2_{L^2}.
\ea\ee
Hence, from (\ref{gn11}), (\ref{bp12}), (\ref{bp190}), and (\ref{bp113}), we deduce that
\be\la{bp115}\ba
\| \na^2 d \|^2_{L^2} & \le C \left( \| \na \theta \|^4_{L^4} + \| \na^2 \theta \|^2_{L^2} \right) \\
& \le C \left( \| \na^2 \theta \|^2_{L^2} \| \na \theta \|^2_{L^2} + \| \na^2 \theta \|^2_{L^2} \right) \\
& \le C \| \na^2 \theta \|^2_{L^2} \le C \| \Delta \theta \|^2_{L^2},
\ea\ee
which together with (\ref{bp110}) leads to
\be\la{bp116}\ba
\int_0^T \| \na^2 d \|^2_{L^2} dt
\leq C \int_0^T \| \Delta \theta \|^2_{L^2} dt \leq C.
\ea\ee

Combining (\ref{bp116}), (\ref{bp12}), and (\ref{dc1}) yields (\ref{bp01}), thereby completing the proof of Lemma \ref{bpl1}.
\end{proof}

The following $L^\infty(0,T;L^p)$ estimate for $\na d$ will play an important role in the subsequent analysis.

\begin{lemma}\la{bpl2}
For any $2<p<\infty$, there exists a positive constant $C$ depending only on $\mu$, $\ga$, $p$,
$\| \n_0 \|_{L^\infty}$, $\| u_0 \|_{H^1}$, $\| \na d_0 \|_{H^1}$, and $\OM$ such that
\be\ba\la{bp02}
\sup_{0\leq t\leq T} \| \na d \|^p_{L^p} + \int_0^T \int |\na d|^{p-2} |\na^2 d|^2 dx dt \leq C.
\ea\ee
\end{lemma}
\begin{proof}
First, applying $\na$ to $\eqref{nlckv}_3$, we obtain
\be\la{bp21}\ba
\na d_t-\Delta \na d=-\na (u\cdot\na d)+\na(|\na d|^2 d).
\ea\ee
Then, multiplying (\ref{bp21}) by $p |\na d|^{p-2} \na d$, integrating by parts over $\OM$, and using $|d|=1$, we derive
\be\la{bp22}\ba
& \frac{d}{dt} \int |\na d|^p dx
- p \int |\na d|^{p-2} \p_i d^j \p_k\p_k \p_i d^j dx \\
& \le -p \int |\na d|^{p-2} \p_i d^j \p_k \p_i d^j u^k dx
+ C \int \left( |\na d|^p |\na u| + |\na d|^{p+2} \right) dx \\
& \quad + p \int |\na d|^{p-2} \p_i d^j \p_i(|\na d|^2) d^j dx \\
& \le C \| \na d \|^{p+2}_{L^{p+2}} + C \| |\na d|^{\frac{p}{2}} \|^2_{L^4} \| \na u \|_{L^2}.
\ea\ee
Moreover, integrating by parts over $\OM$ yields
\be\la{bp23}\ba
& - p \int |\na d|^{p-2} \p_i d^j \p_k\p_k \p_i d^j dx \\
& = -p \int_{\p \OM} |\na d|^{p-2} \p_i d^j n_k \p_k\p_i d^j ds
+ p \int \p_k( |\na d|^{p-2} \p_i d^j ) \p_k (\p_i d^j) dx \\
& = p \int_{\p \OM} |\na d|^{p-2} \p_i d^j \p_i n_k \p_k d^j ds
+ p \int |\na d|^{p-2} |\na^2 d|^2 dx
+ \frac{4(p-2)}{p} \int |\na (|\na d|^\frac{p}{2})|^2 dx,
\ea\ee
where in the second equality we have used the identity:
\be\la{bp24}\ba
\p_i d^j \p_i n_k \p_k d^j + \p_i d^j n_k \p_k\p_i d^j 
= \na d^j \cdot \na (n \cdot \na d^j) = 0 \quad \text{ on } \p \OM,
\ea\ee
due to $n \cdot \na d = 0$ on $\p \OM$.

Consequently, we conclude from (\ref{bp22}), (\ref{bp23}), (\ref{gn11}), (\ref{bp01}), and Young's inequality that
\be\la{bp025}\ba
& \frac{d}{dt} \int |\na d|^p dx + p \int |\na d|^{p-2} |\na^2 d|^2 dx \\
& \le C \| |\na d|^p \|_{L^1(\p \OM)}
+ C \| \na d \|^{p+2}_{L^{p+2}} + C \| |\na d|^{\frac{p}{2}} \|^2_{L^4} \| \na u \|_{L^2} \\
& \le C \| |\na d|^p \|_{W^{1,1}(\OM)} + C \| \na d \|^p_{L^p} \| \na^2 d \|^2_{L^2}
+ C \| |\na d|^{\frac{p}{2}} \|_{L^2} \| |\na d|^{\frac{p}{2}} \|_{H^1} \| \na u \|_{L^2} \\
& \le C \| \na d \|^p_{L^p} + C \int |\na d|^{p-1} |\na^2 d| dx
+ C \| \na d \|^p_{L^p} \| \na^2 d \|^2_{L^2} \\
& \quad + C \| \na d \|^p_{L^p} \| \na u \|^2_{L^2}
+ \frac{p}{4} \int |\na d|^{p-2} |\na^2 d|^2 dx \\
& \le \frac{p}{2} \int |\na d|^{p-2} |\na^2 d|^2 dx + C \| \na d \|^p_{L^p}
+ C \| \na d \|^p_{L^p} ( \| \na^2 d \|^2_{L^2} + \| \na u \|^2_{L^2} ) \\
& \le \frac{p}{2} \int |\na d|^{p-2} |\na^2 d|^2 dx
+ C \| \na d \|^{p-2}_{L^{p}} \| \na^2 d \|^2_{L^2}
+ C \| \na d \|^p_{L^p} ( \| \na^2 d \|^2_{L^2} + \| \na u \|^2_{L^2} ) \\
& \le \frac{p}{2} \int |\na d|^{p-2} |\na^2 d|^2 dx
+ C \| \na^2 d \|^2_{L^2} + C \| \na d \|^p_{L^p} ( \| \na^2 d \|^2_{L^2} + \| \na u \|^2_{L^2} ),
\ea\ee
which implies
\be\la{bp25}\ba
\frac{d}{dt} \int |\na d|^p dx + \frac{p}{2} \int |\na d|^{p-2} |\na^2 d|^2 dx
\le C \| \na^2 d \|^2_{L^2} + C \| \na d \|^p_{L^p} ( \| \na^2 d \|^2_{L^2} + \| \na u \|^2_{L^2} ).
\ea\ee
This, combined with (\ref{bp01}) and Gr\"onwall's inequality, yields (\ref{bp02}) and completes the proof of Lemma \ref{bpl2}.
\end{proof}

By virtue of the energy estimates (\ref{bp01}) and (\ref{bp02}),
along with arguments analogous to those in \cite[Corollary 3.1 and Proposition 3.3]{FLW}, we can obtain the following time-uniform estimates.
\begin{lemma}\la{bpl3}
Let $g_{+} \triangleq \max\{ g,0 \}$.
For any $2 \le p <\infty$, there exist positive constants $C$ and $\tilde{M}_1$ depending only on
$p$, $\mu$, $\ga$, $\beta$, $\| \n_0 \|_{L^\infty}$, $\| u_0 \|_{H^1}$, $\| \na d_0 \|_{H^1}$, $\OM$, and $A$ such that
\be\la{yzgj}\ba
\sup_{0\le t\le T} \| \n \|_{L^p} + \int_0^T \int_{\OM} (\n-\tilde{M}_1)^p_{+} dxdt \le C.
\ea\ee
\end{lemma}

\begin{lemma}\la{bpl4}
There exists a positive constant $C$ depending only on
$\mu$, $\ga$, $\beta$, $M_0$, $\| \n_0 \|_{L^\infty}$, $\| u_0 \|_{H^1}$, $\| \na d_0 \|_{H^1}$, $\OM$, and $A$ such that
\be\la{bp04}\ba
\int_0^T \int_{\OM} (\n+1)^{\ga-1} (\n-\ol{\n})^2 dx dt \le C.
\ea\ee
\end{lemma}
\begin{proof}
Multiplying $(\ref{nlckv})_2$ by $\mathcal{B}[\n-\ol{\n} ]$, integrating over $\Omega$,
and applying $(\ref{nlckv})_1$, (\ref{bp01}), (\ref{bp02}), (\ref{yzgj}), and Lemma \ref{iod}, we derive
\be\la{bp41}\ba
& \int (\n^\ga - \ol{\n}^\ga) (\n-\ol{\n}) dx \\
&= \left(\int\rho u\cdot \mathcal{B}[\n-\ol{\n}] dx\right)_t- \int\rho u\cdot\mathcal{B}[ (\n-\ol{\n})_t] dx
-\int\rho u\cdot\nabla\mathcal{B}[\n-\ol{\n}]\cdot u dx \\
& \quad + \mu \int \p_i u \cdot \p_i \mathcal{B}[\n-\ol{\n}] dx
+ \int (\mu+\lambda)\div u (\n-\ol{\n}) dx
+ \int \mathcal{B}[\n-\ol{\n}] \cdot \na d \cdot \Delta d dx \\
& \le \left(\int\rho u\cdot\mathcal{B}[\n-\ol{\n}] dx\right)_t
+ C \|\n  u\|^2_{L^2} + C \| \n \|_{L^4} \| u \|^2_{L^4} \| \na \mathcal{B}[\n-\ol{\n}] \|_{L^4} \\
& \quad + C \| \na \mathcal{B}[\n-\ol{\n}] \|_{L^2} \| \na u \|_{L^2}
+ C \int \left( (\n - \tilde{M}_1)^\beta_+ + \tilde{M}_1^\beta \right) |\na u| |\n-\ol{\n}| dx \\
& \quad + C \| \mathcal{B}[\n-\ol{\n}] \|_{L^4} \| \na d \|_{L^4} \| \na^2 d \|_{L^2} \\
& \leq \left(\int\rho u\cdot\mathcal{B}[\n-\ol{\n}] dx\right)_t
+ \ep \int (\n+1)^{\ga-1} (\n-\ol{\n})^2 dx + C(\ep) \| \na u \|^2_{L^2} \\
& \quad + C(\ep) \int (\n -\tilde{M}_1)^{\frac{2\beta(\ga+1)}{\ga-1}}_+ dx + C \|\na^2 d\|^2_{L^2}.
\ea\ee
Moreover, it follows from (\ref{bp01}) and (\ref{yzgj}) that
\be\la{bp42}\ba
\int (\n+1)^{\ga-1} (\n-\ol{\n})^2 dx \le C \int (\n^\ga - \ol{\n}^\ga) (\n-\ol{\n}) dx,
\ea\ee
and
\be\la{bp43}\ba
\left| \int\rho u\cdot\mathcal{B}[\n-\ol{\n}] dx \right|
\le C \| \sqrt{\n} \|_{L^4} \| \sqrt{\n} u \|_{L^2} \| \n-\ol{\n} \|_{L^4}
\le C.
\ea\ee

Substituting (\ref{bp42}) into (\ref{bp41}), choosing $\ep$ sufficiently small, integrating over $(0,T)$, and using (\ref{bp01}), (\ref{yzgj}),
(\ref{bp42}), and (\ref{bp43}), we obtain (\ref{bp04}).
This completes the proof of Lemma \ref{bpl4}.
\end{proof}

\begin{lemma}\la{bpl5}
There exists a positive constant $C$ depending only on
$\mu$, $\ga$, $\beta$, $M_0$, $\| \n_0 \|_{L^\infty}$, $\| u_0 \|_{H^1}$, $\| \na d_0 \|_{H^1}$, $\OM$, and $A$ such that
\be\la{bp05}\ba
\sup_{0\le t\le T}\int \n |u|^{2+\nu} dx \le C,
\ea\ee
where
\be\la{bp005}\ba
\nu \triangleq R_T^{-\frac{\beta}{2}} \nu_0,
\ea\ee
for some suitably small generic constant $\nu_0 \in (0,1)$ depending only on $\mu$, $\ga$, and $\OM$.
\end{lemma}
\begin{proof}
We first recall the following weighted div-curl estimates from \cite[Lemma 3.4]{FLL}:
there exist positive constants $C$ and $\hat{\nu}$ depending only on $\OM$ such that
\be\ba\la{wdc1}
\int_\OM |u|^{\nu} |\na u|^2 dx \le C \int_\OM |u|^{\nu} \left( (\div u)^2 +\o^2  \right)dx,
\ea\ee
for any $\nu \in (0,\hat{\nu})$.

Multiplying $(\ref{nlckv})_2$ by $|u|^\nu u$ and integrating over $ \OM$, we obtain
\be\la{bp51}\ba 
& \frac{1}{(2+\nu)} \frac{d}{dt}\int \n |u|^{2+\nu} dx
+ \int |u|^\nu \left(\mu \o^2
+ (2\mu+\lam) (\div u)^2 \right) dx \\
& \le C \nu  \int  \left( (2\mu+\lam) |\div u|+\mu |\o| \right)  |u|^\nu |\na u| dx
+ C \int |\n^\ga - \ol{\n}^\ga| |u|^\nu |\na u|dx \\
& \quad - \int |u|^{\nu} u \cdot \na d \cdot \Delta d dx \\
& \triangleq I_1+I_2+I_3.
\ea\ee
For $I_1$, we deduce from (\ref{bp005}), (\ref{wdc1}), and Young's inequality that
\be\la{bp52}\ba
I_1 & \le \frac{1}{2} \int |u|^\nu \left(\mu \o^2
+ (2\mu+\lam) (\div u)^2 \right) dx
+ \frac{C\nu_0^2}{2} \int |u|^\nu |\na u|^2 dx \\
& \le \frac{1+\hat{C}\nu_0^2}{2} \int |u|^\nu \left(\mu \o^2
+ (2\mu+\lam) (\div u)^2 \right) dx,
\ea\ee
provided $\nu \in (0,\hat{\nu})$, where $\hat{C}>0$ depends only on $\mu$ and $\OM$.

Note that for $\nu < \frac{\ga-1}{\ga+1}$, we have $\frac{1-\nu}{2}-\frac{1}{\ga+1} \in (0,1)$ and
\be\ba\nonumber
|\n-\ol{\n}|^{\frac{2}{1-\nu}} \le C(\n+1)^{\frac{2\nu}{1-\nu}} (\n-\ol{\n})^2
\le C(\n+1)^{\ga-1} (\n-\ol{\n})^2.
\ea\ee
Hence, for $s$ satisfying
$\frac{1}{s}=\frac{1-\nu}{2}-\frac{1}{\ga+1}$, an application of
Young's and Poincar\'e's inequalities yields
\be\la{bp53}\ba
I_2 & \le C \int (\n^{\ga-1}+1) |\n-\ol{\n}| |u|^\nu |\na u| dx \\
& \le C \int \left( (\n-\tilde{M}_1)_{+}^{\ga-1} + 1 \right) |\n-\ol{\n}|
|u|^\nu |\na u| dx \\
& \le C \left( \int (\n-\tilde{M}_1)_{+}^{s(\ga-1)} dx
+ \int \left( |\n-\ol{\n}|^{\ga+1} + |\n-\ol{\n}|^{\frac{2}{1-\nu}} \right) dx
+ \int |\na u|^2 dx \right) \\
& \le C \int (\n-\tilde{M}_1)_{+}^{s(\ga-1)} dx + C B^2_1,
\ea\ee
provided $\nu < \frac{\ga-1}{\ga+1}$.

Integrating by parts over $\OM$ and applying the boundary conditions
$(n \cdot \na d)|_{\p \OM}=0$ and $(u \cdot n)|_{\p \OM}=0$, we derive
\be\la{bp504}\ba
I_3 & = \int \p_k (|u|^{\nu} u^i)  \p_i d^j \p_k d^j dx
+ \int |u|^{\nu} u^i  \p_i \p_k d^j \p_k d^j dx \\
& \le C \int |u|^\nu |\na u| |\na d|^2 dx
\le C \int (1 + |u|) |\na u| |\na d|^2 dx \\
& \le C \| \na u \|^2_{L^2} + C \| \na d \|^4_{L^4}
+ C \| u \|_{L^4} \| \na u \|_{L^2} \| \na d \|^2_{L^8} \\
& \le C \| \na u \|^2_{L^2} + C \| \na^2 d \|^2_{L^2}.
\ea\ee
Substituting (\ref{bp52})--(\ref{bp504}) into (\ref{bp51}) and choosing
$\nu_0 < \min \left\{\hat{\nu},\frac{1}{\sqrt{2\hat{C}}},\frac{\ga-1}{\ga+1} \right\}$, we obtain
\be\la{bp54}\ba 
\frac{d}{dt}\int \n |u|^{2+\nu}dx
\le C \int (\n-\tilde{M}_1)_{+}^{s(\ga-1)} dx + C B^2_1.
\ea\ee

Integrating (\ref{bp54}) with respect to $t$ and using (\ref{bp01}),
(\ref{yzgj}) and (\ref{bp04}), we obtain (\ref{bp05}) and complete the proof of Lemma \ref{bpl5}.
\end{proof}

\begin{lemma}\la{bpl6}
For any $2 < p<\infty$ and $\ep \in (0,1)$, there exists a positive constant $C$ depending only on $\mu$, $\ga$, $\ep$, $p$, $\beta$, $\|\n_0\|_{L^\infty}$, $\| u_0 \|_{H^1}$, $\| \na d_0 \|_{H^1}$, $\OM$, and $A$ such that
\be\la{bp06}\ba
\| \nabla u \|_{L^{p}}
& \le C R^{\frac{1}{2}-\frac{1}{p}+\ep}_T (1+B_1)^{\frac{2}{p}}
(1+B_1+B_2)^{1-\frac{2}{p}},
\ea\ee
where $B_1$, $B_2$, and $R_T$ are defined in \eqref{b1}, \eqref{b2}, and \eqref{mdsj}, respectively.

Moreover, when $p<\frac{2(\ga+1)}{\ga}$ and $\ga<2\beta$, it holds that
\be\la{bp006}\ba
\| \nabla u \|_{L^{p}}
& \le C R^{\frac{1}{2}-\frac{1}{p}+\ep}_T B_1^{\frac{2}{p}} (1+B_1+B_2)^{1-\frac{2}{p}}.
\ea\ee
\end{lemma}
\begin{proof}
First, we use (\ref{gw}) to rewrite $(\ref{nlckv})_2$ as
\be\la{bp61}\ba
\n\dot{u}= \na G + \mu \na^\bot \o - \na d \cdot \Delta d.
\ea\ee
By virtue of the boundary conditions (\ref{yjybjtj}), we deduce that $G$ and $\o$ satisfy the following elliptic equations respectively:
\be\la{bp62}\ba
\begin{cases}
\Delta G=\div \left( \rho \dot{u} + \na d \cdot \Delta d \right) & \mathrm{in}\, \,  \OM, \\
\frac {\p G}{\p n}= \left( \rho \dot{u} + \na d \cdot \Delta d \right) \cdot n
- \mu n^\bot \cdot \na (A u \cdot n^\bot) &\mathrm{on}\, \,  \p \OM,
\end{cases}
\ea\ee
and
\be\la{bp63}\ba
\begin{cases}
\mu \Delta \o =\na^\bot \cdot \left( \rho \dot{u} + \na d \cdot \Delta d \right) & \mathrm{in}\, \,  \OM, \\
\o = -A u \cdot n^\bot &\mathrm{on}\, \,  \p \OM.
\end{cases}
\ea\ee
The standard $L^p$ estimate of elliptic equations
(see \cite{GT}, \cite[Lemma 4.27]{NS}) shows that for any integer $k \ge 0$ and $1<p<\infty$,
\be\la{bp64}\ba
\| \na G \|_{W^{k,p}} + \| \na \o \|_{W^{k,p}} \le C \left( \| \n \dot{u}\|_{W^{k,p}}
+ \| \na d \cdot \Delta d \|_{W^{k,p}} + \| \na u \|_{W^{k,p}} \right).
\ea\ee
This, combined with Poincar\'e's inequality, gives
\be\la{bp66}\ba
\| G \|_{H^1} + \| \o \|_{H^1}
& \le C \left( \| \rho \dot{u} \|_{L^2} + \| \na u \|_{L^2} + \| \na d \cdot \Delta d \|_{L^2} \right) + C |\ol{G}| \\
& \le C R^{1/2}_T B_2 + C B_1 + C \| \na d \|_{L^4} \| \Delta d \|_{L^4} \\
& \le C R^{1/2}_T B_2 + C B_1,
\ea\ee
where in the second inequality we have used the following estimates:
\be\nonumber\ba
\left| \int_\OM G dx \right|
= \left| \int_\OM \left( \lam(\n) \div u - (\n^\ga - \ol{\n}^\ga) \right) dx \right|
\le C B_1,
\ea\ee
due to (\ref{yzgj}).

In addition, it follows from (\ref{yzgj}) and H\"older's inequality that
\be\la{bp67}\ba
\| G \|^2_{L^2} \le C (1+R^\beta_T B^2_1), \quad
\left\| \frac{G}{2\mu+\lam} \right\|^2_{L^2} \le C(1+B^2_1).
\ea\ee
Combining (\ref{gn11}), (\ref{dc1}) and (\ref{bp67}), we obtain
\be\la{bp68}\ba
\| \na u \|_{L^p}
& \le C \left(\| \div u \|_{L^p} + \| \o \|_{L^p} \right) \\
& \le C \left\| \frac{G}{2\mu+\lam} \right\|_{L^p}
+ C \left\| \frac{P-P(\ol{\n})}{2\mu+\lam} \right\|_{L^p}
+ C \| \o \|^{\frac{2}{p}}_{L^2} \| \o \|^{1-\frac{2}{p}}_{H^1} \\
& \le C \left\| \frac{G}{2\mu+\lam} \right\|^{\frac{2}{p}-\ep}_{L^2}
\| G \|_{L^{\frac{2(1+\ep)p-4}{p\ep}}}^{-\frac{2}{p}+1+\ep}
+ C \left( B^{\frac{2}{p}}_1 \| \o \|^{1-\frac{2}{p}}_{H^1}
+ B_1 + 1 \right) \\
& \le C (1+B_1)^{\frac{2}{p}-\ep} \| G \|^{\ep}_{L^2} \| G \|^{1-\frac{2}{p}}_{H^1}
+ C \left( B^{\frac{2}{p}}_1 \| \o \|^{1-\frac{2}{p}}_{H^1}
+ B_1 + 1 \right) \\
& \le C R^{\frac{\beta \ep}{2}}_T (1+B_1)^{\frac{2}{p}}
\left( \| G \|_{H^1} + \| \o \|_{H^1} \right)^{1-\frac{2}{p}}
+ C (1+B_1),
\ea\ee
which together with (\ref{bp66}) yields (\ref{bp06}).

We now prove (\ref{bp006}).
Note that when $p<\frac{2(\ga+1)}{\ga}$ and $\ga<2\beta$, it holds that $p(\ga-\beta)-2<(\ga-1)$.
We thus obtain
\be\la{bp69}\ba
\left\| \frac{P-P(\ol{\n})}{2\mu+\lam} \right\|^p_{L^p}
\le C \int (\n+1)^{p(\ga-\beta)-2} (\n-\ol{\n})^2 dx \le C B^2_1.
\ea\ee
Using (\ref{gw}) and choosing $p=2$ in (\ref{bp69}) leads to
\be\la{bp610}\ba
\left\| \frac{G}{2\mu+\lam} \right\|^2_{L^2}
\le C B^2_1 + \left\| \frac{P-P(\ol{\n})}{2\mu+\lam} \right\|^2_{L^2}
\le C B^2_1.
\ea\ee
Arguing as in the derivation of (\ref{bp68}), we conclude from (\ref{bp69}) and (\ref{bp610}) that
\be\ba\nonumber
\| \na u \|_{L^p}
& \le C \left\| \frac{G}{2\mu+\lam} \right\|_{L^p}
+ C \left\| \frac{P-P(\ol{\n})}{2\mu+\lam} \right\|_{L^p}
+ C \| \o \|^{\frac{2}{p}}_{L^2} \| \o \|^{1-\frac{2}{p}}_{H^1} \\
& \le C B_1^{\frac{2}{p}-\ep} \| G \|^{\ep}_{L^2} \| G \|^{1-\frac{2}{p}}_{H^1}
+ C \left( B^{\frac{2}{p}}_1 \| \o \|^{1-\frac{2}{p}}_{H^1}
+ B^{\frac{2}{p}}_1 \right) \\
& \le C R^{\beta \ep}_T B_1^{\frac{2}{p}}
\left( \| G \|_{H^1} + \| \o \|_{H^1} \right)^{1-\frac{2}{p}}
+ C B^{\frac{2}{p}}_1,
\ea\ee
which together with (\ref{bp66}) gives (\ref{bp006}) and completes the proof of Lemma \ref{bpl6}.
\end{proof}

\begin{lemma}\la{bpl7}
For any $\ep \in (0,1)$,
there exists a positive constant $C$ depending only on
$\ep$, $\mu$, $\ga$, $\beta$, $M_0$, $\| \n_0 \|_{L^\infty}$, $\| u_0 \|_{H^1}$, $\| \na d_0 \|_{H^1}$, $\OM$, and $A$ such that
\be\la{bp07}\ba
\sup_{0 \le t \le T} \log(e+B^2_1) + \int_0^T \frac{B^2_2}{e+B^2_1} dt
\le C R^{1+\ep}_T.
\ea\ee
\end{lemma}
\begin{proof}
Multiplying (\ref{bp61}) by $2 \dot{u}$ and integrating by parts over $\OM$, we derive
\be\la{bp71}\ba
& \frac{d}{dt} \int \left(\mu \o^2 + \frac{G^2}{2\mu + \lam}\right)dx
+ 2 \int \n |\dot{u}|^2 dx \\
& = - \mu \int \o^2 \div u dx + 4 \int G\nabla u^1 \cdot\nabla^{\perp}u^2 dx
-2 \int G (\div u)^2 dx \\
& \quad - \int \frac{ (\beta-1)\lam - 2\mu }{(2\mu + \lam)^2} G^2\div u dx
-2\beta \int \frac{ \lam (P-P(\ol{\n})) }{ (2\mu +\lam)^2 } G \div u dx
+ 2\ga \int \frac{P}{2\mu +\lam} G \div u dx \\
& \quad + 2 \int_{\p\OM} G u \cdot \na u \cdot n ds
+ 2 \mu \int_{\p \OM} \o (\dot{u} \cdot n^\bot) ds
-2 \int \dot{u} \cdot \na d \cdot \Delta d dx
= \sum_{i=1}^9 I_i,
\ea\ee
where we have used (\ref{bp701}) and (\ref{bp7001}).

Using the energy estimates (\ref{bp01}), (\ref{yzgj}), and (\ref{bp06}), and following the arguments in \cite[Proposition 3.6]{FLW}, we arrive at
\be\la{bp72}\ba
\sum_{i=1}^8 I_i \le - \mu \frac{d}{dt} \int_{\p \OM} A |u|^2 ds
+ \frac{1}{8} B^2_2
+ C R^{1+\ep}_T (1+B^2_1) B^2_1.
\ea\ee
Next, we estimate $I_9$.
First, by virtue of the boundary conditions $(n \cdot \na d)|_{\p \OM}=0$ and $(u \cdot n)|_{\p \OM}=0$,
we integrate by parts over $\OM$ and use H\"older's inequality to derive
\be\la{bp73}\ba
& -2 \int u_t \cdot \na d \cdot \Delta d dx \\
& = -2 \int u^i_t \p_i d^j \p_k \p_k d^j dx \\
& = 2 \int \p_k u^i_t \p_i d^j \p_k d^j dx
+ 2 \int u^i_t \p_i \p_k d^j \p_k d^j dx \\
& = 2 \int (\na d \odot \na d) \cdot \na u_t dx
+ \int u_t \cdot \na ( |\na d|^2 ) dx \\
& = 2 \frac{d}{dt} \int (\na d \odot \na d) \cdot \na u dx
- 2 \int (\na d_t \odot \na d)\cdot\na u dx \\
& \quad - 2 \int(\na d \odot \na d_t)\cdot\na u dx
- \int \div u_t |\na d|^2 dx \\
& \le \frac{d}{dt} \int \left( 2 (\na d \odot \na d) \cdot \na u - |\na d|^2 \div u \right) dx
+ C \| \nabla d_t \|_{L^2} \| \na d \|_{L^4} \| \na u \|_{L^4}.
\ea\ee
Consequently, we deduce from (\ref{bp73}), (\ref{gn11}), (\ref{bp02}), (\ref{bp06}), and Young's inequality that
\be\la{bp74}\ba
I_9 & = - 2 \int u_t \cdot \na d \cdot \Delta d dx
- 2 \int u \cdot \na u \cdot \na d \cdot \Delta d dx \\
& \le \frac{d}{dt} \int \left( 2 (\na d \odot \na d) \cdot \na u - |\na d|^2 \div u \right) dx
+ C \| \nabla d_t \|_{L^2} \| \na d \|_{L^4} \| \na u \|_{L^4} \\
& \quad + C \| u \|_{L^8} \| \na u \|_{L^4} \| \na d \|_{L^8} \| \na^2 d \|_{L^2} \\
& \le \frac{d}{dt} \int \left( 2 (\na d \odot \na d) \cdot \na u - |\na d|^2 \div u \right) dx
+ \frac{1}{4} B^2_2 + C R^{1+\ep}_T (1+B^2_1) B^2_1.
\ea\ee
Moreover, applying $\na$ to $\eqref{nlckv}_3$ yields
\be\la{bp75}\ba
\na d_t-\Delta \na d=-\na (u\cdot\na d)+\na(|\na d|^2 d).
\ea\ee
Multiplying (\ref{bp75}) by $\nabla d_t-\Delta \nabla d$, integrating the resulting equality over $\Omega$, and using (\ref{gn11}), (\ref{bp02}), (\ref{bp06}), and Young's inequality, we arrive at
\be\la{bp76}\ba
& \frac{d}{dt} \| \Delta d \|^2_{L^2} + \| \na d_t \|^2_{L^2} + \| \na \Delta d \|^2_{L^2} \\
& \le C \int \left( |\na (u \cdot \na d)|^2 + |\na (|\na d|^2 d)|^2 \right) dx \\
& \le C \int \left( |\na u|^2 |\na d|^2 + |u|^2 |\na^2 d|^2
+ |\na^2 d|^2 |\na d|^2 + |\na d|^6 \right) dx \\
& \le C \| \na u \|^2_{L^4} \| \na d \|^2_{L^4}
+ C \| u \|^2_{L^8} \| \na^2 d \|^2_{L^\frac{8}{3}}
+ C \| \na^2 d \|^2_{L^4} \| \na d \|^2_{L^4} + C \| \na d \|^6_{L^6} \\
& \le C \| \na u \|^2_{L^4} \| \na d \|^2_{L^4}
+ C \| \na u \|^2_{L^2} \| \na d\|_{L^4} \left( \| \na^2 d\|_{L^2} + \| \na^3 d\|_{L^2} \right) \\
& \quad + C \| \na^2 d \|_{L^2} \left( \| \na^2 d \|_{L^2} + \| \na^3 d \|_{L^2} \right)
+ C \| \na d \|^4_{L^4} \| \na^2 d \|^2_{L^2} \\
& \le \frac{1}{8} B^2_2 + C R^{1+\ep}_T (1+B^2_1) B^2_1.
\ea\ee

Substituting (\ref{bp72}) and (\ref{bp74}) into (\ref{bp71}) and adding the result to (\ref{bp76}), we get
\be\la{bp77}\ba
\frac{d}{dt} B_3 + \frac{1}{2} B^2_2
\le C R^{1+\ep}_T (1+B^2_1) B^2_1,
\ea\ee
where
\be\ba\nonumber
B_3 \triangleq \int \left( \mu \o^2 + \frac{G^2}{2\mu + \lam} + |\Delta d|^2
- 2 (\na d \odot \na d) \cdot \na u + |\na d|^2 \div u \right) dx
+ \mu \int_{\p \OM} A |u|^2 ds.
\ea\ee
Furthermore, from (\ref{bp01}) and (\ref{yzgj}), we can find a constant $\hat{C}_1>e$ such that
\be\la{bp78}\ba
\frac{1}{C} (\hat{C}_1+B_3) \le \hat{C}_1 + B^2_1 \le C(\hat{C}_1+B_3).
\ea\ee
Dividing (\ref{bp77}) by $\hat{C}_1+B_3$ and applying (\ref{bp78}), we obtain
\be\la{bp79}\ba
\frac{d}{dt} \log (\hat{C}_1+B_3) + \frac{1}{2} \frac{B^2_2}{\hat{C}_1+B_3}
\le C R^{1+\ep}_T B^2_1.
\ea\ee
Integrating (\ref{bp79}) over $(0,T)$ and using
(\ref{bp01}) and (\ref{bp04}), we derive (\ref{bp07}) and finish the proof of Lemma \ref{bpl7}.
\end{proof}

We next use the preceding a priori estimates to establish the upper bound for $\n$.
More precisely, combining the continuity equation $(\ref{nlckv})_1$ with the definition of the effective viscous flux in (\ref{gw}), we obtain
\be\la{evf0}\ba
\frac{D}{Dt} \Lambda(\n) + (P-P(\ol{\n})) = -G,
\ea\ee
where $\Lambda(\n)=2\mu \log \n + \frac{1}{\beta} \n^\beta$.
Hence, to derive the upper bound for $\n$, it is crucial to obtain the $L^\infty$ estimate for $G$.
For this purpose, we follow the approach in \cite{FLL} to derive a pointwise representation of $G$.

Recall from (\ref{bp62}) that for any $t\in [0,T]$, $G$ satisfies the Neumann problem
\be\la{evf}\ba
\begin{cases}
\Delta G=\div \left( \rho \dot{u} + \na d \cdot \Delta d \right) & \mathrm{in}\, \,  \OM, \\
\frac {\p G}{\p n}= \left( \rho \dot{u} + \na d \cdot \Delta d \right) \cdot n
- \mu n^\bot \cdot \na (A u \cdot n^\bot) &\mathrm{on}\, \,  \p \OM.
\end{cases}
\ea\ee
The Green's function $N(x,y)$ for the Neumann problem on the unit disc $\mathbb{D}$ (see \cite{STT}) is given by
\be\la{glhs}\ba
N(x,y)=-\frac{1}{2\pi}\bigg(\log|x-y|+\log\left||x|y-\frac{x}{|x|}\right|\bigg).
\ea\ee
Furthermore, by the Riemann mapping theorem (see \cite{SES}), there exists a conformal mapping
$\varphi=(\varphi_1, \varphi_2):\overline{\Omega}\rightarrow\overline{\mathbb{D}}$.
We define the pull back Green's function $\widetilde{N}(x,y)$ on $\Omega$ by
\be\la{pglhs}\ba
\widetilde{N}(x,\, y)=N\big(\varphi(x),\varphi(y)\big) \ \ \mathrm{for}\ x,y\in\Omega.
\ea\ee
Using this pull back Green's function $\widetilde{N}$, we can obtain the following pointwise representation of $G$; the proof is similar to that of \cite[Lemma 3.7]{FLL}.

\begin{lemma}\label{bpl8}
Assume that $G\in C\big([0,T];C^1(\ol{\Omega})\cap C^2(\Omega)\big)$ satisfies the boundary value problem $(\ref{evf})$.
Then, for any $x\in\Omega$ and $t \in [0,T]$,
\be\la{bp08}\ba
G(x,t)
& = \int_\Omega \nabla_y\widetilde{N}(x,\, y) \cdot \left( \rho \dot{u} + \na d \cdot \Delta d \right) (y)\, dy
- \int_{\partial\Omega} \frac{\partial \widetilde{N}}{\partial n}(x,\, y) G(y) dS_y \\
& \quad + \mu \int_{\partial\Omega} \widetilde{N}(x,\, y) \left( n^\bot \cdot \nabla \o \right) dS_y \\
& = \frac{D}{Dt}\psi(x,t) - K_1(x,t) - K_2(x,t),
\ea\ee
where
\be\la{bp008}\ba
\psi(x,t)&\triangleq\int_\Omega\nabla_y\widetilde{N}\big(x,y\big)\cdot \rho u(y,t) dy,\\
K_1(x,t) & \triangleq
- \int_\Omega \nabla_y\widetilde{N}(x,y) \cdot \na d \cdot \Delta d \, dy +
\int_{\partial\Omega} \left[ \frac{\partial \widetilde{N}}{\partial n} G(y)
- \mu \widetilde{N}(x,y) \left( n^\bot \cdot \nabla \o \right) \right] dS_y, \\
K_2(x,t) & \triangleq \sum_{i,j=1}^2 \int_\Omega \left[ \p_{x_i} \p_{y_j}\widetilde{N}(x,y)\cdot u^i(x)+\p_{y_i} \p_{y_j}\widetilde{N}(x,y)\cdot u^i(y) \right]\rho u^j(y)dy.
\ea\ee
\end{lemma}

\begin{lemma}\la{bpl9}
For any $\ep>0$ and $0 \le t_1 < t_2 \le T$,
there exists a positive constant $C$ depending only on
$\ep$, $\mu$, $\ga$, $\beta$, $M_0$, $\| \n_0 \|_{L^\infty}$, $\| u_0 \|_{H^1}$, $\| \na d_0 \|_{H^1}$, $\OM$, and $A$
such that when $\ga<2\beta$, it holds that
\be\ba\la{bp09}
\int_{t_1}^{t_2} - G(x(t),t) dt
\le C R^{1+\ep}_T (t_2-t_1) + C R^{1+\frac{\beta}{4}+3\ep}_T
+ C R_T^{\frac{2+\beta}{3}},
\ea\ee
when $\ga \ge 2\beta$, we have
\be\ba\la{bp009}
\int_{t_1}^{t_2} - G(x(t),t) dt
\le C R^{1+\frac{\beta}{4}+2\ep}_T (t_2-t_1+1) + C R_T^{\frac{2+\beta}{3}},
\ea\ee
where $x(t)$ is the flow line determined by $x(t)'=u(x(t),t)$.
\end{lemma}
\begin{proof}
First, we deduce from (\ref{bp08}) that
\be\la{bp91}\ba
- G(x(t),t) & = - \frac{d}{dt} \psi(t) + K_1(x(t),t) + K_2(x(t),t) \\
& \le - \frac{d}{dt} \psi(t) + \| K_1(\cdot,t) \|_{L^\infty}
+ \| K_2(\cdot,t) \|_{L^\infty},
\ea\ee
where we denote $\psi(x(t),t)$ by $\psi(t)$.

For $\nu$ as in Lemma \ref{bpl5}, it follows from (\ref{bp05}), (\ref{bp005}) and H\"older's inequality that
\be\ba\nonumber
| \psi(t) | \le C
\left| \int_{\OM} \nabla_y \widetilde{N}\big(x,y\big)\cdot \rho u(y,t) dy \right|
& \le C \int_\OM |x-y|^{-1} \n |u| dy \\
& \le C \left( \int_{\OM} |x-y|^{-\frac{2+\nu}{1+\nu}} dy \right)^{\frac{1+\nu}{2+\nu}}
\left(\int_{\OM}\n^{2+\nu} |u|^{2+\nu} dy \right)^{\frac{1}{2+\nu}} \\
& \le C \nu^{-\frac{1+\nu}{2+\nu}} R_T^{\frac{1+\nu}{2+\nu}} \\
& \le C R_T^{\frac{2+\beta}{3}},
\ea\ee
which yields
\be\la{bp92}\ba
- \int_{t_1}^{t_2} \frac{d}{dt} \psi(t) dt = \psi(t_1) - \psi(t_2) \le C R_T^{\frac{2+\beta}{3}}.
\ea\ee
For the first term in $K_1$, (\ref{bp02}) and the standard elliptic estimates give
\be\la{bp93}\ba
\left| \int_\OM  \nabla_y\widetilde{N}(x,y) \cdot \na d \cdot \Delta d dy \right|
& \le C \int_\OM  |x-y|^{-1} | \na d | |\na^2 d| dy \\
& \le C \| |x-y|^{-1} \|_{L^{\frac{3}{2}}} \| \na d \|_{L^6} \| \na^2 d \|_{L^6} \\
& \le C (B_1 + B_2).
\ea\ee
Recall from \cite[Lemma 3.6]{FLL} that for any $x\in \OM$, $y\in \p \OM$,
\be\la{bp94}\ba
\frac{\partial \widetilde{N}}{\partial n}(x,\, y)=-\frac{1}{2 \pi} |\na \varphi_1 (y)|,
\ea\ee
which together with (\ref{bp66}) and Sobolev embedding shows that
\be\la{bp95}\ba
\int_{\partial\Omega}\left|\frac{\partial \widetilde{N}}{\partial n}(x,\, y) G(y) \right|dS_y
\leq C\int_{\partial\Omega}\big|G\big|\, dS\leq C\big\|G\big\|_{H^1}
\leq C R^{1/2}_T B_2 + C B_1.
\ea\ee
Moreover, we deduce from (\ref{bp06}) and H\"older's inequality that
\be\la{bp96}\ba
\left| \int_{\partial\Omega} - \mu \widetilde{N}(x,\, y) \left( n^\bot \cdot \nabla \o \right) dS_y \right|
& = \mu \left| \int _{\partial \Omega }{\widetilde{N}}(x, y) n^\bot \cdot \nabla (Au\cdot n^\bot )\mathrm{d}S_y\right| \\
& = \mu \left| \int _{\Omega }\mathrm {div}(\nabla ^\bot (Au\cdot n^\bot ){\widetilde{N}}(x, y))d y\right| \\
&\leq C\int _{\Omega }\vert \nabla {\widetilde{N}}(x, y)\vert (\vert \nabla u\vert +\vert u\vert ) dy \\
&\leq C\Vert \nabla {\widetilde{N}}\Vert _{L^\frac{4}{3}} \Vert \nabla u\Vert _{L^4} \\
&\le C R^{1/2}_T \left(1 + B_1 + B_2 \right).
\ea\ee
Hence, combining (\ref{bp93}), (\ref{bp95}), and (\ref{bp96}), and applying Young's inequality, we obtain
\be\la{bp97}\ba
\int_{t_1}^{t_2} \| K_1(\cdot,t) \|_{L^\infty} dt
& \le C  R^{1/2}_T \int_{t_1}^{t_2} \left(1 + B_1 + B_2 \right) dt \\
& \le C R_T \int_{t_1}^{t_2} \left(1 + B^2_1 \right) dt
+ C \int_{t_1}^{t_2} \frac{B^2_2}{e+B^2_1} dt \\
& \le C R_T (t_2 - t_1) + C R_T^{1+\ep}.
\ea\ee
For $K_2$, we recall from \cite[Proposition 3.2]{FLL} that
\be\la{bp98}\ba
\|K_2(\cdot, t)\|_{L^\infty(\Omega)}
& \leq C \left( \sup_{ x \in \overline{\Omega} } \int_\Omega \frac{\rho|u|^2(y)}{|x-y|} dy
+\sup_{x\in\overline{\Omega}}\int_\Omega \frac{|u(x)-u(y)|}{|x-y|^2} \rho |u|(y) dy \right).
\ea\ee
We now estimate each term on the right-hand side of (\ref{bp98}) separately.

Using (\ref{yzgj}) and Poincar\'e's inequality, we derive
\be\la{bp99}\ba
\int_\Omega\frac{\rho|u|^2(y)}{|x-y|}dy
\leq C \left(\int_\Omega |x-y|^{-3/2}dy\right)^{2/3}
\| \n \|_{L^6} \| u \|_{L^{12}}^2
\leq C \|\nabla u\|_{L^2}^2,
\ea\ee
which ensures that
\be\la{bp910}\ba
\int_{t_1}^{t_2} \sup_{x\in\overline{\Omega}}\int_\Omega\frac{\rho|u|^2(y)}{|x-y|} dy dt
\le C \int_0^T \| \na u \|^2_{L^2} dt \le C.
\ea\ee

On the other hand, following the arguments in \cite[Lemma 4.2]{FLW}, we conclude that for any $2<p<6$
\be\la{bp911}\ba
\int_{\OM} \frac{\left|u(x)-u(y) \right|}{|x-y|^2}
\n |u| (y) dy
\le C \| \na u \|_{L^p} R^{1+\frac{\beta}{4}}_T B_1^{\frac{2}{p}}.
\ea\ee

Then, using Lemma \ref{bpl6}, we bound (\ref{bp911}) by considering the following two separate cases:

$Case \ 1: \ga<2\beta.$
For $\ep \in (0,\frac{1}{2})$, choose $2<p<6$ sufficiently close to $2$ such that
\be\ba\nonumber
\frac{p}{2}< \min \left\{ \frac{\ga+1}{\ga},\ 
\frac{1+\beta/4+3\ep}{1+\beta/4+2\ep},\ \frac{1}{1-2\ep} \right\}.
\ea\ee
This, combined with (\ref{bp006}), gives
\be\ba\la{bp9012}
\| \nabla u \|_{L^{p}}
& \le C R^{\frac{1}{2}-\frac{1}{p}+\ep}_T B_1^{\frac{2}{p}} (1+B_1+B_2)^{1-\frac{2}{p}}
\le C R^{2\ep}_T B_1^{\frac{2}{p}} (1+B_1+B_2)^{1-\frac{2}{p}}.
\ea\ee
Inserting (\ref{bp9012}) into (\ref{bp911}) and using Young's inequality, we arrive at
\be\la{bp912}\ba
& \int_{\OM} \frac{\left|u(x)-u(y) \right|}{|x-y|^2}
\n |u| (y) dy \\
& \le C R^{1+\frac{\beta}{4}+2\ep}_T B_1^{\frac{4}{p}} (1+B_1+B_2)^{1-\frac{2}{p}} \\
& \le C R^{ (1+\frac{\beta}{4}+2\ep) \frac{p}{2} }_T B^2_1
+ C (1+B_1+B_2) \\
& \le C\left( 1+R^{1+\frac{\beta}{4}+3\ep}_T B^2_1+\frac{B^2_2}{e+B^2_1} \right).
\ea\ee
Integrating (\ref{bp912}) over $(t_1,t_2)$ and applying (\ref{bp01}), (\ref{bp04}) and (\ref{bp07}) yields
\be\la{bp913}\ba
\int_{t_1}^{t_2} \sup_{x \in \overline{\OM}}
\left( \int_{\OM} \frac{\left|u(x)-u(y) \right|}{|x-y|^2}
\n |u| (y) dy \right) dt
\le C(t_2-t_1) + C R^{1+\frac{\beta}{4}+3\ep}_T.
\ea\ee

$Case \ 2: \ga \ge 2\beta.$
From (\ref{bp06}) and (\ref{bp911}), we deduce
\be\la{bp914}\ba
& \int_{\OM} \frac{\left|u(x)-u(y) \right|}{|x-y|^2}
\n |u| (y) dy \\
& \le C R^{ \frac{3}{2}-\frac{1}{p}+\frac{\beta}{4}+\ep }_T
( B_1^{\frac{2}{p}} + B_1^{\frac{4}{p}} )
(1+B_1+B_2)^{1-\frac{2}{p}} \\
& \le C R^{ (\frac{3}{2}-\frac{1}{p}+\frac{\beta}{4}+\ep) \frac{p}{2} }_T (1+B^2_1)
+ C (1+B_1+B_2) \\
& \le C\left( R^{1+\frac{\beta}{4}+2\ep}_T (1+B^2_1) + \frac{B^2_2}{e+B^2_1} \right),
\ea\ee
provided $2<p<6$ is sufficiently close to $2$ and satisfies
$\frac{p}{2} \le \frac{3/2+\beta/4+2\ep}{3/2+\beta/4+\ep}$.

Integrating (\ref{bp914}) over $(t_1,t_2)$, we obtain after using (\ref{bp01}), (\ref{bp04}) and (\ref{bp07}) that
\be\la{bp915}\ba
\int_{t_1}^{t_2} \sup_{x \in \overline{\OM}}
\left( \int_{\OM} \frac{\left|u(x)-u(y) \right|}{|x-y|^2}
\n |u| (y) dy \right) dt
\le C R^{1+\frac{\beta}{4}+2\ep}_T (t_2-t_1+1).
\ea\ee
The combination of (\ref{bp91}), (\ref{bp92}), (\ref{bp97}), (\ref{bp98}), (\ref{bp910}), (\ref{bp913}), and (\ref{bp915}) gives (\ref{bp09}) and (\ref{bp009}), thus completing the proof of Lemma \ref{bpl9}.
\end{proof}

\begin{lemma}\la{bpl10}
There exists a positive constant $C$ depending only on
$\mu$, $\ga$, $\beta$, $M_0$, $\| \n_0 \|_{L^\infty}$, $\| u_0 \|_{H^1}$, $\| \na d_0 \|_{H^1}$, $\OM$, and $A$ such that
\be\ba\la{bp010}
& \sup_{0\leq t\leq T} \left( \|\n\|_{L^\infty} + \| u \|_{H^1} + \| \na d \|_{H^1} \right) \\
& + \int_0^T \left( \| u \|^2_{H^1} + \| \sqrt{\n} \dot{u} \|^2_{L^2} + \| \na^2 d \|^2_{H^1} + \| \na d_t \|^2_{L^2} \right) dt
\le C.
\ea\ee
\end{lemma}
\begin{proof}
First, along a flow line $x(t)$ as in Lemma \ref{bpl9}, it follows from (\ref{evf0}) that
\be\la{bp0101}\ba
\frac{d}{dt} \Lambda(\n) + P = -G + P(\ol{\n}).
\ea\ee
The function $y=\Lambda(\n)$ is strictly increasing on $(0,\infty)$,
guaranteeing the existence of its inverse function $\n=\Lambda^{-1}(y)$ for all $y\in(-\infty,\infty)$.
We may rewrite (\ref{bp0101}) as
\be\ba\nonumber
y'(t) = g(y) + h'(t),
\ea\ee
with
\be\la{bp0102}\ba
y=\Lambda(\n), \quad g(y)=-P( \Lambda^{-1}(y) ), \quad h=\int_0^t \left( P(\ol{\n}) - G \right) ds.
\ea\ee
Clearly, $g(\infty)=-\infty$.
Next, we apply Lemma \ref{bpl9} to estimate $h$ in two separate cases.

$Case \ 1: \ga<2\beta.$
From (\ref{bp01}) and (\ref{bp09}), we conclude that
\be\ba\nonumber
h(t_2)-h(t_1) \le C \left( R^{1+\frac{\beta}{4}+3\ep}_T + R_T^{\frac{2+\beta}{3}} \right)
+ C R^{1+\ep}_T (t_2-t_1).
\ea\ee
Choose $N_0$, $N_1$, and $\overline{\zeta}$ in Lemma \ref{zli} as follows:
\be\la{bp0103}\ba
N_0 = C \left( R^{1+\frac{\beta}{4}+3\ep}_T + R_T^{\frac{2+\beta}{3}} \right),\quad
N_1 = C R^{1+\ep}_T,\quad
\overline{\zeta} = \Lambda \left( \left( C R^{1+\ep}_T \right)^{1/\ga} \right),
\ea\ee
which together with (\ref{bp0102}) yields
\be\ba\nonumber
g(\zeta)=-( \Lambda^{-1}(\zeta) )^\ga \le - N_1 = - C R^{1+\ep}_T
\quad \text{ for all } \zeta \ge \overline{\zeta}.
\ea\ee
Furthermore, $R_T \ge 1$ implies $\overline{\zeta} \le C R^{ (1+\ep)\frac{\beta}{\ga} }_T$.
Combining this with (\ref{bp0103}) and Lemma \ref{zli} leads to
\be\la{bp0104}\ba
R^\beta_T \le C R_T^{ \max\{ 1+\frac{\beta}{4}+3\ep,\frac{2+\beta}{3},
(1+\ep)\frac{\beta}{\ga} \} }.
\ea\ee
Since $\beta>4/3$ and $\ga>1$,
we take $0<\ep<\min\{ (3\beta-4)/12,\ga-1 \}$ and deduce from (\ref{bp0104}) that
\be\la{bp0105}\ba
\sup_{0 \le t \le T} \| \n \|_{L^\infty} \le C.
\ea\ee

$Case \ 2: \ga \ge 2\beta.$
In view of (\ref{bp01}) and (\ref{bp009}), we have
\be\ba\la{bp01050}
h(t_2)-h(t_1) \le C \left( R^{1+\frac{\beta}{4}+2\ep}_T + R_T^{\frac{2+\beta}{3}} \right)
+ C R^{1+\frac{\beta}{4}+2\ep}_T (t_2-t_1).
\ea\ee
Arguing as in $Case ~1$ and applying Lemma \ref{zli}, we obtain
\be\la{bp0106}\ba
R^\beta_T \le C R_T^{ \max\{ 1+\frac{\beta}{4}+2\ep,\frac{2+\beta}{3},
(1+\frac{\beta}{4}+2\ep)\frac{\beta}{\ga} \} }.
\ea\ee
Moreover, $\ga \ge 2\beta$ implies that
$(1+\frac{\beta}{4}+2\ep)\frac{\beta}{\ga} \le 1+\frac{\beta}{4}+2\ep$,
which together with (\ref{bp0106}) leads to
\be\la{bp0107}\ba
R^\beta_T \le C R_T^{ \max\{ 1+\frac{\beta}{4}+2\ep,\frac{2+\beta}{3} \} }.
\ea\ee
For $\beta>4/3$,
we choose $0<\ep<(3\beta-4)/8$ and conclude from (\ref{bp0107}) that
\be\la{bp0108}\ba
\sup_{0 \le t \le T} \| \n \|_{L^\infty} \le C.
\ea\ee
Combining (\ref{bp0105}), (\ref{bp0108}), (\ref{bp01}), (\ref{bp07}), and 
Poincar\'e's inequality, we arrive at (\ref{bp010}) and complete the proof of Lemma \ref{bpl10}.
\end{proof}

\subsection{A priori estimates (II): higher order estimates}

The goal of this subsection is to establish some necessary higher-order estimates, which ensure that the local strong solution can be extended globally in time.

\begin{lemma}\la{bpgl1}
There exists a positive constant $C$ depending only on
$\mu$, $\ga$, $\beta$, $M_0$, $\| \n_0 \|_{L^\infty}$, $\| u_0 \|_{H^1}$, $\| \na d_0 \|_{H^1}$, $\OM$, and $A$ such that
\be\ba\la{bpg01}
\sup_{0\le t\le T}
\si \left( \| \sqrt{\n} \dot{u} \|^2_{L^2} + \| \na d_t \|^2_{L^2} + \| \na^3 d \|^2_{L^2} \right)
+ \int_0^{T} \si \left( \| \na \dot{u} \|^2_{L^2} + \| \na^2 d_t \|^2_{L^2} \right) dt \le C,
\ea\ee
with $\si \triangleq \min\{ 1,t \}$.
Moreover, for any $p\in [1,\infty)$, there is a positive constant $C$ depending only on $p$, $\mu$, $\ga$, $\beta$, $\| \n_0 \|_{L^\infty}$, $\| u_0 \|_{H^1}$, $\| \na d_0 \|_{H^1}$, $\OM$, and $A$ such that
\be\ba\la{bpg001}
\sup_{0 \le t \le T} \si \| \na u \|^2_{L^p} \le C.
\ea\ee
\end{lemma}
\begin{proof}
First, we adapt the idea from \cite{CL,H1,FLL} to prove (\ref{bpg01}).
Applying the operator $\dot{u}^j[\frac{\pa}{\pa t}+\div(u\cdot)]$ to
$(\ref{bp61})^j,$ summing with respect to $j,$
and integrating over $\OM$, we obtain after integration by parts and using the boundary condition (\ref{yjybjtj}) that
\be\la{bpg11}\ba
\frac{d}{dt}\left(\frac{1}{2}\int\rho|\dot{u}|^2dx \right)
&=\int \bigg( {\dot{u}}\cdot \nabla G_t + {\dot{u}}^j\mathrm {div}(u \partial _jG) \bigg) dx\\
&\quad + \mu \int \bigg( {\dot{u}} \cdot \na^\bot \o_t
+ {\dot{u}}^j\partial _k( u^k ( \na^\bot \o )^j ) \bigg) dx \\
& \quad - \int \bigg( {\dot{u}} \cdot (\na d \cdot \Delta d)_t
+ {\dot{u}}^j\partial _k( u^k \p_j d \cdot \Delta d ) \bigg) dx \\
&=I_1+I_2+I_3.
\ea\ee
Following the arguments in \cite[Proposition 4.3]{FLW}, we can get for any $\ep>0$,
\be\la{bpg12}\ba
I_1 + I_2 & \le -\frac{d}{dt} \int_{\partial \Omega } G ( u \cdot \nabla n \cdot u ) ds
- \mu \| \div \dot{u} \|^2_{L^2} - \mu \| \curl \dot{u} \|^2_{L^2} \\
& \quad + \ep \| \na \dot{u} \|^2_{L^2} + C(\ep) ( B^2_1 + \| \sqrt{\n} \dot{u} \|^4_{L^2} ).
\ea\ee

Moreover, for any $1\le p<\infty$, the boundary condition $u \cdot n|_{\p \OM}=0$ implies the
following div-curl type estimates (see \cite{CL,FLL}):
\be\la{bpg110}\ba
\| \dot{u} \|_{L^p} & \le C \left( \| \na \dot{u} \|_{L^2} + \| \na u \|^2_{L^2} \right), \\
\| \na \dot{u} \|_{L^2} & \le C \left( \| \div \dot{u} \|_{L^2} + \| \curl \dot{u} \|_{L^2} + \| \na u \|^2_{L^4} \right).
\ea\ee

We now estimate $I_3$.
Integrating by parts over $\OM$ and applying (\ref{gn11}), (\ref{tygj1}), (\ref{bp010}), and Young's inequality, we derive
\be\la{bpg14}\ba
I_3 & = \int \left( \p_k \dot{u}^j (\p_j d^{l} \p_k d^{l})_{t}
+ \dot{u}^j (\p_j \p_k d^{l} \p_k d^{l})_{t}
+ \p_k \dot{u}^j u^k \p_j d \cdot \Delta d \right) dx \\
& \le \frac{1}{2} \int \dot{u} \cdot \na ( |\na d|^2 )_t dx
+ C \| \na \dot{u} \|_{L^2} \left( \| \na d \|_{L^4} \| \na d_t \|_{L^4}
+ \| u \|_{L^8} \| \na d \|_{L^8} \| \na^2 d \|_{L^4} \right) \\
& \le \frac{1}{2} \int \dot{u} \cdot \na ( |\na d|^2 )_t dx
+ \ep \| \na \dot{u} \|^2_{L^2} + \frac{1}{8} \| \Delta d_t \|^2_{L^2}
+ C(\ep) ( B^2_1 + B^2_2 ) \\
& \le 2 \ep \| \na \dot{u} \|^2_{L^2} + \frac{1}{4} \| \Delta d_t \|^2_{L^2}
+ C(\ep) ( B^2_1 + B^2_2 ),
\ea\ee
where in the last inequality we have used
\be\la{bpg15}\ba
\int \dot{u} \cdot \na ( |\na d|^2 )_t dx
& = \int_{\p \OM} (u \cdot \na u \cdot n) ( |\na d|^2 )_t ds
- \int \div \dot{u} ( |\na d|^2 )_t dx \\
& \le - \int_{\p \OM} (u \cdot \na n \cdot u) ( |\na d|^2 )_t ds
+ C \| \na \dot{u} \|_{L^2} \| \na d \|_{L^4} \| \na d_t \|_{L^4} \\
& \le C \| u \|^2_{H^1} \| \na d \|_{H^1} \| \na d_t \|_{H^1}
+ C \| \na \dot{u} \|_{L^2} \| \na d \|_{L^4} \| \na d_t \|_{L^4} \\
& \le \ep \| \na \dot{u} \|^2_{L^2} + \frac{1}{8} \| \Delta d_t \|^2_{L^2}
+ C(\ep) ( B^2_1 + B^2_2 ),
\ea\ee
due to (\ref{bjds}), (\ref{gn11}), (\ref{tygj1}), (\ref{bp010}), and Young's inequality.

Substituting (\ref{bpg12}) and (\ref{bpg14}) into (\ref{bpg11}) yields
\be\la{bpg16}\ba
& \frac{1}{2} \frac{d}{dt} \| \sqrt{\n} \dot{u} \|^2_{L^2}
+ \mu \| \div \dot{u} \|^2_{L^2} + \mu \| \curl \dot{u} \|^2_{L^2} \\
& \le -\frac{d}{dt} \int_{\partial \Omega } G ( u \cdot \nabla n \cdot u ) ds
+ \frac{1}{4} \| \Delta d_t \|^2_{L^2} + 3 \ep \| \na \dot{u} \|^2_{L^2}
+ C(\ep) ( B^2_1 + \| \sqrt{\n} \dot{u} \|^4_{L^2} + B^2_2 ).
\ea\ee

On the other hand, differentiating (\ref{bp75}) with respect to $t$ gives
\be\la{bpg17}\ba
\nabla d_{tt}-\Delta\nabla d_{t} =-\nabla (u\cdot \nabla d)_{t}+\nabla (|\nabla d|^{2}d)_{t}.
\ea\ee
Multiplying (\ref{bpg17}) by $\na d_t$, integrating over $\OM$, and applying
(\ref{gn11}), (\ref{tygj1}), (\ref{bp010}), (\ref{bpg110}), and Young's inequality, we obtain
\be\la{bpg18}\ba
& \frac{1}{2}\frac{d}{dt}\|\nabla d_{t}\|_{L^{2}}^{2} + \|\Delta d_{t}\|_{L^{2}}^{2} \\
& = \int \left( u_t \cdot \na d + u \cdot \na d_t - (|\nabla d|^{2}d)_{t} \right) \cdot \Delta d_t dx \\
& = \int \left( \dot{u} \cdot \na d - u \cdot \na u \cdot \na d
+ u \cdot \na d_t - (|\nabla d|^{2}d)_{t} \right) \cdot \Delta d_t dx \\
& = - \int \left( \p_i \dot{u} \cdot \na d + \dot{u} \cdot \na \p_i d \right) \cdot \p_i d_t dx
- \int \left( u \cdot \na u \cdot \na d
- u \cdot \na d_t + (|\nabla d|^{2}d)_{t} \right) \cdot \Delta d_t dx \\
& \le C \left( \| \na \dot{u} \|_{L^2} \| \na d \|_{L^4}
+ \| \dot{u} \|_{L^4} \| \na^2 d \|_{L^2} \right) \| \na d_t \|_{L^4}
+ C \| u \|_{L^8} \| \na u \|_{L^4} \| \na d \|_{L^8} \| \Delta d_t \|_{L^2} \\
& \quad + C \left( \| u \|_{L^4} \| \na d_t \|_{L^4} + \| \na d \|^2_{L^8} \| d_t \|_{L^4}
+ \| \na d \|_{L^4} \| \na d_t \|_{L^4}  \right) \| \Delta d_t \|_{L^2} \\
& \le \frac{1}{4} \| \Delta d_t \|^2_{L^2} + \ep \| \na \dot{u} \|^2_{L^2}
+ C(\ep) ( B^2_1 + B^2_2 ),
\ea\ee
where in the last inequality we have used the following estimate:
\be\la{bpg19}\ba
\| d_t \|^2_{L^4} & \le C \left( \| u \cdot \na d \|^2_{L^4}
+ \| \na^2 d \|^2_{L^4} + \| |\na d|^2 \|^2_{L^4} \right) \\
& \le C \left( \| u \|^2_{L^8} \| \na d \|^2_{L^8} + \| \na^2 d \|^2_{H^1}
+ \| \na d \|^4_{L^8} \right)
\le C ( B^2_1 + B^2_2 ),
\ea\ee
due to (\ref{gn11}) and (\ref{bp010}).

Combining (\ref{bpg16}), (\ref{bpg18}), and (\ref{bpg110}), and choosing $\ep$ sufficiently small, we arrive at
\be\la{bpg111}\ba
& \frac{1}{2} \frac{d}{dt} \left( \| \sqrt{\n} \dot{u} \|^2_{L^2} + \| \na d_t \|^2_{L^2} \right)
+ \frac{\mu}{2} \| \div \dot{u} \|^2_{L^2} + \frac{\mu}{2} \| \curl \dot{u} \|^2_{L^2}
+ \frac{1}{2} \| \Delta d_t \|^2_{L^2} \\
& \le -\frac{d}{dt} \int_{\partial \Omega } G ( u \cdot \nabla n \cdot u ) ds + C ( B^2_1 + \| \sqrt{\n} \dot{u} \|^4_{L^2} + B^2_2 ),
\ea\ee
where we have used the following estimate:
\be\la{bpg112}\ba
\| \na u \|^4_{L^4}
& \le C \left( \| \div u \|^4_{L^4} + \| \o \|^4_{L^4} \right) \\
& \le C \left( \| G \|^4_{L^4} + \| P-P(\ol{\n}) \|^4_{L^4} + \| \o \|^4_{L^4} \right) \\
& \le C \left( \| G \|^2_{L^2} \| G \|^2_{H^1} + \| P-P(\ol{\n}) \|^2_{L^2}
+ \| \o \|^2_{L^2} \| \o \|^2_{H^1} \right) \\
& \le C \left( \| G \|^2_{H^1} + \| \o \|^2_{H^1} + B^2_1 \right) \\
& \le C \left( B^2_1 + B^2_2 \right),
\ea\ee
owing to (\ref{dc1}), (\ref{gn11}) and (\ref{bp010}).

In addition, in view of (\ref{bp75}), (\ref{bp010}), (\ref{bpg112}), and Young's inequality, we have
\be\ba\nonumber
\| \nabla \Delta d \|_{L^{2}}^{2}
\leq & C \left( \|\nabla d_{t}\|_{L^{2}}^{2} + \||\nabla u||\nabla d|\|_{L^{2}}^{2}
+ \||u||\nabla^{2}d|\|_{L^{2}}^{2} + \||\nabla d|^{3}\|_{L^{2}}^{2}
+ \||\nabla^{2}d||\nabla d|\|_{L^{2}}^{2} \right) \\
\leq & C \left( \|\nabla d_{t}\|_{L^{2}}^{2} + \|\nabla u\|_{L^{4}}^{2} \|\nabla d\|_{L^{4}}^{2}
+ \| u \|_{L^{4}}^{2} \| \nabla^{2} d \|^2_{L^{4}}
+ \|\nabla d\|_{L^{6}}^{6} + \| \na d \|^2_{L^4} \| \na^2 d \|^2_{L^4} \right) \\
\leq & C \left( \|\nabla d_{t}\|_{L^{2}}^{2}
+ B_1 B_2 + B^2_1 + \|\nabla^{2} d\|_{L^{4}}^{2} \right) \\
\leq & C \|\nabla d_{t}\|_{L^{2}}^{2}
+ C B_1 \left( \| \sqrt{\n} \dot{u} \|_{L^2} + \| \na d_t \|_{L^2} + \| \na \Delta d \|_{L^2} \right) + C B^2_1 \\
\leq & \frac{1}{2} \| \nabla \Delta d \|_{L^{2}}^{2}
+ C \left( B^2_1 + \| \sqrt{\n} \dot{u} \|^2_{L^2} + \|\nabla d_{t}\|_{L^{2}}^{2} \right),
\ea\ee
which implies
\be\la{bpg113}\ba
\|\nabla \Delta d\|_{L^{2}}^{2}
\le C \left( B^2_1 + \| \sqrt{\n} \dot{u} \|^2_{L^2} + \|\nabla d_{t}\|_{L^{2}}^{2} \right).
\ea\ee
Furthermore, it follows from (\ref{bp66}), (\ref{bp010}), (\ref{bpg113}), and Young's inequality that
\be\la{bpg114}\ba
\left| \int _{\partial \Omega } G (u\cdot \nabla n\cdot u) ds \right|
& \le C \| G \|_{H^1} \| \na u\|^2_{L^2} \le C B_1 B_2 + C B^2_1 \\
& \le C B_1 \left( \| \sqrt{\n} \dot{u} \|_{L^2} + \| \na d_t \|_{L^2} + \| \na \Delta d \|_{L^2} \right) + C B^2_1 \\
& \le C B_1 \left( \| \sqrt{\n} \dot{u} \|_{L^2} + \| \na d_t \|_{L^2} \right) + C B^2_1 \\
& \le \frac{1}{4} \left( \| \sqrt{\n} \dot{u} \|^2_{L^2} + \| \na d_t \|^2_{L^2} \right) + C B^2_1.
\ea\ee

Multiplying (\ref{bpg111}) by $\si$ and using Gr\"onwall's inequality, we obtain
(\ref{bpg01}) after using (\ref{bp010}), (\ref{bpg113}), and (\ref{bpg114}).

Finally, (\ref{bp06}) and (\ref{bpg01}) imply (\ref{bpg001}), which finishes the proof of Lemma \ref{bpgl1}.
\end{proof}

Next, using the uniform estimates (\ref{bp010}) and (\ref{bpg001}), we can derive the following exponential decay estimate.
The proof follows from a slight modification of the argument used in Lemma \ref{ppel}.
\begin{lemma}\la{bpel}
For any $p \in [1,\infty)$, there exist positive constants
$C$ and $\alpha_1$ depending only on
$p$, $\mu$, $\ga$, $\beta$, $M_0$, $\| \n_0 \|_{L^\infty}$, $\| u_0 \|_{H^1}$, $\| \na d_0 \|_{H^1}$, $\OM$, and $A$ such that for any $1 \le t <\infty$,
\be\la{bpe01}\ba
\| \n-\ol{\n_0}\|_{L^p} + \| \na u \|_{L^p}
+ \| \sqrt{\n} \dot{u} \|_{L^2} + \| \na d \|_{H^2} + \| \na d_t \|_{L^2} \le C e^{-\alpha_1 t}.
\ea\ee
\end{lemma}

By combining (\ref{bp010}), (\ref{bpg01}), and (\ref{bpg110}) with the argument used in the proof of Lemma \ref{ppgl2}, we can obtain the following estimate.

\begin{lemma}\la{bpgl2}
There exists a positive constant $C$ depending only on 
$T$, $q$, $\mu$, $\ga$, $\beta$, $M_0$, $\| \rho_0 \|_{W^{1,q}}$, $\| u_0 \|_{H^1}$, $\| \na d_0 \|_{H^1}$, $\OM$, and $A$ such that
\be\la{bpg02}\ba
&\sup_{0\le t\le T} \left( \| \n \|_{W^{1,q}} + t \| u \|^2_{H^2} \right) \\
& + \int_0^T \left( \|\nabla^2 u\|^{(q+1)/q}_{L^q}
+ t \|\nabla^2 u\|_{L^q}^2+t\| u_t\|_{H^1}^2 + t \| \na^4 d \|^2_{L^2} \right) dt\le C.
\ea\ee
\end{lemma}

\section{Proofs of Theorems \ref{thpp}--\ref{thbp2}}
Having established the required a priori estimates in Sections 3 and 4, we are now in a position to prove the main results.
The proofs of Theorems \ref{thpp}--\ref{thbp2} follow standard arguments, and thus we only briefly outline the main steps here and refer to \cite{HL2,HL3,VK,WX} for further details.

\noindent\textbf{Proof of Theorem \ref{thpp}}.
Let $(\n_0,u_0,d_0)$ be the initial data in Theorem \ref{thpp}, satisfying (\ref{pssol1}).
For any $\de \in (0,1)$, define $\n^\de_0=\n_0 + \de$.
According to Lemma \ref{lct}, the problem (\ref{nlckv})--(\ref{zqbjtj}) with the initial data $(\n_0^\de,u_0,d_0)$ has a unique local strong solution $(\n^\de,u^\de,d^\de)$ on $\OM \times (0,T^\de]$ for some $T^\de>0$.
Moreover, the a priori estimates established in Lemmas \ref{ppl8}, \ref{ppgl1}, and \ref{ppgl2} ensure that this solution $(\n^\de,u^\de,d^\de)$ can be extended to $\OM \times (0,T]$ for any $T>0$ and satisfies all the estimates in Lemmas \ref{ppl8}, \ref{ppgl1}, and \ref{ppgl2} uniformly in $\de$.
Letting $\de \to 0$ and using standard compactness arguments as in \cite{HL2,L2,P,VK}, we conclude that the problem (\ref{nlckv})--(\ref{zqbjtj}) admits a global strong solution $(\n,u,d)$ satisfying (\ref{pssol2}).

Moreover, (\ref{pp01}), (\ref{pp08}), and (\ref{ppe01}) imply the uniform density upper bound (\ref{pup}) and the exponential decay estimate (\ref{ped}).

Finally, the uniqueness of the strong solution $(\n,u,d)$ satisfying (\ref{pssol2}) follows from the arguments in \cite{LZLL}.
This completes the proof of Theorem \ref{thpp}.

\noindent\textbf{Proof of Theorem \ref{thpb}}.
Using the a priori estimates established in Section 4, we can prove Theorem \ref{thpb} by an argument similar to that used in the periodic case.

\noindent\textbf{Proof of Theorem \ref{thbp2}}.
In view of the decay estimates (\ref{ped}) and (\ref{ed}), the proof of Theorem \ref{thbp2} is similar to that of \cite[Theorem 1.2]{CL}.
We therefore omit the details.

\bigskip

\noindent\textbf{Data availability.} No data was used for the research described in the article.

\bigskip

\noindent\textbf{Conflict of interest.} The authors declare that they have no conflict of interest.

\end{document}